\documentclass[12pt,twoside]{amsart}
\usepackage[margin=3cm]{geometry}
\usepackage[colorlinks=false]{hyperref}
\usepackage[english]{babel}
\usepackage{graphicx,titling}
\usepackage{float}
\usepackage{amsmath,amsfonts,amssymb,amsthm}
\usepackage{lipsum}
\usepackage[T1]{fontenc}
\usepackage{fourier}
\usepackage{color}
\usepackage[latin1]{inputenc}
\usepackage{esint}
\usepackage{caption}
\usepackage{ccicons}

\makeatletter
\def\blfootnote{\xdef\@thefnmark{}\@footnotetext}
\makeatother

\newcommand\ccnote{
    \blfootnote{\copyright\,\, Guillaume Dubach, Pierre Germain, and Benjamin Harrop-Griffiths}
    \blfootnote{\ccLogo\, \ccAttribution\,\, Licensed under a \href{https://creativecommons.org/licenses/by/4.0/}{Creative Commons Attribution License (CC-BY)}.}
}

\usepackage[export]{adjustbox}
\numberwithin{equation}{section}
\usepackage{setspace}

\renewcommand{\leq}{\leqslant}

\renewcommand{\geq}{\geqslant}
\renewcommand{\mathbb}{\varmathbb}
\usepackage{fancyhdr}
\newtheorem{theorem}{Theorem}[section]
\newtheorem{lemma}[theorem]{Lemma}
\newtheorem{corollary}[theorem]{Corollary}
\newtheorem{proposition}[theorem]{Proposition}
\newtheorem{definition}[theorem]{Definition}
\newtheorem{remark}[theorem]{Remark}
\usepackage{enumitem,wasysym,cite}

\newcommand{\mathscr}{\mathcal}

\newtheorem{ex}[theorem]{Example}
\newtheorem{fact}[theorem]{Fact}

\newcommand{\eq}[2]{\begin{equation}#2\label{#1}\end{equation}}
\newcommand{\pde}[2]{\ensuremath{\begin{cases}#1\smallskip\\#2\end{cases}}}

\newcommand{\R}{\ensuremath{\mathbb{R}}}
\newcommand{\C}{\ensuremath{\mathbb{C}}}
\newcommand{\Z}{\ensuremath{\mathbb{Z}}}

\newcommand{\Schwartz}{\ensuremath{\mathscr{S}}}
\newcommand{\Test}{\ensuremath{C^\infty_c}}
\newcommand{\Cont}{\ensuremath{C}}

\renewcommand{\Re}{\ensuremath{\operatorname{Re}}}

\renewcommand{\epsilon}{\varepsilon}
\newcommand{\mc}{\mathcal}
\newcommand{\mb}{\mathbf}
\newcommand{\mr}{\mathrm}
\newcommand{\mf}{\mathfrak}
\newcommand{\mbb}{\mathbb}
\newcommand{\mtt}[1]{{\normalfont \texttt{#1}}}

\newcommand{\<}{\ensuremath{\langle}}
\renewcommand{\>}{\ensuremath{\rangle}}
\newcommand{\p}{\ensuremath{\partial}}

\newcommand{\bigO}{O}

\newcommand{\bbo}{\ensuremath{\mathbb 1}}
\newcommand{\bbE}{\ensuremath{\mathbb E}}
\newcommand{\bbP}{\ensuremath{\mathbb P}}

\newcommand{\bpf}{\begin{proof}}
\newcommand{\epf}{\end{proof}}
\newcommand{\qtq}[1]{\quad\text{#1}\quad}
\newcommand{\LHS}[1]{\mr{LHS}\eqref{#1}}
\newcommand{\RHS}[1]{\mr{RHS}\eqref{#1}}

\DeclareMathOperator{\tr}{tr}
\DeclareMathOperator{\diag}{diag}
\newcommand{\kin}{\mr{kin}}
\newcommand{\lin}{\mr{lin}}
\newcommand{\nonlin}{\mr{nonlin}}
\usepackage{stmaryrd}
\newcommand{\cC}{\mb C}
\newcommand{\cM}{\mc M}
\newcommand{\app}{\mr{app}}
\newcommand{\err}{\mr{err}}
\newcommand{\bbU}{\mbb U}

\newcommand{\sbrack}[1]{^{(#1)}}
\newcommand{\sprack}[1]{^{[#1]}}
\newcommand{\cK}{\mc K}

\newcommand{\cE}{\mc E}

\newcommand{\wt}{\widetilde}

\newcommand{\Interactions}{\mtt{Interactions}}
\newcommand{\Oscillations}{\mtt{Oscillations}}

\newcommand{\ttL}{\mtt{L}}
\newcommand{\ttK}{\mtt{K}}
\newcommand{\Identifications}{\mtt{Identifications}}
\newcommand{\Data}{\mtt{Data}}
\newcommand{\cN}{\mc N}
\newcommand{\cL}{\mc L}
\newcommand{\cH}{\mc H}
\newcommand{\bA}{\mb A}

\newcommand{\cB}{\mc B}
\newcommand{\fL}{\mf L}
\newcommand{\cF}{\mc F}

\newcommand{\vk}{\xi}
\newcommand{\ttn}{n}
\newcommand{\tti}{\mtt{i}}

\newcommand{\mfi}{\mf i}
\newcommand{\mfj}{\mf j}

\DeclareMathOperator{\Wg}{Wg}

\newcommand{\ccpair}{(\sigma, \tau)}
\newcommand{\cc}{\sigma \tau^{\! -1}}
\newcommand{\ccn}{\omega}
\newcommand{\CC}{{\mathscr{C} \hspace{-.26em} \mathscr{P}}}
\newcommand{\at}{\mathfrak{a}}
\newcommand{\At}{\mathfrak{A}}
\newcommand{\cyc}{ \mathbin{\rotatebox[origin=c]{-90}{ $\cancer$}} \hspace{.02in} }
\newcommand{\Moeb}{\text{M\oe b}}
\newcommand{\fix}{ {\scriptstyle \circ} \hspace{.015in} }
\newcommand{\fixhat}{ {\scriptstyle \hat{\circ}}\hspace{.015in} }

\usepackage{tikz}
\usetikzlibrary{shapes.misc}
\usetikzlibrary{shapes.symbols}
\usetikzlibrary{shapes.geometric}
\usetikzlibrary{decorations.markings}
\tikzset{
	dot/.style={circle,fill=black,draw=black,inner sep=.5mm,thick},
	ghostdot/.style={circle,fill=white,draw=black,inner sep=.5mm,thick},
	indot/.style={rectangle,fill=white,draw=black,inner sep=.75mm,thick},
    gluedot/.style={rectangle,fill=white,draw=black,densely dashed,inner sep=1mm,thick},
	outdot/.style={regular polygon, regular polygon sides=3,fill=white,draw=black,inner sep=0.4mm,thick},
	->-/.style={decoration={markings,mark=at position 0.5 with {\arrow{>}}},postaction={decorate}}
}
\definecolor{col1}{HTML}{000000}
\definecolor{col2}{HTML}{EE33FF}
\definecolor{col3}{HTML}{009988}
\definecolor{col4}{HTML}{FF7733}
\definecolor{col5}{HTML}{3399EE}
\definecolor{col6}{HTML}{EE3377}
\definecolor{col7}{HTML}{AAAAAA}

\address{Guillaume Dubach, Centre de Math\' ematiques Laurent Schwartz, \'Ecole Polytechnique, 91120 Palaiseau, France}
\email{guillaume.dubach@polytechnique.edu}
\medskip
\address{Pierre Germain, Department of Mathematics, Imperial College London, London, SW7 2AZ, UK} 
\email{pgermain@ic.ac.uk}
\medskip
\address{Benjamin Harrop-Griffiths, Department of Mathematics \& Statistics, Georgetown University, Washington, DC 20057, USA}
\email{benjamin.harropgriffiths@georgetown.edu}

\begin{document}

\thispagestyle{empty}

\begin{minipage}{0.28\textwidth}
\begin{figure}[H]
%\centering
\includegraphics[width=2.5cm,height=2.5cm,left]{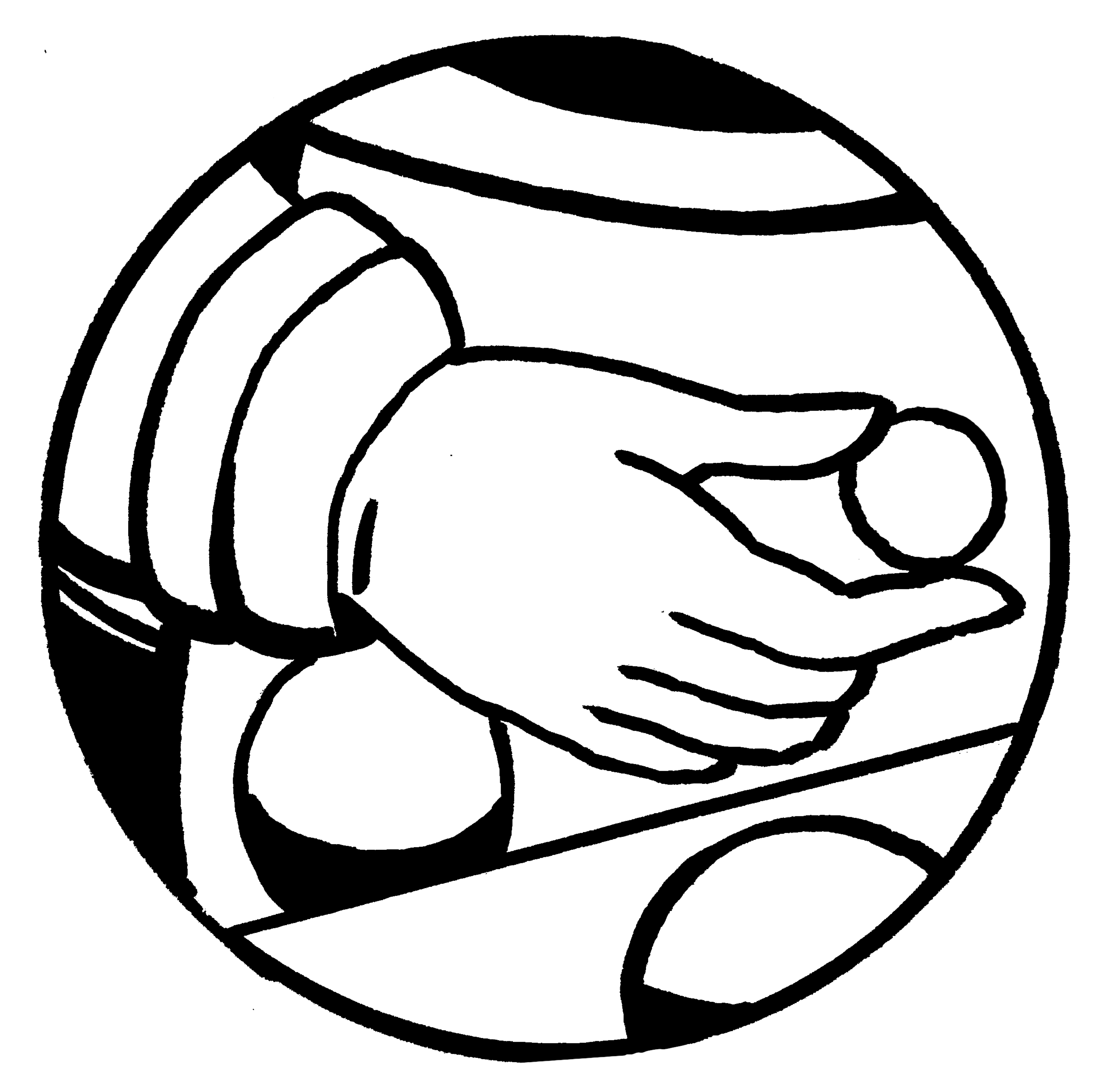}
\end{figure}
\end{minipage}
\begin{minipage}{0.7\textwidth} 
\begin{flushright}
%% The following metadata, in particular
%% the Paper No. and the DOI will be inserted by the journal
Ars Inveniendi Analytica (2023), Paper No. 7, 63 pp.
\\
DOI 10.15781/63wy-ad98
\\
ISSN: 2769-8505
\end{flushright}
\end{minipage}

\ccnote

\vspace{1cm}

%%      -------------------------------------------------------------------------------
%%      -------------------------- TITLE ----------------------------
%%      -------------------------------------------------------------------------------
%% Authors, please put here the full title of the article

\begin{center}
\begin{huge}
\textit{On the derivation of the homogeneous kinetic wave equation for a nonlinear random matrix model}

%\textit{some titles take two lines}

\end{huge}
\end{center}

\vspace{1cm}

%%      -------------------------------------------------------------------------------
%%      -------------------------- AUTHORS AND AFFILIATIONS ----------------------------
%%      -------------------------------------------------------------------------------
%% Authors, please put here your full names and affiliations

\begin{minipage}[t]{.28\textwidth}
\begin{center}
{\large{\bf{Guillaume Dubach}}} \\
\vskip0.15cm
\footnotesize{\'Ecole Polytechnique}
\end{center}
\end{minipage}
\hfill
\noindent
\begin{minipage}[t]{.28\textwidth}
\begin{center}
{\large{\bf{Pierre Germain}}} \\
\vskip0.15cm
\footnotesize{Imperial College London}
\end{center}
\end{minipage}
\hfill
\noindent
\begin{minipage}[t]{.28\textwidth}
\begin{center}
{\large{\bf{Benjamin Harrop-Griffiths}}} \\
\vskip0.15cm
\footnotesize{Georgetown University} 
\end{center}
\end{minipage}

\vspace{1cm}

%%% Please replace "James Mustard" below 
%%% with the name of the managing editor for your submission.
%%% If you are unsure about their identity
%%% please ask an editor-in-chief about.

\begin{center}
\noindent \em{Communicated by Zaher Hani}
\end{center}
\vspace{1cm}

%%      -------------------------------------------------------------------------------
%%      -------------------------- BEGIN ABSTRACT ----------------------------
%%      -------------------------------------------------------------------------------
%% Authors, please put here the ABSTRACT and KEYBOARDS

\noindent \textbf{Abstract.} \textit{We consider a nonlinear system of ODEs, where the underlying linear dynamics are determined by a Hermitian random matrix ensemble. We prove that the leading order dynamics in the weakly nonlinear, infinite volume limit are determined by a solution to the corresponding kinetic wave equation on a non-trivial timescale. Our proof relies on estimates for Haar-distributed unitary matrices obtained from Weingarten calculus, which may be of independent interest.}
\vskip0.3cm

\vspace{0.5cm}

%%      -------------------------------------------------------------------------------
%%      -------------------------- BEGIN ARTICLE ----------------------------
%%      -------------------------------------------------------------------------------
%% Authors, copy the body of your paper here

\section{Introduction}

\subsection{Wave turbulence, old and new}
When considering a dynamical system involving a large number of nonlinearly interacting waves, it is often preferable to consider ensembles of solutions instead of individual trajectories. This statistical mechanics approach to the dynamics is referred to as the theory of \emph{wave turbulence}.

A central question of this theory, which lies at the forefront of current research, is to rigorously derive a \emph{kinetic wave equation} for the non-equilibrium dynamics of the underlying microscopic system. Such kinetic wave equations first appeared in work of  Peierls~\cite{Peierls} and Hasselmann~\cite{hasselmann1, hasselmann2}, and have seen renewed interest after Zakharov, L'vov, and Falkovich \cite{ZLF} discovered power-law type stationary solutions analogous to the Kolmogorov spectra of hydrodynamic turbulence~\cite{kolmogorov2, kolmogorov1}.

In recent years, significant progress has been made on the problem of rigorously deriving the kinetic wave equation in the context of nonlinear dispersive equations. The most popular model has been the nonlinear Schr\"odinger equation 
\eq{S}{
i\p_tu + \Delta u = \mu^2 |u|^2u,
}
set on a torus, with spatially homogeneous random initial data. For this model, a full derivation of the kinetic limit was obtained in~\cite{DH2,DH3}, following earlier results in~\cite{BGHS, CG1, DH1,CG2}, and formal derivations in~\cite{MR4227166,dymov2021formal,dymov2021largeperiod}. Other models have also been considered: The foundational work~\cite{MR2755061} dealt with discrete NLS at statistical equilibrium, \cite{ST} focused on multiplicative white noise forcing (both reaching the kinetic time scale), and \cite{ACG} treated the case of spatially inhomogeneous data.

All of the aforementioned rigorous derivations consider PDEs whose linear dynamics enjoy strong algebraic structures, e.g., the eigenvalues and eigenvectors of the Laplacian on a Euclidean torus. In the present article we consider the opposite case, where the linear dynamics are so disordered that they are best understood in probabilistic terms and modeled as a random ensemble. We prove that the leading order dynamics (in a suitable scaling regime) are once again determined by the solution to a corresponding kinetic wave equation. This testifies to the universality of the kinetic description of wave turbulence.

\subsection{The model and its relevance} We consider solutions \(u\colon \R\to \C^{2N+1}\) of the ODE
\eq{NLS-OG}{
i\frac d{dt}u - Hu = \mu^2\bigl( u\odot \bar u\odot u\bigr),
}
where \(H\) is a \((2N+1)\times (2N+1)\) Hermitian random matrix, \(\mu>0\) is a constant, and
\[
(u\odot v)_j := u_jv_j
\]
defines the Hadamard product. Here, and throughout, we index \(\C^{2N+1}\) by \(j\in \llbracket-N,N\rrbracket\), where \(\llbracket a,b\rrbracket = [a,b]\cap\Z\).

The equation \eqref{NLS-OG} is readily seen to be globally well-posed for all Hermitian matrices \(H\) by combining Picard's Theorem with conservation of the mass,
\[
\cM(t) := \|u(t)\|_{\ell^2}^2 = \sum_{j=-N}^N|u_j(t)|^2.
\]
We remark that \eqref{NLS-OG} has one further conservation law: the Hamiltonian,
\[
\cH(t) := \<u,Hu\> + \tfrac{\mu^2}2\|u\|_{\ell^4}^4,
\]
where \(\<u,v\> = \sum_{j=-N}^N\bar u_jv_j\) and \(\|u\|_{\ell^p}^p = \sum_{j=-N}^N|u_j|^p\).

Our motivation for investigating the model \eqref{NLS-OG} is twofold. First, the Bohigas--Giannoni--Schmit conjecture \cite{MR730191} suggests that the spectral properties of quantum systems with \linebreak chaotic classical limits should be well-approximated, in a statistical sense, by random matrix ensembles. In this way, one may view \eqref{NLS-OG} as a toy model for \eqref{S} posed on a domain with a chaotic geodesic flow or, more generally, for a nonlinear interaction in a strongly mixing medium. Second, the derivation of the kinetic wave equation appears to require a very fine spectral resolution of the associated linear equation. This is currently only possible when the linear flow is integrable, and as discussed above, deterministic flows have already seen significant progress. In this paper, we turn to an example with random dynamics.

We subsequently take \(H = (h_{jk})\) to be distributed according to the \emph{Gaussian Unitary Ensemble} (GUE), defined as follows: The \(h_{jk}\) are independent for \(-N\leq j\leq k\leq N\), and the rescaled entry \(\sqrt{2N+1}\, h_{jk}\) is either a complex standard normal random variable when \(j<k\), or a real standard normal random variable when \(j = k\). GUE is an \textit{integrable ensemble} in the sense of random matrix theory. Moreover, its eigenvectors are independent from its eigenvalues and form a uniformly distributed orthonormal basis. This allows one to study them using the Weingarten calculus~\cite{MR1959915}, which will prove to be an indispensable ingredient in the ensuing analysis.

We expect that the results presented below can be generalized to other Gaussian random matrix ensembles using the corresponding Weingarten calculus for their eigenvectors~\cite{MR2217291,MR1959915}. Moreover, the universality paradigm in random matrix theory suggests one should be able to extend our results from the archetypal Gaussian ensembles to more general random matrix ensembles, e.g., Wigner ensembles.

In addition to the space-homogeneous, delocalized regime we consider here, we believe the model \eqref{NLS-OG} may be well-suited to a rigorous investigation of the effects of localization, or space-dependence. In~\cite{NST}, a heuristic derivation of a kinetic wave equation was proposed, starting from the nonlinear Schr\"odinger equation on the lattice with a random potential. In the diffusion limit, a porous medium equation was derived, the degenerate diffusion stemming from Anderson localization. In the Physics literature, understanding the combined effects of interaction (nonlinearity) and a disordered medium has received much attention. A kinetic description seems appropriate in the case of weak disorder, see~\cite{TRM,SWDC,SF} for recent publications on this subject. While a fully rigorous approach to these derivations seems out of reach, we believe that considering band matrices in~\eqref{NLS-OG} could provide an entry point to addressing questions about the effects of localization.

\subsection{The kinetic limit in the weakly turbulent regime}

Rather than working directly with the formulation \eqref{NLS-OG}, we instead diagonalize the matrix \(H\), taking \(\Lambda = \diag(\lambda)\) with
\[\lambda = (\lambda_{-N},\dots,\lambda_N)\in \R_{\geq}^{2N+1} := \bigl\{\lambda\in \R^{2n+1}:\lambda_{-N}\leq \dots\leq \lambda_N\bigr\},\] and \(\Psi\in \bbU(2N+1)\) so that
\[
H = \Psi\Lambda\Psi^*.
\]

When \(H\) is a GUE matrix, the vector of eigenvalues \(\lambda\) and the matrix of eigenvectors \(\Psi\) are independent. The vector of eigenvalues \(\lambda\) is distributed with joint density
\eq{Beta_ens}{
\frac 1{Z_N}
\prod\limits_{-N\leq j<k\leq N}
|\lambda_k - \lambda_j|^2 e^{-\frac12(2N+1)\,|\lambda|^2}\,d\lambda,
}
where the partition function \(Z_{N}>0\) is a normalizing constant and \(d\lambda\) is the Lebesgue measure on \(\R_{\geq}^{2N+1}\). The matrix of eigenvectors \(\Psi\) is distributed according to the \emph{Circular Unitary Ensemble} (CUE), i.e., the Haar probability measure on \(\bbU(2N+1)\).

We now introduce
\eq{diagonalized-variable}{
a := \sqrt N\Psi^*u,
}
and write the equation \eqref{NLS-OG} in the form
\eq{NLS}{
i\frac d{dt}a - \Lambda a = \frac{\mu^2}N\Psi^*\left(\Psi a\odot \overline{\Psi a}\odot\Psi a\right).
}
We consider initial data for \eqref{NLS} given by
\eq{Init}{
a_k(t=0) = A(\tfrac kN),
}
where \(A\in\Cont^1[-1,1]\) is a continuously differentiable, complex-valued function. We note that with this choice of data: 
\begin{itemize}
    \item The vector $u(t=0)$ has size $\sim 1$ in $\ell^2$;
    \item Its entries have typical size $N^{-1/2}$, or equivalently $a(t=0)$ has entries of typical size $1$.
\end{itemize}
It is helpful to keep in mind these orders of magnitude, which will also hold for positive time.

Taking \(a(t)\) to be the solution of \eqref{NLS} with initial data \eqref{Init}, we consider the regime where:
\begin{itemize}
    \item $N \to \infty$ (infinite volume limit);
    \item $\mu / \sqrt{N} \to 0$ (weakly nonlinear limit).
\end{itemize}
A heuristic derivation, whose validity we will investigate below, shows that in this regime
\eq{Goal}{
\bbE|a_k(t)|^2\sim \rho(\tfrac t{T_\kin},\tfrac kN),\qtq{where}T_{\kin} := \frac{N^2}{\mu^{4}}
}
and \(\rho\colon[0,\infty)\times [-1,1]\to[0,\infty)\) is a solution of the kinetic wave equation
\eq{KWE}{
\pde{\partial_t\rho(t,k) = \cC[\rho(t)](k),}{\rho(0,k) = |A(k)|^2.}
}
Here, the collision operator $\cC$ is given by
\[
\cC[\rho](k) := \tfrac\pi2\int_{[-1,1]^3}\delta\bigl(\Theta(k,\ell,m,n)\bigr)\rho(k)\rho(\ell)\rho(m)\rho(n)\Bigl[\tfrac1{\rho(k)} - \tfrac1{\rho(\ell)} + \tfrac1{\rho(m)} - \tfrac1{\rho(n)}\Bigr]\,d\ell \, dm\, dn,
\]
with phase function
\eq{DeterministicPhase}{
\Theta(k,\ell,m,n) := \nu(k) - \nu(\ell) + \nu(m) - \nu(n),
}
where $\nu:[-1,1] \to [-2,2]$ is defined by
\eq{nu}{
\int_0^{\nu(k)}d\sigma_{\mr{sc}} = \tfrac k2,
}
and the semicircle density
\eq{Semicircle}{d\sigma_{\mr{sc}}(x) := \tfrac1{2\pi}(4 - x^2)_+^{\frac12}\,dx\qtq{for}x_+ = \max\{x,0\}.}
We remark that \(\nu(k/N)\) represents the deterministic location of the \(k\)\textsuperscript{th} eigenvalue of the random matrix \(H\): see Lemma~\ref{l:LSC} below.

\subsection{Timescales and parameter ranges} Assuming that the solution \(a(t)\) of \eqref{NLS} remains roughly equidistributed in its entries, the relevant timescales for the problem are:
\begin{itemize}
\item The linear timescale, \(T_\lin = 1\), on which nontrivial linear effects occur;
\item The nonlinear timescale, \(T_\nonlin = N/\mu^2\), on which nontrivial nonlinear effects occur;
\item The kinetic timescale, \(T_\kin = N^2/\mu^4\), on which we expect \eqref{Goal} to be valid.
\end{itemize}
The weakly nonlinear regime \(\mu/\sqrt N\ll1\) corresponds to the case that
\[
T_\lin\ll T_\nonlin\ll T_\kin.
\]

On timescales \(T\ll T_\lin\), we expect the linear dynamics to dominate. This is readily seen from the fact that the resonance modulus,
\eq{Osc}{
\Omega(k,\ell,m,n) = \lambda_k - \lambda_\ell + \lambda_m - \lambda_n,
}
satisfies \(|\Omega|\lesssim 1\) with overwhelming probability. This suggests that at best we can hope \eqref{Goal} will hold for timescales \(T\) satisfying
\[
T_{\lin} = 1 \ll T\lesssim T_{\kin}.
\]

As we will discuss in Section~\ref{s:LSC}, the eigenvalues of \(H\) in the bulk of the spectrum fluctuate about their deterministic positions at scales \(\gg 1/N\). If \(T\gtrsim N\), the resonance modulus will not be equidistributed at scale \(1/T\), which prevents us from passing from the empirical distribution of the eigenvalues to the semicircle distribution. As a consequence, we only expect \eqref{Goal} to hold at times \(T\sim T_{\kin}\) if we have
\[
T_{\kin}\ll N.
\]

Combining these heuristics, we conjecture that \eqref{Goal} will hold in the regime where
\[
1\ll T\lesssim T_{\kin}\qtq{and}N^{\frac14}\ll\mu\ll N^{\frac12}.
\]

\subsection{The main result} Our main result proves that \eqref{Goal} does indeed describe the leading order asymptotic behavior in the regime where
\[
1\ll T\ll T_{\kin}^{\frac23}\qtq{and}N^{\frac14}\ll\mu\ll N^{\frac12}.
\]

\begin{theorem}\label{t:main}
Let \(1/4<\beta<1/2\) be fixed and \(A\in\Cont^1[-1,1]\). For each choice of integer \linebreak\(N\geq 1\), vector of eigenvalues \(\lambda\in \R^{2N+1}_{\geq}\), and matrix of eigenvectors \(\Psi\in \bbU(2N+1)\), let \linebreak\(a = a(N,\lambda,\Psi)\in C^\infty(\R;\C^{2N+1})\) be the smooth global solution of \eqref{NLS} corresponding to the choice of initial data \eqref{Init} and parameter \(\mu = N^\beta\).

Taking $\lambda$ and $\Psi$ to be distributed according to~\eqref{Beta_ens} and the Haar measure on $\mathbb{U}(2N+1)$, respectively, for any \(\epsilon>0\) and times \(N^\epsilon\leq t\leq N^{-\epsilon}T_\kin^{2/3}\), we have
\eq{Result}{
\bbE|a_k(t)|^2 = |A\bigl(\tfrac kN\bigr)|^2 + \tfrac t{T_{\kin}}\cC\bigl[|A|^2\bigr]\bigl(\tfrac kN\bigr) + R_k(t),
}
where
\eq{R}{
\tfrac1N\|R(t)\|_{\ell^1}\lesssim_{\beta,A,\epsilon} N^{-\frac\epsilon 4}\tfrac t{T_{\kin}},
}
and \(T_{\kin} = N^2/\mu^4\) is defined as in \eqref{Goal}.
\end{theorem}

The timescale \(T_{\kin}^{2/3}\) obtained in Theorem~\ref{t:main} is nontrivial, in the sense that \(T_{\nonlin} = T_{\kin}^{1/2}\ll T_{\kin}^{2/3}\). While our proof follows a similar strategy to previous work, e.g.,~\cite{BGHS,DH1,DH2,ACG,CG1,CG2}, the model \eqref{NLS} requires multiple innovations to obtain estimates in the random matrix setting. Indeed, as the reader will soon appreciate, some significantly involved analysis is required to make even a modest step towards the conjectured range on which we expect \eqref{Goal} to hold.

We begin our proof of Theorem~\ref{t:main} by iterating Duhamel's formula to obtain an approximate solution to \eqref{NLS}. The first few iterates in this construction are well-approximated by the kinetic wave equation, and are analyzed in detail Section~\ref{s:LOT}. The remainder is bounded by writing each iterate in terms of the initial data. These higher order iterates are analyzed in Section~\ref{s:app}.

A key point of departure from previous results is that all frequencies in our model interact together during each nonlinear interaction. To control the complex behavior of this soup of interacting waves we require several tools from random matrix theory, which are discussed in Section~\ref{s:WG}.

For the eigenvalues, we rely on rigidity estimates derived from the local semicircle law. These enable us to treat the eigenvalues as being located at their deterministic positions, up to small fluctuations.

For the eigenvectors, we crucially rely on the Weingarten calculus: an analog of the Wick calculus in the setting of random unitary matrices. To implement the Weingarten calculus, we introduce a graph-theoretic approach to obtaining bounds. Several of our estimates for CUE matrices appear to be new and we believe they may be of independent interest.

Once we have constructed our approximate solution, we complete the proof of Theorem~\ref{t:main} by estimating the remainder using the contraction principle. This relies on estimates for the linearization of \eqref{NLS} about the approximate solution, which are proved in Section~\ref{s:lin}. Here, we use a \(T^*T\) argument to reduce the problem to similar estimates to those obtained for the approximate solution itself.

The remainder of the paper is structured as follows: In the final part of this introduction we present some properties of the kinetic wave equation \eqref{KWE}. In Section~\ref{s:P} we discuss the construction of the approximate solution, state our main estimates, and then prove Theorem~\ref{t:main}. In Section~\ref{s:WG} we introduce the tools we require from random matrix theory and prove several estimates for random unitary matrices. In Section~\ref{s:LOT} we consider the leading order terms in our approximate solution, from which the kinetic wave equation arises. In Section~\ref{s:app} we prove estimates for the remaining terms in our approximate solution. Finally, in Section~\ref{s:lin} we consider the linearization of \eqref{NLS} about the approximate solution.

\subsection{Properties of the kinetic wave equation} We conclude this introduction by discussing some properties of the kinetic wave equation \eqref{KWE}. From \eqref{nu}, we see that
\eq{nudiff}{
\nu'(k) = \pi \bigl(4-\nu(k)^2\bigr)_+^{-\frac12}\qtq{for}k\in(-1,1).
}
In particular, denoting \(L^p = L^p[-1,1]\) and taking \(q\in L^\infty\) and \(f,g,h\in L^1\), we have
\begin{align*}
&\left|\tfrac\pi2\int_{[-1,1]^4}\delta\bigl(\Theta(k,\ell,m,n)\bigr)q(k)f(\ell)g(m)h(n)\,dk\,d\ell\,dm\,dn\right|\\
&= \left|\tfrac12\int_{[-1,1]^3}\!q\bigl(\nu^{-1}\big[\nu(\ell) - \nu(m) + \nu(n)\bigr]\bigr) f(\ell)g(m)h(n)\,\bigl(4\!-\!\bigl[\nu(\ell)\!-\!\nu(m)\!+\!\nu(n)\bigr]^2 \bigr)^{\frac12}_+ \,d\ell\, dm\,dn \right|\\
&\lesssim \|q\|_{L^\infty}\|f\|_{L^1}\|g\|_{L^1}\|h\|_{L^1}.
\end{align*}
By symmetry and duality, we obtain an estimate for the the collision operator,
\[
\|\cC[\rho] - \cC[\phi]\|_{L^1}\lesssim \bigl(\|\rho\|_{L^1} + \|\phi\|_{L^1}\bigr)^2\|\rho - \phi\|_{L^1}.
\]
As a consequence, the map \(\rho\mapsto \cC[\rho]\) is locally Lipschitz on \(L^1\) and Picard's Theorem gives local well-posedness of \eqref{KWE} for any initial data in \(L^1\). We note that the collision operator maps continuous functions to continuous functions, and hence the flow may be restricted to the subspace \(\Cont[-1,1]\subset L^1\).

Choosing a test function \(f\in \Cont[-1,1]\) and symmetrizing gives us the formula
\begin{align*}
\frac{d}{dt} \int_{-1}^1 f(k) \rho(t,k)\,dk &= \tfrac\pi8\int_{[-1,1]^4} \delta\bigl(\Theta(k,\ell,m,n)\bigr)\rho(t,k)\rho(t,\ell)\rho(t,m)\rho(t,n)\\
&\qquad\qquad\qquad\times\Bigl[\tfrac1{\rho(t,k)}- \tfrac1{\rho(t,\ell)} + \tfrac1{\rho(t,m)} - \tfrac1{\rho(t,n)}\Bigr]\\
&\qquad\qquad\qquad\qquad\times\Bigl[ f(k) - f(\ell) + f(m) - f(n) \Bigr] \,dk \,d\ell \, dm \,dn.
\end{align*}
Choosing successively $f(k)=1$ and $f(k)=\nu(k)$, we obtain the conservation of mass and energy,
\[
\frac{d}{dt} \int_{-1}^1 \rho(t,k)\,dk = 0\qtq{and}\frac{d}{dt} \int_{-1}^1 \nu(k) \rho(t,k)\,dk = 0.
\]

Writing \eqref{KWE} in the form
\[
\p_t\rho(t,k) = g(t,k)\rho(t,k) + h(t,k),
\]
where
\begin{align*}
g(t,k) &= \tfrac\pi2\int_{[-1,1]^3}\delta\bigl(\Theta(k,\ell,m,n)\bigr)\rho(t,\ell)\rho(t,m)\rho(t,n)\Bigl[- \tfrac1{\rho(t,\ell)} + \tfrac1{\rho(t,m)} - \tfrac1{\rho(t,n)}\Bigr]\,d\ell \, dm\, dn,\\
h(t,k) &= \tfrac\pi2\int_{[-1,1]^3}\delta\bigl(\Theta(k,\ell,m,n)\bigr)\rho(t,\ell)\rho(t,m)\rho(t,n)\,d\ell \, dm\, dn,
\end{align*}
we see that if the initial data \(\rho(t=0)\) is non-negative (respectively positive) then \(\rho(t)\) is non-negative (respectively positive) for all times \(t\geq 0\). In particular, conservation of mass ensures that \eqref{KWE} is globally well-posed, forwards in time, for all non-negative initial data in \(L^1\).

Taking \(\rho(t=0)\) to be positive and \(F\in \Cont^1(\R_+)\), another symmetrization gives the expression
\begin{align*}
&\frac{d}{dt} \int_{-1}^1 F\bigl(\rho(t,k)\bigr) \,dk\\
&\qquad =\tfrac\pi8\int_{[-1,1]^4} \delta\bigl(\Theta(k,\ell,m,n)\bigr)\rho(t,k)\rho(t,\ell)\rho(t,m)\rho(t,n)\\
&\qquad \quad\qquad\qquad\quad\times\Bigl[\tfrac1{\rho(t,k)}- \tfrac1{\rho(t,\ell)} + \tfrac1{\rho(t,m)} - \tfrac1{\rho(t,n)}\Bigr]\\
&\qquad \quad\qquad \qquad \qquad\quad\quad\times\Bigl[F'\bigl(\rho(t,k)\bigr) - F'\bigl(\rho(t,\ell)\bigr) + F'\bigl(\rho(t,m)\bigr) - F'\bigl(\rho(t,n)\bigr)\Bigr] \,dk \,d\ell \, dm \,dn.
\end{align*}
Setting $F(\rho) = \log \rho$ results in the $H$-theorem (entropy increases in time) for positive solutions
\begin{align*}
\frac{d}{dt} \int_{-1}^1 \log\rho(t,k) \,dk&=\tfrac\pi8\int_{[-1,1]^4} \delta\bigl(\Theta(k,\ell,m,n)\bigr)\rho(t,k)\rho(t,\ell)\rho(t,m)\rho(t,n)\\
&\quad\qquad\qquad\quad\times\Bigl[\tfrac1{\rho(t,k)}- \tfrac1{\rho(t,\ell)} + \tfrac1{\rho(t,m)} - \tfrac1{\rho(t,n)}\Bigr]^2 \,dk \,d\ell \, dm \,dn\\
&\geq 0.
\end{align*}

The entropy dissipation vanishes on the Rayleigh--Jeans distribution
$$
\phi(k) = \frac\alpha{\beta + \nu(k)},
$$
where $\alpha>0$ and $\beta>2$ are constants. It is natural to conjecture that the Rayleigh--Jeans distributions are maximizers of the entropy for fixed mass and energy, see~\cite{RST} for a proof of this fact in a closely related context, along with a convergence result as $t \to \infty$.

\subsection*{Acknowledgments.} The authors are grateful to Paul Bourgade for several very helpful discussions and providing a number of important references. The authors also thank Beno\^it Collins, Sho Matsumoto and Jonathan Novak for informative discussions about the Weingarten calculus. The authors are very grateful to the anonymous referee for several insightful comments and suggestions.

G. Dubach gratefully acknowledges funding from the European Union's Horizon 2020 research and innovation programme under the Marie Sk{\l}odowska-Curie Grant Agreement No. 754411. P. Germain was supported by the Simons collaborative grant on weak turbulence. B. Harrop-Griffiths is grateful to the Department of Mathematics at UCLA, which he was a member of when this work was completed.

\section{Statement of the main estimates and proof of Theorem~\ref{t:main}}\label{s:P}

In this section, we first introduce some preliminary reductions, identities, and estimates. We then state our main linear and nonlinear estimates, and use these to prove Theorem~\ref{t:main}.

\subsection{Wick ordering} As in the case of the NLS equation considered in, e.g.,~\cite{CG1,CG2,DH1,DH2}, we introduce a ``Wick ordering'' to remove certain resonant terms. Precisely, we introduce a time-dependent rotation
\[
\underline u(t) = e^{it\mu^2\frac{2\cM}{2N+1}} u(t)
\]
so that \eqref{NLS-OG} becomes
\[
i\frac d{dt}\underline u - H\underline u = \mu^2\Bigl[\bigl(\underline u\odot\bar{\underline u}\odot \underline u\bigr) - \tfrac{2\cM}{2N+1} \underline u\Bigr],
\]
where we note that
\(
\cM = \|u\|_{\ell^2}^2 = \|\underline u\|_{\ell^2}^2
\)
is still conserved.

Taking
\eq{diagonalized-variable-WO}{
\underline a:=\sqrt N\Psi^*\underline u,
}
the equation \eqref{NLS} becomes
\eq{NLS-WO}{
i\frac d{dt}\underline a - \Lambda\underline a = \frac{\mu^2}N\Bigl[\Psi^*\bigl(\Psi \underline a\odot \overline{\Psi \underline a}\odot \Psi \underline a\bigr) - \tfrac{2N\cM}{2N+1}\underline a\Bigr].
}
Observing that \(|a_k|^2 = |\underline a_k|^2\), we subsequently drop the underlines and take \eqref{diagonalized-variable-WO} to be our definition of \(a\) instead of \eqref{diagonalized-variable}.

\subsection{Decomposing the solution}  We now write \eqref{NLS-WO} using the Duhamel formula as
\eq{Duhamel}{
a(t) = e^{-it\Lambda}a(0) - i\int_0^t e^{-i(t-s)\Lambda}\cN[a(s),a(s),a(s)]\,ds,
}
where we denote
\begin{align*}
\cN[a,b,c] &= \tfrac{\mu^2}{3N}\Bigl[\Psi^*\bigl(\Psi a(s)\odot\overline{\Psi b(s)}\odot \Psi c(s)\bigr) - \tfrac1{2N+1}\<b,c\> a - \tfrac1{2N+1}\<b,a\> c\bigr]\\
&\quad + \tfrac{\mu^2}{3N}\Bigl[\Psi^*\bigl(\Psi b(s)\odot\overline{\Psi c(s)}\odot \Psi a(s)\bigr)-\tfrac1{2N+1}\<c,a\>b - \tfrac1{2N+1}\<c,b\>a\Bigr]\\
&\quad + \tfrac{\mu^2}{3N}\Bigl[\Psi^*\bigl(\Psi c(s)\odot\overline{\Psi a(s)}\odot \Psi b(s)\bigr)-\tfrac1{2N+1}\<a,b\>c-\tfrac1{2N+1}\<a,c\>b\Bigr],
\end{align*}
and recall that our convention is that the inner product \(\<\cdot,\cdot\>\) is linear in the second variable. We remark that if
\eq{Mom}{
\gamma(k,\ell,m,n) := \sum_{j=-N}^N \overline{\psi_{jk}}\,\psi_{j \ell}\,\overline{\psi_{jm}}\,\psi_{jn} - \tfrac1{2N+1} \bbo_{\{k=\ell,m=n\}}-\tfrac1{2N+1} \bbo_{\{k=n,\ell=m\}}
}
then the \(k\)\textsuperscript{th} entry
\[
\cN[a,b,c]_k = \tfrac{\mu^2}{3N}\sum_{\ell,m,n}\gamma(k,\ell,m,n)\,\bigl[a_\ell\,\overline{ b_m}\, c_n + b_\ell\,\overline{ c_m}\, a_n + c_\ell\,\overline{ a_m}\, b_n\bigr].
\]
For a sufficiently large integer \(M\geq1\), we construct an approximate solution of \eqref{Duhamel} by
\eq{app-sum}{
a^\app(t) = \sum\limits_{m=0}^M a\sbrack m(t),
}
where
\[
a\sbrack 0(t) = e^{-it\Lambda}a(0),
\]
and for \(m\geq 0\)
\eq{inductive am+1}{
a\sbrack{m+1}(t) = -i\sum\limits_{m_1 + m_2 + m_3 = m}\int_0^t e^{-i(t-s)\Lambda}\cN\bigl[a\sbrack{m_1}(s),a\sbrack{m_2}(s),a\sbrack{m_3}(s)\bigr]\,ds.
}

We then introduce the error term
\eq{Error}{
\cE = - i\frac d{dt}a^\app + \Lambda a^\app + \cN[a^\app,a^\app,a^\app],
}
so that the remainder
\eq{Decomp}{
a^\err(t) = a(t) - a^\app(t)
}
satisfies
\eq{Remainder}{
\pde{
i\dfrac d{dt}a^\err - \Lambda a^\err = \cL a^\err + 3\cN[a^\app,a^\err,a^\err] + \cN[a^\err,a^\err,a^\err] + \cE,
}{
a^\err(0) = 0,
}
}
where the real linear operator
\eq{Linearization}{
\cL a = 3\cN\bigl[a^\app,a^\app,a\bigr].
}

\subsection{Feynman diagrams and Duhamel iterates}\label{s:Feyn}

To understand the structure of \(a\sbrack m\), it will be useful to introduce \emph{Feynman diagrams} that encode the expansion of the Duhamel integral using graphs. These diagrams will play an important role in our estimates for the approximate solution and linear operator in Sections~\ref{s:app} and~\ref{s:lin}.

Feynman diagrams originally appeared in derivations of the kinetic wave equation in the work of Lukkarinen and Spohn~\cite{MR2755061}; see also~\cite{MR2333210} for a related application. Due to the particular features of the random matrix setting, we adopt some slightly different conventions to expedite the framework.

A Feynman diagram for \(a\sbrack m\) is a graph \(G\) with \((m+1)^2+1\) vertices and \((m+1)^2\) edges. The shape of the graph is uniquely determined by an \emph{interaction history}, \(\ell = (\ell_1,\dots,\ell_m)\), where each \(\ell_r\in \llbracket1,2(m-r)+1\rrbracket\).

The vertices of the graph represent the linear and nonlinear interactions of waves. We index the vertices by \(v_{r,j}\), where \(r\in \llbracket0,m+1\rrbracket\) and \(j\in \llbracket1,2(m-r)+1\rrbracket\) for \(0\leq r\leq m\), with \(j=1\) when \(r=m+1\). We distinguish four sets of vertices:
\begin{itemize}
\item The output vertex, \(v_{m+1,1}\), which we denote by \tikz[baseline=-2.4,scale=0.15]{\node[outdot] {};};\smallskip
\item The input vertices, \(v_{0,j}\), which we denote by \tikz[baseline=-2.4,scale=0.15]{\node[indot] {};};\smallskip
\item The linear interaction vertices, \(v_{r,j}\) for \(1\leq r\leq m\) and \(j\neq \ell_r\), which we denote by \tikz[baseline=-2.4,scale=0.15]{\node[ghostdot] {};};\smallskip
\item The nonlinear interaction vertices, \(v_{r,\ell_r}\) for \(1\leq r\leq m\), which we denote by \tikz[baseline=-2.4,scale=0.15]{\node[dot] {};}.
\end{itemize}

The edges of the graph represent the linear propagation of waves. We index the edges by \(e_{r,j}\), where \(r\in \llbracket0,m\rrbracket\) and \(j\in \llbracket1,2(m-r)+1\rrbracket\). For \(0\leq r\leq m-1\), the edge \(e_{r,j}\) connects the vertex
\[
v_{r,j}\qtq{to}\begin{cases}v_{r+1,j}&\qtq{if}j<\ell_{r+1},\medskip\\v_{r+1,\ell_{r+1}}&\qtq{if}\ell_{r+1}\leq j\leq\ell_{r+1}+2,\medskip\\v_{r+1,j-2}&\qtq{if}j>\ell_{r+1}+2.\end{cases}
\]
The edge \(e_{m,1}\) connects the vertex \(v_{m,1}\) to \(v_{m+1,1}\).

Given a time \(t>0\) and a frequency \(k\in \llbracket-N,N\rrbracket\), we convert each Feynman diagram \(G=G(\ell)\) into an integral by imposing a set of \emph{Feynman rules} that encode the contribution of each element of the graph. We then sum over all choices of \(\ell\) to obtain an expression for \(a\sbrack m_k(t)\).

Before imposing these rules, we first label each edge, \(e_{r,j}\), of the graph with an integer \(k_{r,j}\in \llbracket-N,N\rrbracket\). This represents the frequency of the corresponding linear wave. We denote the array of frequencies by \(\cK = (k_{r,j})\).

Second, we choose times \(s = (s_0,\dots,s_m)\in \R_+^{m+1}\) so that, with \(t_0 = 0\) and \(t_r = \sum_{j=0}^{r-1} s_j\) for \(r\in\llbracket1, m+1\rrbracket\), we obtain a partition \(0=t_0<t_1<\dots<t_{m+1}=t\). We will subsequently treat time as the vertical direction in our graph. For each \(r\), we locate the vertices \(v_{r,j}\) at time \(t_r\) with the index \(j\) running from left to right. The edges \(e_{r,j}\) then lie in the time interval \((t_r,t_{r+1})\) with the index \(j\) again running from left to right.

We now impose the following Feynman rules for our graph:
\begin{itemize}
\item The output vertex, \(v_{m+1,1}\), contributes a factor of \(\bbo_{\{k= k_{m,1}\}}\,\delta\left(\sum_{r=0}^ms_r - t\right)\).
\smallskip
\item The input vertices, \(v_{0,j}\), contribute factors of
\[
\begin{cases}A\bigl(\tfrac{k_{0,j}}N\bigr)&\text{if }j\text{ is odd},\medskip\\
\overline{ A\bigl(\tfrac{k_{0,j}}N\bigr)}&\text{if }j\text{ is even}.
\end{cases}
\]
\smallskip
\item The linear interaction vertices, \(v_{r,j}\) for \(1\leq r\leq m\) and \(j\neq \ell_r\), contribute factors of
\[
\begin{cases}
\bbo_{\{k_{r-1,j} = k_{r,j}\}}&\qtq{if}j<\ell_r,\medskip\\
\bbo_{\{k_{r-1,j} = k_{r,j-2}\}}&\qtq{if}j>\ell_r+2.
\end{cases}
\]
\smallskip
\item The nonlinear interaction vertices, \(v_{r,\ell_r}\) for \(1\leq r\leq m\), contribute factors of \linebreak\(-i\frac{\mu^2}N\gamma_{r,\ell_r}\), where
\eq{Momm}{
\gamma_{r,\ell_r}(\cK) = \begin{cases} \gamma\bigl(k_{r,\ell_r},k_{r-1,\ell_r},  k_{r-1,\ell_r+1},k_{r-1,\ell_r+2}\bigr)&\qtq{if}\ell_r\text{ is odd},\medskip\\-\overline{\gamma\bigl(k_{r,\ell_r},k_{r-1,\ell_r},  k_{r-1,\ell_r+1},k_{r-1,\ell_r+2}\bigr)}&\qtq{if}\ell_r\text{ is even},\end{cases}
}
and \(\gamma\) is defined as in \eqref{Mom}.
\smallskip
\item The edges \(e_{r,j}\) contribute factors of
\[
\begin{cases}
e^{-is_j\lambda_{k_{r,j}}}&\text{if }j\text{ is odd},\medskip\\
e^{is_j\lambda_{k_{r,j}}}&\text{if }j\text{ is even}.
\end{cases}
\]
\end{itemize}
An expression for \(a\sbrack m_k(t)\) is now obtained by multiplying these contributions together, integrating over \(s\in \R_+^{m+1}\), and summing over \(\ell,\cK\).

As an example, we have the following:
\begin{ex}[The case \(m=2\)]
For \(m=2\) we have 3 possible diagrams:
\begin{figure}[h!]
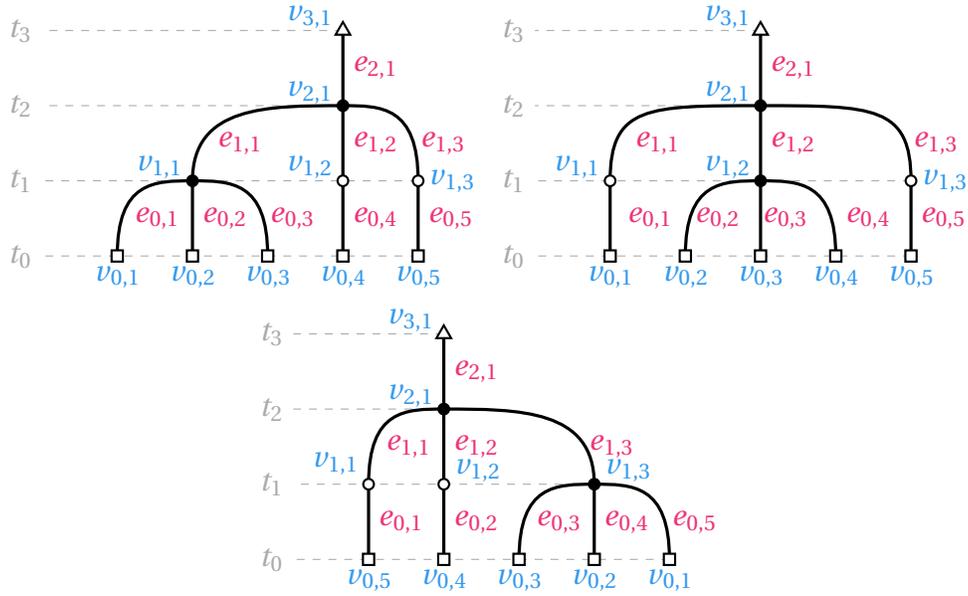

\begin{center}
\tikz{
\draw[dashed,col7] (2,0)--(-3,0) node[left] {\(t_0\)};
\draw[dashed,col7] (2,1)--(-3,1) node[left] {\(t_1\)};
\draw[dashed,col7] (1,2)--(-3,2) node[left] {\(t_2\)};
\draw[dashed,col7] (1,3)--(-3,3) node[left] {\(t_3\)};
\draw[very thick] (1,0)--(1,1) {};
\draw[very thick] (1,2)--(1,3) {};
\draw[very thick] (-1,0)--(-1,1) {};
\draw[very thick] (2,0)--(2,1) {};
\draw[very thick] (1,1)--(1,2) {};
\draw[very thick] (-2,0) .. controls (-2,1) and (-1.5,1) .. (-1,1);
\draw[very thick] (0,0) .. controls (0,1) and (-.5,1) .. (-1,1);
\draw[very thick] (2,1) .. controls (2,2) and (1.5,2) .. (1,2);
\draw[very thick] (-1,1) .. controls (-1,2) and (0,2) .. (1,2);
\node[indot] at (-2,0) {};
\node[below] at (-2,0) {\({\color{col5}v_{0,1}}\)};
\node[indot] at (-1,0) {};
\node[below] at (-1,0) {\({\color{col5}v_{0,2}}\)};
\node[indot] at (0,0) {};
\node[below] at (0,0) {\({\color{col5}v_{0,3}}\)};
\node[indot] at (1,0) {};
\node[below] at (1,0) {\({\color{col5}v_{0,4}}\)};
\node[indot] at (2,0) {};
\node[below] at (2,0) {\({\color{col5}v_{0,5}}\)};
\node[dot] at (-1,1) {};
\node[above left] at (-1,.9) {\({\color{col5}v_{1,1}}\)};
\node[dot] at (1,2) {};
\node[above left] at (1,1.9) {\({\color{col5}v_{2,1}}\)};
\node[ghostdot] at (1,1) {};
\node[above left] at (1,.9) {\({\color{col5}v_{1,2}}\)};
\node[ghostdot] at (2,1) {};
\node[right] at (2,1) {\({\color{col5}v_{1,3}}\)};
\node[right] at (-1.9,.5) {\({\color{col6}e_{0,1}}\)};
\node[right] at (-1,.5) {\({\color{col6}e_{0,2}}\)};
\node[right] at (-.1,.5) {\({\color{col6}e_{0,3}}\)};
\node[right] at (1,.5) {\({\color{col6}e_{0,4}}\)};
\node[right] at (2,.5) {\({\color{col6}e_{0,5}}\)};
\node[right] at (-.8,1.5) {\({\color{col6}e_{1,1}}\)};
\node[right] at (1,1.5) {\({\color{col6}e_{1,2}}\)};
\node[right] at (1.9,1.5) {\({\color{col6}e_{1,3}}\)};
\node[right] at (1,2.5) {\({\color{col6}e_{2,1}}\)};
\node[outdot] at (1,3) {};
\node[above left] at (1,2.9) {\({\color{col5}v_{3,1}}\)};
}
\tikz{
\draw[dashed,col7] (2,0)--(-3,0) node[left] {\(t_0\)};
\draw[dashed,col7] (2,1)--(-3,1) node[left] {\(t_1\)};
\draw[dashed,col7] (0,2)--(-3,2) node[left] {\(t_2\)};
\draw[dashed,col7] (0,3)--(-3,3) node[left] {\(t_3\)};
\draw[very thick] (0,0)--(0,1) {};
\draw[very thick] (0,2)--(0,3) {};
\draw[very thick] (-2,0)--(-2,1) {};
\draw[very thick] (2,0)--(2,1) {};
\draw[very thick] (0,1)--(0,2) {};
\draw[very thick] (-1,0) .. controls (-1,1) and (-0.5,1) .. (0,1);
\draw[very thick] (1,0) .. controls (1,1) and (0.5,1) .. (0,1);
\draw[very thick] (2,1) .. controls (2,2) and (1.5,2) .. (0,2);
\draw[very thick] (-2,1) .. controls (-2,2) and (-1.5,2) .. (0,2);
\node[indot] at (-2,0) {};
\node[below] at (-2,0) {\({\color{col5}v_{0,1}}\)};
\node[indot] at (-1,0) {};
\node[below] at (-1,0) {\({\color{col5}v_{0,2}}\)};
\node[indot] at (0,0) {};
\node[below] at (0,0) {\({\color{col5}v_{0,3}}\)};
\node[indot] at (1,0) {};
\node[below] at (1,0) {\({\color{col5}v_{0,4}}\)};
\node[indot] at (2,0) {};
\node[below] at (2,0) {\({\color{col5}v_{0,5}}\)};
\node[dot] at (0,1) {};
\node[above left] at (-2,.9) {\({\color{col5}v_{1,1}}\)};
\node[dot] at (0,2) {};
\node[above left] at (0,1.9) {\({\color{col5}v_{2,1}}\)};
\node[ghostdot] at (2,1) {};
\node[above left] at (0,.9) {\({\color{col5}v_{1,2}}\)};
\node[ghostdot] at (-2,1) {};
\node[right] at (2,1) {\({\color{col5}v_{1,3}}\)};
\node[right] at (-1.9,.5) {\({\color{col6}e_{0,1}}\)};
\node[right] at (-1,.5) {\({\color{col6}e_{0,2}}\)};
\node[right] at (-.1,.5) {\({\color{col6}e_{0,3}}\)};
\node[right] at (1,.5) {\({\color{col6}e_{0,4}}\)};
\node[right] at (2,.5) {\({\color{col6}e_{0,5}}\)};
\node[right] at (-1.8,1.5) {\({\color{col6}e_{1,1}}\)};
\node[right] at (0,1.5) {\({\color{col6}e_{1,2}}\)};
\node[right] at (1.9,1.5) {\({\color{col6}e_{1,3}}\)};
\node[right] at (0,2.5) {\({\color{col6}e_{2,1}}\)};
\node[outdot] at (0,3) {};
\node[above left] at (0,2.9) {\({\color{col5}v_{3,1}}\)};
}
\tikz{
\draw[dashed,col7] (2,0)--(-3,0) node[left] {\(t_0\)};
\draw[dashed,col7] (1,1)--(-3,1) node[left] {\(t_1\)};
\draw[dashed,col7] (-1,2)--(-3,2) node[left] {\(t_2\)};
\draw[dashed,col7] (-1,3)--(-3,3) node[left] {\(t_3\)};
\draw[very thick] (-1,0)--(-1,1) {};
\draw[very thick] (-1,2)--(-1,3) {};
\draw[very thick] (1,0)--(1,1) {};
\draw[very thick] (-2,0)--(-2,1) {};
\draw[very thick] (-1,1)--(-1,2) {};
\draw[very thick] (2,0) .. controls (2,1) and (1.5,1) .. (1,1);
\draw[very thick] (0,0) .. controls (0,1) and (.5,1) .. (1,1);
\draw[very thick] (-2,1) .. controls (-2,2) and (-1.5,2) .. (-1,2);
\draw[very thick] (1,1) .. controls (1,2) and (0,2) .. (-1,2);
\node[indot] at (2,0) {};
\node[below] at (2,0) {\({\color{col5}v_{0,1}}\)};
\node[indot] at (1,0) {};
\node[below] at (1,0) {\({\color{col5}v_{0,2}}\)};
\node[indot] at (0,0) {};
\node[below] at (0,0) {\({\color{col5}v_{0,3}}\)};
\node[indot] at (-1,0) {};
\node[below] at (-1,0) {\({\color{col5}v_{0,4}}\)};
\node[indot] at (-2,0) {};
\node[below] at (-2,0) {\({\color{col5}v_{0,5}}\)};
\node[dot] at (1,1) {};
\node[above right] at (1,.9) {\({\color{col5}v_{1,3}}\)};
\node[dot] at (-1,2) {};
\node[above left] at (-1,1.9) {\({\color{col5}v_{2,1}}\)};
\node[ghostdot] at (-1,1) {};
\node[above right] at (-1,.9) {\({\color{col5}v_{1,2}}\)};
\node[ghostdot] at (-2,1) {};
\node[above left] at (-2,1) {\({\color{col5}v_{1,1}}\)};
\node[right] at (1.9,.5) {\({\color{col6}e_{0,5}}\)};
\node[right] at (1,.5) {\({\color{col6}e_{0,4}}\)};
\node[right] at (.1,.5) {\({\color{col6}e_{0,3}}\)};
\node[right] at (-1,.5) {\({\color{col6}e_{0,2}}\)};
\node[right] at (-2,.5) {\({\color{col6}e_{0,1}}\)};
\node[right] at (.8,1.5) {\({\color{col6}e_{1,3}}\)};
\node[right] at (-1,1.5) {\({\color{col6}e_{1,2}}\)};
\node[right] at (-1.9,1.5) {\({\color{col6}e_{1,1}}\)};
\node[right] at (-1,2.5) {\({\color{col6}e_{2,1}}\)};
\node[outdot] at (-1,3) {};
\node[above left] at (-1,2.9) {\({\color{col5}v_{3,1}}\)};
}
\end{center}
\caption{Feynman diagrams for \(m=2\) with \(\ell = (1,1),(2,1),(3,1)\), respectively.}
\end{figure}

The corresponding contribution to our integral for the \(\ell = (1,1)\) diagram is
\begin{align*}
&\underbrace{\bbo_{\{k = k_{2,1}\}}\delta(s_0 + s_1 + s_2 - t)}_{{\color{col5}v_{3,1}}}\times \underbrace{e^{-is_2\lambda_{k_{2,1}}}}_{{\color{col6}e_{2,1}}}\times \underbrace{-i\tfrac{\mu^2}N\gamma_{2,1}}_{{\color{col5}v_{2,1}}}\times \underbrace{e^{-is_1\lambda_{k_{1,1}}}}_{{\color{col6}e_{1,1}}}\times \underbrace{e^{is_1\lambda_{k_{1,2}}}}_{{\color{col6}e_{1,2}}}\times \underbrace{e^{-is_1\lambda_{k_{1,3}}}}_{{\color{col6}e_{1,3}}}\\
&\qquad\times  \underbrace{-i\tfrac{\mu^2}N\gamma_{1,1}}_{{\color{col5}v_{1,1}}}\times \underbrace{\bbo_{\{k_{0,4} = k_{1,2}\}}}_{{\color{col5}v_{1,2}}}\times\underbrace{\bbo_{\{k_{0,5} = k_{1,3}\}}}_{{\color{col5}v_{1,3}}}\times \underbrace{e^{-is_0\lambda_{k_{0,1}}}}_{{\color{col6}e_{0,1}}}\times \underbrace{e^{is_0\lambda_{k_{0,2}}}}_{{\color{col6}e_{0,2}}}\times \underbrace{e^{-is_0\lambda_{k_{0,3}}}}_{{\color{col6}e_{0,3}}}\times \underbrace{e^{is_0\lambda_{k_{0,4}}}}_{{\color{col6}e_{0,4}}}\times \underbrace{e^{-is_0\lambda_{k_{0,5}}}}_{{\color{col6}e_{0,5}}}\\
&\qquad \qquad \times  \underbrace{A\bigl(\tfrac{k_{0,1}}N\bigr)}_{{\color{col5}v_{0,1}}}\times  \underbrace{\overline{A\bigl(\tfrac{k_{0,2}}N\bigr)}}_{{\color{col5}v_{0,2}}}\times  \underbrace{A\bigl(\tfrac{k_{0,3}}N\bigr)}_{{\color{col5}v_{0,3}}}\times  \underbrace{\overline{A\bigl(\tfrac{k_{0,4}}N\bigr)}}_{{\color{col5}v_{0,4}}}\times  \underbrace{A\bigl(\tfrac{k_{0,5}}N\bigr)}_{{\color{col5}v_{0,5}}},
\end{align*}
for the \(\ell = (2,1)\) diagram, it is
\begin{align*}
&\underbrace{\bbo_{\{k = k_{2,1}\}}\delta(s_0 + s_1 + s_2 - t)}_{{\color{col5}v_{3,1}}}\times\underbrace{e^{-is_2\lambda_{k_{2,1}}}}_{{\color{col6}e_{2,1}}}\times \underbrace{-i\tfrac{\mu^2}N\gamma_{2,1}}_{{\color{col5}v_{2,1}}}  \times \underbrace{e^{-is_1\lambda_{k_{1,1}}}}_{{\color{col6}e_{1,1}}}\times \underbrace{e^{is_1\lambda_{k_{1,2}}}}_{{\color{col6}e_{1,2}}}\times \underbrace{e^{-is_1\lambda_{k_{1,3}}}}_{{\color{col6}e_{1,3}}}\\
&\qquad\times \underbrace{\bbo_{\{k_{0,1} = k_{1,1}\}}}_{{\color{col5}v_{1,1}}}\times  \underbrace{-i\tfrac{\mu^2}N\gamma_{1,2}}_{{\color{col5}v_{1,2}}}\times\underbrace{\bbo_{\{k_{0,5} = k_{1,3}\}}}_{{\color{col5}v_{1,3}}} \times \underbrace{e^{-is_0\lambda_{k_{0,1}}}}_{{\color{col6}e_{0,1}}}\times \underbrace{e^{is_0\lambda_{k_{0,2}}}}_{{\color{col6}e_{0,2}}}\times \underbrace{e^{-is_0\lambda_{k_{0,3}}}}_{{\color{col6}e_{0,3}}}\times \underbrace{e^{is_0\lambda_{k_{0,4}}}}_{{\color{col6}e_{0,4}}}\times \underbrace{e^{-is_0\lambda_{k_{0,5}}}}_{{\color{col6}e_{0,5}}}\\
&\qquad \qquad \times  \underbrace{A\bigl(\tfrac{k_{0,1}}N\bigr)}_{{\color{col5}v_{0,1}}}\times  \underbrace{\overline{A\bigl(\tfrac{k_{0,2}}N\bigr)}}_{{\color{col5}v_{0,2}}}\times  \underbrace{A\bigl(\tfrac{k_{0,3}}N\bigr)}_{{\color{col5}v_{0,3}}}\times  \underbrace{\overline{A\bigl(\tfrac{k_{0,4}}N\bigr)}}_{{\color{col5}v_{0,4}}}\times  \underbrace{A\bigl(\tfrac{k_{0,5}}N\bigr)}_{{\color{col5}v_{0,5}}},
\end{align*}
and for the \(\ell = (3,1)\) diagram is
\begin{align*}
&\underbrace{\bbo_{\{k = k_{2,1}\}}\delta(s_0 + s_1 + s_2 - t)}_{{\color{col5}v_{3,1}}}\times \underbrace{e^{-is_2\lambda_{k_{2,1}}}}_{{\color{col6}e_{2,1}}}\times \underbrace{-i\tfrac{\mu^2}N\gamma_{2,1}}_{{\color{col5}v_{2,1}}}  \times \underbrace{e^{-is_1\lambda_{k_{1,1}}}}_{{\color{col6}e_{1,1}}}\times \underbrace{e^{is_1\lambda_{k_{1,2}}}}_{{\color{col6}e_{1,2}}}\times \underbrace{e^{-is_1\lambda_{k_{1,3}}}}_{{\color{col6}e_{1,3}}}\\
&\qquad\times \underbrace{\bbo_{\{k_{0,1} = k_{1,1}\}}}_{{\color{col5}v_{1,1}}}\times\underbrace{\bbo_{\{k_{0,2} = k_{1,2}\}}}_{{\color{col5}v_{1,2}}}\times \underbrace{-i\tfrac{\mu^2}N\gamma_{1,3}}_{{\color{col5}v_{1,3}}} \times \underbrace{e^{-is_0\lambda_{k_{0,1}}}}_{{\color{col6}e_{0,1}}}\times \underbrace{e^{is_0\lambda_{k_{0,2}}}}_{{\color{col6}e_{0,2}}}\times \underbrace{e^{-is_0\lambda_{k_{0,3}}}}_{{\color{col6}e_{0,3}}}\times \underbrace{e^{is_0\lambda_{k_{0,4}}}}_{{\color{col6}e_{0,4}}}\times \underbrace{e^{-is_0\lambda_{k_{0,5}}}}_{{\color{col6}e_{0,5}}}\\
&\qquad \qquad \times  \underbrace{A\bigl(\tfrac{k_{0,1}}N\bigr)}_{{\color{col5}v_{0,1}}}\times  \underbrace{\overline{A\bigl(\tfrac{k_{0,2}}N\bigr)}}_{{\color{col5}v_{0,2}}}\times  \underbrace{A\bigl(\tfrac{k_{0,3}}N\bigr)}_{{\color{col5}v_{0,3}}}\times  \underbrace{\overline{A\bigl(\tfrac{k_{0,4}}N\bigr)}}_{{\color{col5}v_{0,4}}}\times  \underbrace{A\bigl(\tfrac{k_{0,5}}N\bigr)}_{{\color{col5}v_{0,5}}}.
\end{align*}
An expression for \(a_k\sbrack 2(t)\) is obtained by summing these expressions, then summing over \(\cK\) and integrating over \(s\).
\end{ex}

It will be useful to derive an explicit expression for the Duhamel iterates. Here, we have the following:

\begin{lemma}\label{l:Profile} For \(m\geq 1\), denote the profile
\eq{Profile}{
f\sbrack m(t) = e^{it\Lambda}a\sbrack m(t).
}
Then, for any \(T>0\) we have the expression
\begin{align}
&f_k^{(m)}(t)\bbo_{\R_\pm}(t)\label{Simba}\\
&\qquad= \frac{ie^{\pm t/T}}{2\pi}\left(\frac{\mu^2}{N} \right)^{m} \sum_{\ell,\mathcal{K}}\Biggl\{\Delta(k,\cK,\ell)\, \bA(\cK)\, \prod_{r=1}^{m}\gamma_{r,\ell_r}(\cK) \int_\R \prod_{r=0}^m \frac1{\alpha + \omega_r(\cK,\ell) \pm \frac{i}{T} } e^{-i\alpha t} \,d\alpha\Biggr\},\notag
\end{align}
valid for all \(t\in \R\), where \(\gamma_{r,\ell_r}(\cK)\) is defined as in \eqref{Momm} and:
\begin{itemize}[topsep=0pt, itemsep=\parskip]
\item The linear interactions
\eq{Pairing}{
\Delta(k,\cK,\ell) = \bbo_{\{k = k_{m,1}\}}\prod_{r=0}^{m-1}\left[\prod_{j=1}^{\ell_{r+1}-1}\bbo_{\{k_{r,j} = k_{r+1,j}\}}\prod_{j=\ell_{r+1}+3}^{2(m-r)+1}\bbo_{\{k_{r,j} = k_{r+1,j-2}\}}\right];
}
\item The modulations
\eq{thetaj}{
\omega_r(\cK,\ell) = \sum_{j=r+1}^m(-1)^{\ell_j-1}\Omega\bigl(k_{j,\ell_j},k_{j-1,\ell_j},  k_{j-1,\ell_j+1},k_{j-1,\ell_j+2}\bigr)\qtq{for}0\leq r\leq m-1,
}
with \(\omega_m(\cK,\ell) = 0\) and \(\Omega(k,\ell,m,n)\) defined as in \eqref{Osc};
\item The initial data
\eq{DataParity}{
\bA(\cK) = A\bigl(\tfrac{k_{0,1}}N\bigr)\,\overline{A\bigl(\tfrac{k_{0,2}}N\bigr)}\,\dots\,\overline{A\bigl(\tfrac{k_{0,2m}}N\bigr)}\,A\bigl(\tfrac{k_{0,2m+1}}N\bigr).
}
\end{itemize}
\end{lemma}
\begin{proof}
First consider the case that \(t>0\). Writing \eqref{Duhamel} as
\[
f_k(t) = A\left( \tfrac{k}{N}\right) - i\frac{\mu^2}N\sum_{\ell,m,n}\int_0^te^{is\Omega(k,\ell,m,n)}\gamma(k,\ell,m,n)f_\ell(s)\overline{f_m(s)}f_n(s)\,ds,
\]
the identity \eqref{Simba} follows from induction on \(m\) and the identity
\eq{Resolvent}{
\int_{\R^{m+1}_+}\left[\prod_{r=0}^me^{is_r\omega_r}\right]\,\delta\left(\sum_{r=0}^ms_r - t\right)\,ds_0\dots ds_r = \frac{e^{t/T}}{2\pi}i^{m+1}\int_\R \prod_{r=0}^m \frac1{\alpha + \omega_r + \frac{i}{T} } e^{-i\alpha t} \,d\alpha,
}
which is proved in~\cite[Lemma~4.2]{CG1} by writing
\[
\delta\left(\sum_{r=0}^ms_r - t\right) = e^{-\tfrac1T\left(\sum_{r=0}^ms_r - t\right)}\frac1{2\pi}\int_\R e^{i\alpha\left(\sum_{r=0}^ms_r - t\right)}\,d\alpha
\]
and then integrating in \(s_0,\dots,s_r\). We also note that
\[
\int_\R \prod_{r=0}^m \frac1{\alpha + \omega_r + \frac{i}{T} } e^{-i\alpha t} \,d\alpha
\]
vanishes for \(t<0\) which shows that \eqref{Simba} is valid for all times \(t\in \R\).

The case \(t<0\) now follows from the case \(t>0\) and the observation that \(\widetilde a(t) := -\overline{a(-t)}\) satisfies the equation
\[
i\frac d{dt}\widetilde a - \Lambda \widetilde a = \frac{\mu^2}N\Bigl[ \overline\Psi^*\bigl(\overline \Psi \widetilde a\odot\overline{\overline\Psi \widetilde a}\odot \overline\Psi \widetilde a\bigr) - \tfrac{2N\cM}{2N+1}\widetilde a\Bigr].
\]
\end{proof}

\subsection{The \(X^b_T\) spaces} To bound our solution, we introduce a modification of Bourgain's \(X^{s,b}\) spaces that are well-adapted to our problem. Using the unitary normalization of the Fourier transform,
$$
\widehat{f}(\tau) = \tfrac{1}{\sqrt{2\pi}} \int_\R e^{-it\tau} f(t) \,dt,
$$
for a diagonal matrix \(\Lambda\in M_{2N+1}(\R)\), $b \in \mathbb{R}$, and $T \geq 1$, we define the space \(X^b_T\) with norm
$$
\| a \|_{X^b_T}^2 := \left\| \left( \tfrac{1}{T} + |\tau| \right)^b \widehat{e^{it \Lambda} a} (\tau) \right\|_{L^2_\tau \ell^2}^2 = \sum_{k=-N}^N\left\| \left( \tfrac{1}{T} + |\tau| \right)^b \widehat a_k (\tau-\lambda_k) \right\|_{L^2_\tau }^2.
$$

The following properties follow directly from the definition, see, e.g.~\cite{MR2233925}:
\begin{lemma} Let \(\Lambda\in M_{2N+1}(\R)\) be a diagonal matrix, $b \in \mathbb{R}$, $T \geq 1$, and \(\chi_T(t) = \chi(\tfrac tT)\) for some \(\chi\in \Schwartz(\R)\). Then we have the following estimates:
\begin{itemize}
\item[(i)] (Free solution) If \(f\in \C^{2N+1}\) then
\eq{free bird}{
\left\| \chi_T\, e^{-it\Lambda} f \right\|_{X^b_T} \lesssim_{b,\chi} T^{\frac{1}{2}-b} \| f \|_{\ell^2}.
}
\item[(ii)] (Time restriction) If \(a\in X^b_T\) then
\eq{Truncation}{
\left\| \chi_T\, a \right\|_{X^b_T} \lesssim_{b,\chi}  \| a \|_{X_T^b}.
}
\item[(iii)] (Time continuity) If \(b>\frac12\) and \(a\in X^b_T\) then \(\chi_T\,a\in \Cont(\R;\C^{2N+1})\) and
\eq{Cont}{
\| \chi_T\,a \|_{L^\infty_t \ell^2} \lesssim_{b,\chi} T^{b-\frac{1}{2}} \| a \|_{X^b_T}
}
\item[(iv)] (Hyperbolic regularity) If $a\in X^b_T$ is a solution of
$$
\pde{
\displaystyle
i \frac{d}{dt} a - \Lambda a = F,}{
a(0) = 0,
}
$$
where \(F\in X^{b-1}_T\) then
\begin{equation}
\label{hyperbolicregularity}
\left\| \chi_T\, a \right\|_{X^b_T} \lesssim_{b,\chi} \| F \|_{X^{b-1}_T}.
\end{equation}
\end{itemize}
\end{lemma}

\subsection{The main estimates} We are now in a position to state our main estimates.

We subsequently assume that the initial data \(A\in \Cont^1[-1,1]\) has been fixed and ignore the dependence of the implicit constants on \(A\). We take \(\chi\in \Test(\R)\) to be an even, real-valued bump function that is identically \(1\) on \([-1,1]\) and supported on \([-2,2]\). We denote the rescaling \(\chi_T(t) = \chi(\frac tT)\) and similarly ignore the dependence of the implicit constants on \(\chi\).

It will be useful to introduce the following definition from~\cite{MR3068390}:
\begin{definition}[Stochastic domination]\label{d:SD}
For random variables \(X,Y\) depending on \(N\) we write \[X\prec Y\] if, for any \(K,\delta>0\), we have
\[
\bbP[X>N^\delta Y] = \bigO_{\delta,K}\bigl(N^{-K}\bigr)\quad\text{as}\quad N\rightarrow\infty,
\]
i.e., for any \(\delta>0\) and sufficiently large \(N\gg1\) we have \(X\leq N^\delta Y\) with overwhelming probability. If the implicit constant depends on some deterministic parameter \(\alpha\), we indicate this by writing \(X\prec_\alpha Y\). Finally, we write \(X = \bigO_\prec(Y)\) if \(|X|\prec Y\).
\end{definition}
We remark that the relation \(\prec\) is transitive and behaves like an inequality with respect to arithmetic operations; see~\cite[Lemma~4.4]{MR3068390}.

Our first estimate considers the leading order terms in the approximation: \(a\sbrack 0\), \(a\sbrack 1\), and \(a\sbrack 2\). Here, we have the following estimate, which is proved in Section~\ref{s:LOT}:

\begin{proposition}[The leading order terms]\label{p:LOT}
For any \(\delta>0\) and \(1\leq t\leq N\) we have
\begin{align}\label{LOT}
&\tfrac1N\left\|\bbE\bigl|a_k\sbrack 0(t) + a_k\sbrack 1(t)\bigr|^2 + 2\Re\bbE\Bigl[\overline{a_k\sbrack0(t)}a_k\sbrack 2(t)\Bigr] - |A(\tfrac kN)|^2 - \tfrac t{T_{\kin}}\cC\bigl[|A|^2\bigr](\tfrac kN)\right\|_{\ell_k^1}\\
&\qquad\lesssim_\delta\tfrac{t}{T_\kin} \left(\tfrac1{\sqrt t} + N^\delta\tfrac tN\right).\notag
\end{align}
\end{proposition}

Our second estimate is proved in Section~\ref{s:app}. This gives bounds for \(a\sbrack m\) that will primarily be used to control higher order error terms:

\begin{proposition}[The Duhamel iterates]\label{p:app} For all \(m\geq 1\), \(\frac12<b<\frac12+m\), and \(1\leq T\leq N\) we have the estimate
\eq{AppXb}{
\tfrac1N\|a\sbrack m\|_{X^b_T}^2\prec_{b,m} \left(\tfrac{T^{4/3}}{T_{\kin}}\right)^m.
}
Further, for all \(\delta>0\) we have the refinements
\begin{align}
\tfrac1N\bbE\|a\sbrack 1\|_{X^b_T}^2 &\lesssim_{b,\delta}N^\delta \tfrac T{T_{\kin}},\label{AppXb-1}\\
\tfrac1N\bbE\|a\sbrack 2\|_{X^b_T}^2 &\lesssim_{b,\delta}N^\delta \left(\tfrac T{T_{\kin}}\right)^2.\label{AppXb-2}
\end{align}

\end{proposition}

The linear term is bounded in Section~\ref{s:lin}, relying on similar estimates to the proof of Proposition~\ref{p:app}:

\begin{proposition}[The linear operator]\label{p:Lin}
Let \(M\geq 0\) and \(\cL\) be defined as in \eqref{Linearization}. Then, if \(\frac12<b\leq \frac56\) and \(1\leq T\leq \min\{N,T_\kin\}\) we have the estimate
\eq{Lin}{
\|\chi_T\,\cL\|_{X_T^b\to X^{b-1}_T}\prec_{M,b} (TN)^{2b-1}\left(\tfrac{T^{4/3}}{T_{\kin}}\right)^{\frac12}.
}
\end{proposition}

For the nonlinear terms we use the following (deterministic) bound:

\begin{lemma}[The nonlinear terms]\label{l:Nonlinear}
For \(\frac12<b<1\) we have the estimate
\eq{NL}{
\|\chi_T\,\cN[u,v,w]\|_{X^{b-1}_T}\lesssim_b T^{2b} \tfrac{\mu^2}N \|u\|_{X^b_T}\|v\|_{X^b_T}\|w\|_{X^b_T}.
}
\end{lemma}
\bpf
We first use the embedding \(\ell^2\subset \ell^4\) with the fact that \(\|\Psi e^{-it\Lambda}\|_{\ell^2\to\ell^2} = 1\) to estimate
\eq{Strichartz}{
\|\chi_T\Psi e^{-it\Lambda}f\|_{L^4_t\ell^4}^4\lesssim T\|f\|_{\ell^2}^4\qtq{for any}f\in \C^{2N+1}.
}
Taking the inverse Fourier transform, we have
\[
u_k(t) = \tfrac1{\sqrt{2\pi}}\int_\R e^{it\tau}e^{-it\lambda_k}\widehat u_k(\tau - \lambda_k)\,d\tau.
\]
Setting \(g_k(\tau) = \widehat u_k(\tau - \lambda_k)\) and applying \eqref{Strichartz} followed by the Cauchy--Schwarz inequality gives us
\[
\|\chi_T\,\Psi u\|_{L^4_t\ell^4} \lesssim \int_\R \|\chi_T \Psi e^{-it\Lambda}g(\tau)\|_{L^4_t\ell^4}\,d\tau\lesssim T^{\frac14} \int_\R\|g\|_{\ell^2}\,d\tau\lesssim_b T^{b-\frac14}\|u\|_{X^b_T}.
\]
By duality, for any \(b'<-\frac12\) and \(h\in \C^{2N+1}\) we have the bound
\[
\|\chi_T\,\Psi^*h\|_{X^{b'}_T}\lesssim_{b'} T^{-b'-\frac14}\|h\|_{L^{4/3}_t\ell^{4/3}}.
\]
Combining the above bounds and using that \(\chi_T = \chi_T\chi_{2T}\), we obtain
\begin{align*}
\|\chi_T\,\Psi^*(\Psi u\odot\overline{\Psi v}\odot\Psi w)\|_{X^{b'}_T}&\lesssim_{b'} T^{-b'-\frac14}\|\chi_{2T}\,\Psi u\|_{L^4_t\ell^4} \|\chi_{2T}\,\Psi v\|_{L^4_t\ell^4} \|\chi_{2T}\,\Psi w\|_{L^4_t\ell^4}\\
&\lesssim_{b,b'} T^{3b-b'-1}\|u\|_{X^b_T}\|v\|_{X^b_T}\|w\|_{X^b_T}.
\end{align*}
As \(X_T^0 = L^2_t\ell^2\), we may use the embedding \(\ell^2\subset \ell^\infty\) and \eqref{Cont} to yield the bound
\begin{align*}
\|\chi_T\,\Psi^*(\Psi u\odot\overline{\Psi v}\odot\Psi w)\|_{X^0_T} &\lesssim \|\chi_{2T}\,\Psi u\|_{L^2_t\ell^2}\|\chi_{2T}\, \Psi v\|_{L^\infty_t\ell^\infty}\|\chi_{2T}\, \Psi w\|_{L^\infty_t\ell^\infty}\\
&\lesssim \|\chi_{2T}\,u\|_{L^2_t\ell^2}\|\chi_{2T}\,v\|_{L^\infty_t\ell^2}\|\chi_{2T}\,w\|_{L^\infty_t\ell^2}\\
&\lesssim_b T^{3b-1}\|u\|_{X^b_T}\|v\|_{X^b_T}\|w\|_{X^b_T}.
\end{align*}
We then interpolate between these two estimates to get
\[
\|\chi_T\,\Psi^*(\Psi u\odot\overline{\Psi v}\odot\Psi w)\|_{X^{b-1}_T}\lesssim_b T^{2b}\|u\|_{X^b_T}\|v\|_{X^b_T}\|w\|_{X^b_T}.
\]
To complete the proof of \eqref{NL}, we apply \eqref{Truncation} and \eqref{Cont} to estimate
\begin{align*}
\left\|\chi_T\,\<v,u\>w\right\|_{X_T^{b-1}} &\lesssim \| \chi_{2T}\,u\|_{L^\infty_t\ell^2}\|\chi_{2T}\, v\|_{L^\infty_t\ell^2}\|\chi_{2T}\, w\|_{X_T^{b-1}}\lesssim_b T^{2b} \|u\|_{X_T^b}\|v\|_{X_T^b}\|w\|_{X_T^b}.
\end{align*}
The remaining terms in \(\cN[u,v,w]\) are bounded using symmetric estimates.
\epf

\subsection{Proof of Theorem~\ref{t:main}} We now complete the proof of Theorem~\ref{t:main}. Let \(\frac14<\beta<\frac12\) and \(\epsilon>0\) be fixed as in the hypothesis, and choose some fixed
\[
\tfrac12<b<\min\bigl\{1-\beta,\tfrac12 + \tfrac\epsilon{24}\bigr\}.
\]
We subsequently ignore the dependence of the implicit constants on \(\beta,\epsilon,b\).

Denoting the sample space by \(\Omega\), for each \(\omega\in \Omega\), let \(\Lambda = \Lambda(\omega)\), \(\Psi = \Psi(\omega)\), and take \(a = a(\omega)\in \Cont(\R;\C^{2N+1})\) to be the corresponding (unique, global) solution of \eqref{NLS-WO}. Let \(M = \lceil\frac94+\frac6\epsilon\rceil\) and define \(a^\app\), \(a^\err\) as in \eqref{app-sum},~\eqref{Decomp}, respectively.

We first prove an estimate for the remainder term \(a^\err\). Given \(N\geq 1\) and \(1\leq T\leq N\), let \(E_N\subseteq \Omega\) be the event that
\begin{align}
\tfrac 1N\|a\sbrack m\|_{X_T^b}^2 \leq N^{\frac\epsilon3} \left(\tfrac{T^{4/3}}{T_{\kin}}\right)^m\text{ for all \(m\in \llbracket1,M\rrbracket\)}\qtq{and}
\|\chi_T\cL\|_{X_T^b\to X_T^{b-1}} \leq N^{\frac\epsilon6} \left(\tfrac{T^{4/3}}{T_{\kin}}\right)^{\frac12}.\label{d2}
\end{align}
Applying Propositions~\ref{p:app} and~\ref{p:Lin}, where we use that \(T\leq N\) and \(4b-2<\epsilon/6\) in the latter case, we see that \eqref{d2} holds with overwhelming probability. In particular, for any \(K>0\) we have
\eq{Overwhelmed}{
\bbP\bigl[E_N^c\bigr] \lesssim_K N^{-K},
}
provided \(N\gg1\) is sufficiently large. Let us also note that, as \(T\leq N\) and \(\mu = N^{\beta}\), Lemma~\ref{l:Nonlinear} and our choice of \(b\) ensure that
\begin{equation}
\|\chi_T \cN[u,v,w]\|_{X_T^{b-1}}\leq N \|u\|_{X_T^b}\|v\|_{X_T^b}\|w\|_{X_T^b},\label{d3}
\end{equation}
provided \(N\gg1\) is sufficiently large.

We first consider the contribution of \(E_N\). If \(\omega\in E_N\) and \(1\leq T\leq N^{-\epsilon}T_{\kin}^{3/4}\), we define the operator
\begin{align*}
&\Upsilon_T(a)(t)\\
&:=\!-i\chi_T(t)\!\int_0^t\!\!\!e^{-i(t-s)\Lambda}\chi_T(s)\Bigl\{\cL a(s)\!+\!3\cN[a^\app(s),a(s),a(s)]\!+\!\cN[a(s),a(s),a(s)]\!+\!\cE(s)\Bigr\}\,ds.
\end{align*}
We claim that for sufficiently large \(N\gg1\) the map \(\Upsilon_T\) is a contraction on the ball
\[
\cB_T = \Bigl\{a\in X_T^b:\|a\|_{X_T^b}\leq N^{-\frac32-\frac\epsilon2}\Bigr\}.
\]

To prove this, first apply \eqref{d2} to bound
\[
\|\chi_T\,\cL a\|_{X_T^{b-1}}\leq N^{-\frac\epsilon2} \|a\|_{X_T^b},
\]
where we note that our hypothesis on \(T\) ensures that \(\frac{T^{4/3}}{T_{\kin}}\leq T^{-\frac 43\epsilon}\). Second, by applying \eqref{free bird} to get
\eq{a0 opt}{
\|\chi_{2T}\,a\sbrack 0\|_{X_T^b}\lesssim N^{\frac12},
}
and \eqref{d2} with \eqref{Truncation} to bound
\eq{am opt}{
\|\chi_{2T}\,a\sbrack m\|_{X_T^b}\leq N^{\frac12+\frac{1-4m}6\epsilon}\qtq{for}m\in \llbracket 1,M\rrbracket,
}
we obtain the estimate
\[
\|\chi_{2T}\,a^\app\|_{X_T^b}\lesssim N^{\frac12}.
\]
As a consequence, the inequality \eqref{d3} gives us
\begin{align*}
\|\chi_T\,\cN[a^\app,a,a]\|_{X_T^{b-1}} \lesssim N^{\frac32}\|a\|_{X_T^b}^2\qtq{and}\|\chi_T\,\cN[a,a,a]\|_{X_T^{b-1}} \lesssim N\|a\|_{X_T^b}^3.
\end{align*}
Third, the error \(\cE\) defined in \eqref{Error} can be written as
\[
\cE = \sum_{(m_1,m_2,m_3)\in S_M}\cN\bigl[a\sbrack{m_1},a\sbrack{m_2},a\sbrack{m_3}\bigr],
\]
where
\[
S_M = \bigl\{(m_1,m_2,m_3)\in \llbracket 0,M\rrbracket^3 : m_1+m_2+m_3\geq M\bigr\}.
\]
Using that \(\frac52 + \frac\epsilon 2-\frac23\epsilon M\leq -\frac32-\epsilon\), we may apply \eqref{d3}, \eqref{a0 opt}, and \eqref{am opt} to bound
\[
\|\chi_T \,\cE\|_{X_T^{b-1}}\lesssim N^{-\frac32 - \epsilon}.
\]
Combining these estimates with \eqref{hyperbolicregularity}, for \(a\in \cB_T\) we have
\[
\|\Upsilon_T(a)\|_{X_T^b} \lesssim N^{-\frac 32-\epsilon}.
\]
Applying similar estimates for the difference of two solutions, for all \(a,\widetilde a\in \cB_T\) we have
\[
\|\Upsilon_T(a) - \Upsilon_T(\widetilde a)\|_{X_T^b}\lesssim N^{-\frac \epsilon2}\|a - \widetilde a\|_{X_T^b}.
\]
As a consequence, for all sufficiently large \(N\gg1\) the map \(\Upsilon_T\) is a contraction on the ball \(\cB_T\), as claimed.

We now apply Banach's Fixed Point Theorem to show that for sufficiently large \(N\gg1\), each \(\omega\in E_N\),  and each \(1\leq T\leq N^{-\epsilon}T_\kin^{3/4}\) there exists a unique fixed point \(a^\err_T\in \cB_T\) of \(\Upsilon_T\). By uniqueness of solutions to the ODE \eqref{NLS}, we see that \(a^\err(t) = a_T^\err(t)\) for all \(|t|\leq T\). Consequently, recalling our choice of \(b\), the estimate \eqref{Cont} shows that
\eq{aerr-1}{
\tfrac1N\bbE\left[\bbo_{E_N}\|a^\err(t)\|_{\ell^2}^2\right]\leq N^{-4}\qtq{for all}|t|\leq N^{-\epsilon}T_{\kin}^{3/4}.
}

To control the contribution of the event \(E_N^c\), we first use the embedding \(\ell^2\subset\ell^\infty\) to bound
\[
\|\cN[a,b,c]\|_{\ell^2}\lesssim \tfrac{\mu^2}N\|a\|_{\ell^2}\|b\|_{\ell^2}\|c\|_{\ell^2}.
\]
Noting that \(\|a\sbrack 0(t)\|_{\ell^2}\lesssim \sqrt N\), we may then use the definition \eqref{inductive am+1} and induction on \(m\geq 0\) to obtain the crude estimate
\[
\|a\sbrack m(t)\|_{\ell^2}\lesssim_m (\mu t)^m\sqrt N.
\]
Combining this with the conservation of mass, which gives us the bound \(\|a(t)\|_{\ell^2}\lesssim \sqrt N\), we have
\[
\|a^{\err}(t)\|_{\ell^2}\lesssim \<\mu^2 t\>^M\sqrt N.
\]
In particular, with \(\mu = N^\beta\), the estimate \eqref{Overwhelmed} gives us the bound
\[
\tfrac1N\bbE\left[\bbo_{E_N^c}\|a^\err(t)\|_{\ell^2}^2\right]\lesssim_K N^{-K}\qtq{for all}|t|\leq N^{-\epsilon} T_\kin^{3/4}\qtq{and any}K>0.
\]
After possibly increasing the size of \(N\gg1 \), we may combine this bound with \eqref{aerr-1} to get
\eq{aerr}{
\tfrac1N\bbE\|a^\err(t)\|_{\ell^2}^2\leq N^{-4}\qtq{for all}|t|\leq N^{-\epsilon}T_{\kin}^{3/4},
}
i.e., we no longer need to condition on the event \(E_N\).

Let us now take \(N^\epsilon\leq t\leq N^{-\epsilon}T_{\kin}^{3/4}\). From Proposition~\ref{p:app} and \eqref{Cont}, for any \(\delta>0\) we may bound
\eq{mgeq3}{
\sum_{m=3}^M \tfrac1N\bbE \|a\sbrack m(t)\|_{\ell^2}^2\lesssim_\delta N^\delta\frac{t^4}{T_{\kin}^3}.
}
Similarly, applying Propositions~\ref{p:LOT},~\ref{p:app}, and \eqref{Cont}, for any \(\delta>0\) we have
\eq{LOB}{
\frac1N\left\|\bbE\bigl|a_k\sbrack 0(t) + a_k\sbrack 1(t) + a_k\sbrack 2(t)\bigr|^2 - \bigl|A\bigl(\tfrac kN\bigr)\bigr|^2 - \tfrac t{T_{\kin}}\cC\bigl[|A|^2\bigr]\bigl(\tfrac kN\bigr)\right\|_{\ell_k^1}\lesssim_{\delta} \frac t{T_{\kin}}N^{\delta-\frac\epsilon2}.
}
Combining the estimates \eqref{aerr}, \eqref{mgeq3}, and \eqref{LOB} with the fact that \(\frac t{T_{\kin}}\geq \frac1N\), we obtain
\[
\LHS{R}\lesssim_{\delta}\frac{t}{T_\kin}\left[N^{\delta-\frac\epsilon2} + N^\delta\frac{t^3}{T_{\kin}^2} + N^{-3}\right].
\]
The estimate \eqref{R} now follows from taking \(\delta = \epsilon/4\).\qed

\section{Random Matrix estimates}\label{s:WG}

In this section we establish estimates concerning the random matrix \(H\). A property of the Gaussian Unitary Ensemble (GUE) is that its eigenvectors and eigenvalues form two independent systems, which we can study separately. The vector of eigenvalues, \(\lambda\), is distributed according to \eqref{Beta_ens}. The matrix of eigenvectors, \(\Psi = (\psi_{jk})\), is distributed according to the Haar measure on \(\bbU(2N+1)\), also known as the circular unitary ensemble ($\rm{CUE}$). Section \ref{s:Wgcalc} presents a method to deal with general expressions involving entries of a ${\rm{CUE}}(d)$ matrix, and can be read independently from the rest of the paper. This method is then applied to particular expressions involving entries of $\Psi$ in Section \ref{s:application_Wg}. Finally, Section \ref{s:LSC} presents the application of the local semicircle law to our estimates involving the eigenvalues of $H$.

\subsection{Weingarten calculus}\label{s:Wgcalc}
Given two integers \(q, d \geq 1\), the expectation of a product of $2q$ terms from a Haar-distributed unitary matrix of size $d \times d$ is given explicitly by the following result of Collins~\cite[Theorem~2.1]{MR1959915} in the so called \textit{stable range} $d \geq q$:

\begin{theorem}[Fundamental Theorem of Weingarten calculus]\label{t:WG} If $\Psi = (\psi_{jk})_{j,k=1}^d$ is a Haar-distributed unitary matrix, then for any choice of indices $(j_l, k_l, j_l', k_l')_{l=1}^q$ with $1 \leq q \leq d$,
\begin{multline}\label{Weingarten}
\bbE\left[ \psi_{j_1k_1} \cdots \psi_{j_qk_q}  \
    \overline{\psi_{j_1'k_1'}} \cdots  \overline{\psi_{j_q'k_q'}} \right]
    = \!
    \sum_{\sigma,\tau \in \mathfrak{S}_q} 
    \delta_{j_1,j_{\tau(1)}'} \! \cdots \delta_{j_q,j_{\tau(q)}'} \
    \delta_{k_1,k_{\sigma(1)}'} \! \cdots \delta_{k_q,k_{\sigma(q)}'}
    \Wg(\sigma \tau^{-1}, d).
\end{multline}
\end{theorem}
The rational function \(\Wg(\ccn, \cdot )\) depends on the cycle structure of the permutation \(\ccn \in \mf S_q\), and thus also implicitly on $q$; we refer the reader to \cite{MR1959915, Collins2021} for a definition. Of particular interest is the leading order term, which is separated from lower order terms by a $d^{-2}$ gap. Indeed,
\eq{WgAsy}{
    \mathrm{Wg}(\ccn, d) = d^{ -2q + \mathscr{C}(\ccn)}
    \Moeb (\ccn)
    \left( 1 + \bigO (d^{-2})\right)
}
where \(\mathscr{C}(\ccn)\) denotes the number of cycles\footnote{Note that equation \eqref{WgAsy} is equivalent to~\cite[Corollary~2.7]{MR2217291} with the relation $ |\ccn| = q - \mathscr{C}(\ccn) $ between $\mathscr{C}(\ccn)$ and \(|\ccn|\), the number of transpositions needed to generate $\ccn$.} of \(\ccn\), and the M\oe bius function is defined as
\eq{Moeb_function}{
{\Moeb} (\ccn) :=
\prod_{i=1}^{\mathscr{C}(\ccn)} (-1)^{C_i-1}\frac{(2C_i-2)!}{C_i!(C_i-1)!},
}
with $C_1, \dots, C_{\mathscr{C}(\ccn)}$ the respective sizes of the cycles of $\ccn$. In particular, the case $\sigma=\tau$ in \eqref{Weingarten} corresponds to an exact pairing of edges, and yields a contribution of the same leading order as one would obtain from Wick's formula when considering Gaussian variables of variance $1/d$, namely:
\eq{WgID}{
\Wg({\mr{Id}}_{\mathfrak{S}_q},d) = d^{-q} + \bigO( d^{-q-2}).
}
To facilitate the computation of general bounds, we introduce a structure of directed bipartite graphs that match the elements of Weingarten's formula \eqref{Weingarten}. For a term of the form 
\eq{Expression}{
\psi_{j_1k_1} \dots \psi_{j_qk_q}  \
    \overline{\psi_{j_1'k_1'}} \dots  \overline{\psi_{j_q'k_q'}},
}
we represent each distinct row index from the set \(\{j_1,\dots,j_q,j_1',\dots,j_q'\}\) as a black vertex \tikz[baseline=-2.4,scale=0.15]{\node[dot] {};} and each distinct column index from the set \(\{k_1,\dots,k_q,k_1',\dots,k_q'\}\) as a white vertex \tikz[baseline=-2.4,scale=0.15]{\node[ghostdot] {};}. The random variables \(\psi_{j_rk_r}\) are represented as edges directed from the black vertex \(j_r\) to the white vertex \(k_r\), whereas the conjugated random variables \(\overline{\psi_{j_r'k_r'}}\) are represented as edges directed from the white vertex \(k_r'\) to the black vertex \(j_r'\). We refer to the corresponding directed graph $G$ as the \emph{Weingarten graph}\footnote{This is not to be confused with the different notion of Weingarten graphs used in \cite{CollinsMatsumoto2017}, where vertices represent permutations.} of the expression \eqref{Expression}. Figure \ref{fig:Wg_graph_ex} gives three basic examples. Reciprocally, we will use the shorthand $\psi_G$ to denote a term such as \eqref{Expression}, where the structure of the indices is encoded in the graph $G$. 

\begin{figure}[t!]
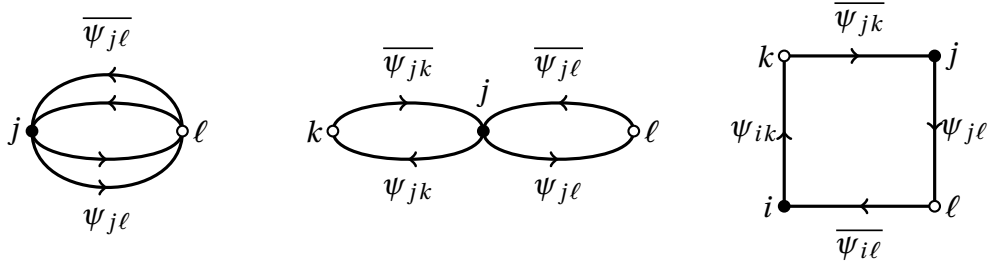

\tikz{
    \begin{scope}[shift={(-6,0)}]
\draw[very thick,->-] (0,0) .. controls (0,-1) and (2,-1) .. (2,0);
\draw[very thick,->-] (2,0) .. controls (2,1) and (0,1) .. (0,0);
\draw[very thick,->-] (0,0) .. controls (0,-.5) and (2,-.5) .. (2,0);
\draw[very thick,->-] (2,0) .. controls (2,.5) and (0,.5) .. (0,0);
\node[ghostdot] at (2,0) {};
\node[right] at (2,0) {\(\ell\)};
\node[dot] at (0,0) {};
\node[left] at (0,0) {\(j\)};
\node at (1,-1.2) {\(\psi_{j\ell}\)};
\node at (1,1.3) {\(\overline{\psi_{j\ell}}\)};
\end{scope}
\draw[very thick,->-] (0,0) .. controls (0,-.5) and (-2,-.5) .. (-2,0);
\draw[very thick,->-] (-2,0) .. controls (-2,.5) and (0,.5) .. (0,0);
\draw[very thick,->-] (0,0) .. controls (0,-.5) and (2,-.5) .. (2,0);
\draw[very thick,->-] (2,0) .. controls (2,.5) and (0,.5) .. (0,0);
\node[ghostdot] at (2,0) {};
\node[right] at (2,0) {\(\ell\)};
\node[ghostdot] at (-2,0) {};
\node[left] at (-2,0) {\(k\)};
\node[dot] at (0,0) {};
\node at (0,.5) {\(j\)};
\node at (1,-.8) {\(\psi_{j\ell}\)};
\node at (-1,-.8) {\(\psi_{jk}\)};
\node at (1,.9) {\(\overline{\psi_{j\ell}}\)};
\node at (-1,.9) {\(\overline{\psi_{jk}}\)};
\begin{scope}[shift={(4,-1)}]
\draw[very thick,->-] (0,0)--(0,2);
\draw[very thick,->-] (0,2)--(2,2);
\draw[very thick,->-] (2,2)--(2,0);
\draw[very thick,->-] (2,0)--(0,0);
\node[dot] at (0,0) {};
\node[left] at (0,0) {\(i\)};
\node[dot] at (2,2) {};
\node[right] at (2,2) {\(j\)};
\node[ghostdot] at (2,0) {};
\node[right] at (2,0) {\(\ell\)};
\node[ghostdot] at (0,2) {};
\node[left] at (0,2) {\(k\)};
\node at (-.4,1) {\(\psi_{ik}\)};
\node at (2.4,1) {\(\psi_{j \ell}\)};
\node at (1,2.5) {\(\overline{\psi_{jk}}\)};
\node at (1,-.5) {\(\overline{\psi_{i \ell}}\)};
\end{scope}
}
\vspace{-.1in}
\caption{Weingarten graphs associated to \( |\psi_{jl}|^4 \), \( |\psi_{jk}|^2 |\psi_{j\ell}|^2 \) and \( \psi_{ik} \overline{\psi_{jk}} \psi_{j\ell} \overline{\psi_{i \ell}} \), where \(i\neq j\) and \(k \neq \ell\). Note that identities between row and column vertices, such as $j=l$, do not appear on the graph -- nor do they make a difference to the corresponding computation.}
\label{fig:Wg_graph_ex}
\end{figure}

Given a Weingarten graph \(G\), we define a \emph{circuit covering} of \(G\) to be a partition of the edges of $G$ in directed circuits.  The reason for considering circuit coverings is that they are equivalent to the pairs of permutations $(\sigma,\tau) \in \mathfrak{S}_q^2$ that give a nonzero contribution in \eqref{Weingarten}.

\begin{fact}\label{fact:correspondence}
Let \(G\) be the Weingarten graph of \eqref{Expression}. Then, the pairs of permutations \( (\sigma,\tau) \in \mf S_q^2 \) meeting the admissibility conditions
\eq{admissible}{
 j_l = j_{\tau(l)}',
\qquad k_l = k_{\sigma(l)}'
\qquad \text{for all}
\quad l=1, \dots, q,
}
are in a natural one-to-one correspondence with the circuit coverings of \(G\), in such a way that each $n$-cycle of $\cc$ corresponds to a $2n$-circuit in $G$.
\end{fact}

\bpf
The bipartite structure of the Weingarten graph forces the $\psi$-edges and $\overline{\psi}$-edges of a circuit to alternate. As illustrated on Figure \ref{fig:Wg_covering_pairs}, $\sigma$ and $\tau$ encode the two kinds of steps: if $\overline{\psi_{j_r' k_r'}}$ follows $\psi_{j_l k_l}$ in the circuit covering, then the admissibility condition $k_l = k_r'$ is implied and the step is encoded by $\sigma(l) = r$; similarly, if $\psi_{j_m k_m}$ follows $\overline{\psi_{j_r' k_r'}}$, then the admissibility condition $j_m = j_r'$ is implied and is the step is encoded by $r=\tau(m)$. Reciprocally, the sequences of edges obtained by applying $\sigma$ and $\tau^{-1}$ alternatively forms a circuit covering if and only if the admissibility conditions are met. It follows by construction that each cycle of size $n$ in $\cc$ corresponds to a circuit of size $2n$ on the graph $G$.
\epf

\begin{figure}[t!]
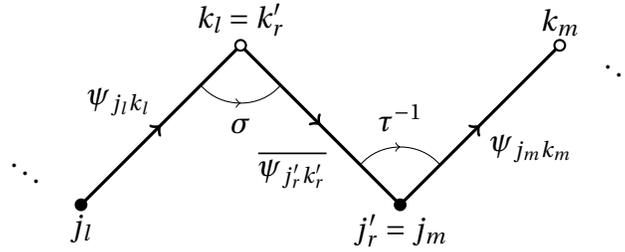

\begin{center}
\tikz[rotate=45,scale=3]{
\draw[very thick,->-] (-1,1)--(0,1);
\draw[very thick,->-] (0,1)--(0,0);
\draw[very thick,->-] (0,0)--(1,0);
\draw[->-] (-.25,1) to[bend right=50] (0,.75);
\node[below] at (-.2,.8) {$\sigma$};
\draw[->-] (0,.25) to[bend left=50] (.25,0);
\node[above] at (.2,.2) {$\tau^{-1}$};
\node at (-1.05,1.3) {$\ddots$};
\node[dot] at (-1,1) {};
\node[below] at (-1,1) {$j_l$};
\node[above left] at (-.5,1) {$\psi_{j_l k_l}$};
\node[ghostdot] at (0,1) {};
\node[above] at (0,1) {$k_l = k'_r$};
\node[below left] at (0,.4) {$\overline{\psi_{j_r' k_r'}}$};
\node[dot] at (0,0) {};
\node[below] at (0,0) {$j_r'=j_m$};
\node[below right] at (.5,0) {$\psi_{j_m k_m}$};
\node[ghostdot] at (1,0) {};
\node[above] at (1,0) {$k_m$};
\node at (1.12,-.25) {$\ddots$};
}
\end{center}
\vspace{-.1in}
\caption{Illustration of the correspondence between circuits and admissible pairs $\ccpair$: the edge $\psi_{j_l k_l}$ can be followed by $\overline{\psi_{j_r' k_r'}}$ in the circuit if and only if $k_l = k_{r'}$, and so on.}
\label{fig:Wg_covering_pairs}
\end{figure}

We call pairs $\ccpair \in \mathfrak{S}_q^2$ that verify the admissibility conditions \eqref{admissible} \textit{covering pairs} of $\psi_G$, and denote their set by $\CC(\psi_G)$. In light of Fact \ref{fact:correspondence}, we will simultaneously think of these as pairs of permutations and as circuit coverings of the graph $G$.

\begin{proposition}\label{p:WGBD}
Let \(G\) be a Weingarten graph with $E=2q$ edges, \(V\) vertices, and \(c\) connected components. Then, we have the general bounds
\eq{WGBD}{
\bbE\left[ \psi_G \right]  = \bigO_q \left( d^{-q} \right)
\qquad \text{and} \qquad
\bbE\left[ \psi_G \right]  = \bigO_q \left( d^{c-V} \right)
}
\end{proposition}

\bpf
It follows from Theorem \ref{t:WG}, \eqref{WgAsy} and Fact \ref{fact:correspondence} that the leading order is given by $d^{-2q+n}$ with $n$ being the maximal number of circuits in a circuit covering, or equivalently the maximal number of cycles of $\cc$ with $\ccpair \in \CC(\psi_G)$. This is at most $d^{-q}$, because $n = \mathscr{C}(\cc) \leq q$.
% or: because of the bipartite structure of the graph (each circuit has an even size).
We now prove that this is also bounded by $d^{c-V}$. In a connected graph, the inequality $V - E \leq 1$ holds, so that more generally we obtain
\begin{equation}\label{graph_Euler_eq1}
V - E \leq c
\end{equation}
by summing over the connected components. Now, consider the graph $\wt{G}$ obtained by removing one edge in each circuit. The fact that each edge appears in exactly one circuit ensures, by inspection, that $\wt{G}$ has $\wt{c}=c$ connected components, $\wt{V}=V$ vertices and $\wt{E}=E-n$ edges. Applying \eqref{graph_Euler_eq1} to $\wt{G}$,
\begin{equation}\label{graph_Euler_eq2}
    V - E + n = \wt{V}-\wt{E} \leq \wt{c} = c.
\end{equation}
This implies that $-2q+n \leq c-V$, which concludes the proof of the second bound.
\epf

Circuits of size $2$ (corresponding to fixed points of $\sigma \tau^{-1}$) play a particular role in maximizing the number of circuits in a circuit covering. We refer to such a circuit between vertex $a$ and vertex $b$ as the $2$-circuit $(a,b)$.  The next lemma simplifies the identification of the leading order by allowing us to choose $2$-circuits first, at no risk to the final estimate. Note that no such statement holds for circuits of size $4$ or more, so that the greedy strategy (i.e., choosing the smallest possible circuits first) is not optimal in general.

\begin{proposition}\label{banana_principle}
If the expression $|\psi_{ab}|^2 \psi_G$ has at least one covering pair, then it has a covering pair with largest possible number of circuits that includes the $2$-circuit $(a,b)$. 
Equivalently: if $\CC(|\psi_{ab}|^2 \psi_G) \neq \emptyset$, then
\begin{equation}
\max_{\ccpair \in \mathrm{\CC}(|\psi_{ab}|^2 \psi_G)} \mathscr{C}(\cc)
= 1+ \max_{\ccpair \in \mathrm{\CC}(\psi_G)} \mathscr{C} (\cc).
\end{equation}
\end{proposition}

\begin{proof}
As illustrated in Figure \ref{fig:banana}, if $\ccpair$ is a maximal covering pair of $|\psi_{ab}|^2 \psi_G$ which does not include the $2$-circuit $(a,b)$, then $\psi_{ab}$ is part of a circuit $(a,b,L_1)$ and $\overline{\psi_{ab}}$ is part of another circuit $(b,a,L_2)$. But then, considering instead the $2$-circuit $(a,b)$ and the concatenation of $L_1$ and $L_2$ yields a circuit-covering with same number of circuits (assumed maximal) that includes the $2$-circuit $(a,b)$.
\end{proof}

\begin{figure}[t!]
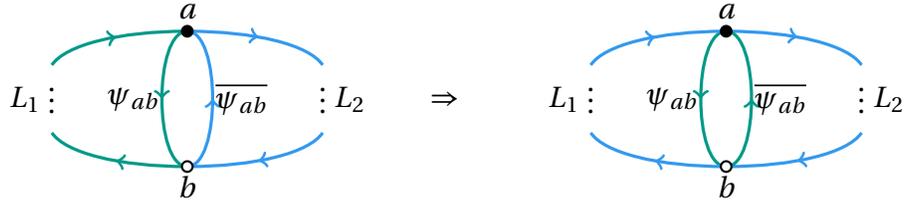

\tikz[rotate=-90,scale=.9]{
\draw[col3,very thick,->-] (0,0) .. controls (0,-.5) and (2,-.5) .. (2,0);
\draw[col5,very thick,->-] (0,0) .. controls (0,1) and (.2,1.8) .. (.5,2);
\draw[col5,very thick,->-] (1.5,2) .. controls (1.8,1.8) and (2,1) .. (2,0);
\draw[col5,very thick,->-] (2,0) .. controls (2,.5) and (0,.5) .. (0,0);
\draw[col3,very thick,->-] (.5,-2) .. controls (.2,-1.8) and (0,-1) .. (0,0);
\draw[col3,very thick,->-] (2,0) .. controls (2,-1) and (1.8,-1.8) .. (1.5,-2);
\node[ghostdot] at (2,0) {};
\node at (2.3,0) {\(b\)};
\node[dot] at (0,0) {};
\node at (-.3,0) {\(a\)};
\node at (1,-.8) {\(\psi_{ab}\)};
\node at (1,.8) {\(\overline{\psi_{ab}}\)};
\node at (.9,2) {\(\vdots\)};
\node at (1,2.4) {\(L_2\)};
\node at (.9,-2) {\(\vdots\)};
\node at (1,-2.4) {\(L_1\)};
\node at (1,3.8) {\(\Rightarrow\)}
}
\hspace{.3in}
\tikz[rotate=-90,scale=.9]{
\draw[col3,very thick,->-] (0,0) .. controls (0,-.5) and (2,-.5) .. (2,0);
\draw[col5,very thick,->-] (0,0) .. controls (0,1) and (.2,1.8) .. (.5,2);
\draw[col5,very thick,->-] (1.5,2) .. controls (1.8,1.8) and (2,1) .. (2,0);
\draw[col3,very thick,->-] (2,0) .. controls (2,.5) and (0,.5) .. (0,0);
\draw[col5,very thick,->-] (.5,-2) .. controls (.2,-1.8) and (0,-1) .. (0,0);
\draw[col5,very thick,->-] (2,0) .. controls (2,-1) and (1.8,-1.8) .. (1.5,-2);
\node[ghostdot] at (2,0) {};
\node at (2.3,0) {\(b\)};
\node[dot] at (0,0) {};
\node at (-.3,0) {\(a\)};
\node at (1,-.8) {\(\psi_{ab}\)};
\node at (1,.8) {\(\overline{\psi_{ab}}\)};
\node at (.9,2) {\(\vdots\)};
\node at (1,2.4) {\(L_2\)};
\node at (.9,-2) {\(\vdots\)};
\node at (1,-2.4) {\(L_1\)};
}
\vspace{-.1in}
\caption{Illustration of the proof of Proposition \ref{banana_principle}: if the first configuration is realizable, so is the second one, which has the same number of circuits \textit{and} includes the $2$-circuit $(a,b)$.}
\label{fig:banana}
\end{figure}

\begin{remark}\label{cycle_notations}
In the rest of this section, we use the following notations to relate cycle structures of permutations of different sizes:
\begin{itemize}
\item if $\sigma \in \mathfrak{S}_q$, then $\fix \sigma \in \mathfrak{S}_{q+1}$ stands for the permutation obtained by extending $\sigma$ by a fixed point;
\item $\cyc \sigma \in \mathfrak{S}_{q+2}$ stands for the extension of $\sigma \in \mathfrak{S}_q$ by a transposition;
\item such notations are used with multiplicative convention, so that $\fix^{\! k} \cyc^{\! l} \sigma \in \mathfrak{S}_{q+k+2l}$ stands for the extension of $\sigma \in \mathfrak{S}_q$ by $k$ fixed points and $l$ transpositions;
\item in the same spirit, $(i, \fix)$ stands for the transposition between $i$ and one extra point added to the domain, so that $(i,\fix) \sigma$ means that this extra point is inserted between $\sigma^{-1}(i)$ and $i$.
\end{itemize}
We use the shorthand $i \in \sigma$ for sums over all indices in the domain of $\sigma$. An extra point is sometimes denoted by $\fixhat$ to avoid an ambiguity. Note that $\Wg(\sigma, \cdot)$ only depends on the cycle structure of $\sigma$, and not on any particular labelling of indices.
\end{remark}

The next theorem deals with expressions involving centered terms of order $2$, i.e., factors $|\psi_{ab}|^2-\frac{1}{d}$, which we call \textit{atoms}. Having these terms centered amounts to `penalizing' circuit coverings that contain the corresponding $2$-circuits. In this context, we take $\At \ccpair$ to be the set of $2$-circuits corresponding to atoms implied by $\ccpair$, and $\at \ccpair = |\At \ccpair|$. We also define the functions $\rho_K (\cdot)$ on permutations with at least $K$ fixed points, by the following formulas. For odd terms,
\begin{equation}\label{def_rho_odd}
    \rho_{2k+1}(\fix^{\!2k+1}\ccn) =
    (2k)!! \sum_{l=0}^k \frac{(2k-2l-1)!!}{(2k-2l)!!}
    \left( 
    4l+
    \sum_{i=1}^{\mathscr{C}(\omega)}
    \frac{4 C_i^2 - 2C_i }{C_i+1}
    \right),
\end{equation}
and for even terms,
\begin{equation}\label{def_rho_even}
% \rho_0(\ccn) = 1, \qquad
\rho_{2k} (\fix^{\!2k} \ccn) = (2k-1)!! 
\end{equation}
%for $k>0$, 
where $n!!$ is the semifactorial of the natural number $n$ (product of all smaller integers of same parity), and the convention that $(-1)!!=1$.

\begin{theorem}\label{thm:Wg_penalty}
If $(a_l, b_l)_{l=1}^L, (j_r,k_r,j_r',k_r')_{r=1}^R$ are indices and $\psi_G = \prod_r \psi_{j_r k_r} \overline{\psi_{j_r' k_r'}}$, then
\begin{equation}
\bbE
\left[ \prod_{l=1}^L \left( |\psi_{a_l b_l}|^2 - \frac{1}{d} \right)
    \psi_G
    \right]
    =
    \hspace{-.1in}
    \sum_{\ccpair \in \CC( \prod_l |\psi_{a_l b_l}|^2 \psi_G)}
    \hspace{-.3in} 
    d^{ - 2 \left\lfloor \frac{\at\ccpair+1}{2}    \right\rfloor }
    \rho_{\at\ccpair} (\cc)
    \Wg(\cc , d ) 
    \left( 1 + \bigO(d^{-1})\right).
    %\times
%    \left\{
%    \begin{array}{ll}
%    1 & \text{if } \at\ccpair = 0\\
%     \bigO \left(
%    d^{ - 2 \left\lfloor \frac{\at\ccpair+1}{2}    \right\rfloor } \right)
%   &  \text{else}
%   \end{array}
%    \right.
\end{equation}
\end{theorem}
To the best of our knowledge, the first result of this kind was obtained in the recent work of Bordenave and Collins \cite{BordenaveCollins2020} for products of centered terms. In principle, their approach could be extended to also include a non-centered term $\psi_G$, which would lead to a generalization of Theorem \ref{thm:Wg_penalty} with atoms of any size, and to a combinatorial understanding of the multiplicative penalty, which in this case reduces to the exponent $-2 \left\lfloor \frac{\at+1}{2} \right\rfloor$ and the constant $\rho_{\at}(\cc)$.

\begin{lemma}\label{lem1}
For any integer $K \geq 1$ and any permutation $\sigma \in \mathfrak{S}_{q}$,
\begin{equation}\label{induction_hyp_2} 
\sum_{k=0}^K \binom{K}{k}  \left(- \frac{1}{d} \right)^{k} \Wg (\fix^{\! K-k} \sigma, d ) =  d^{-2 \left\lfloor \frac{K+1}{2} \right\rfloor} 
\rho_K(\fix^{\! K}\sigma)
\Wg(\fix^{\! K} \sigma, d)
\left( 1 + \bigO \left( d^{-1} \right) \right),
\end{equation}
with $\rho_K(\cdot)$ defined by \eqref{def_rho_odd}, \eqref{def_rho_even}.
\end{lemma}

\begin{proof}
The essential input is a particular case of \cite[Proposition~2.2]{CollinsMatsumoto2017}, which we rewrite as
\begin{equation}\label{one_step_down_cc}
\Wg ( \fix \sigma, d) 
-
\frac{1}{d} \Wg(\sigma,d)  
=
- \frac{1}{d} \sum_{i \in \sigma} \Wg ( (i,\fix) \sigma  , d)
\end{equation}
with the notations introduced above.
% The recursive use of this proposition is key, as it simplifies the difference of two terms of order $d^{-2q+\mathscr{C}(\sigma)}$ as a sum of terms of that are smaller by a factor $d^{-2}$ and thus gives a precise structure to the $\bigO(d^{-2})$ term from \eqref{WgAsy}. 
We first treat the initial cases. For $K=1$, the left hand side can directly be rewritten using \eqref{one_step_down_cc}, as
\begin{equation}
\Wg(\fix \sigma, d) - \frac{1}{d} \Wg (\sigma,d)
= - \frac1d \sum_{i \in \sigma} \Wg ((i,\fix) \sigma) 
= d^{-2} \rho_1(\fix \sigma) \Wg (\fix \sigma,d) \left( 1 + \bigO (d^{-1}) \right)
\end{equation}
where the last estimate comes from comparing leading order terms with \eqref{WgAsy}, which gives the value
\begin{equation}\label{rho_1}
    \rho_1(\fix \sigma) = - \sum_{j \in \sigma} \frac{\Moeb((j,\fix) \sigma)}{\Moeb( \fix \sigma)}
%=\sum_{i=1}^{\mathscr{C}(\sigma)} C_i \frac{\mathrm{cat}_{C_i}}{\mathrm{cat}_{C_i-1}}
=\sum_{i=1}^{\mathscr{C}(\sigma)}  C_i \times \frac{ 4C_i - 2 }{C_i +1} ,
\end{equation}
noting that a cycle of size $C_i$ corresponds to $C_i$ terms with the same contribution in the first sum. For $K=2$, the expression at stake is
\begin{equation}
\Wg( \fix^{\!2} \sigma,d) - \frac{2}{d} \Wg(\fix \sigma,d)  + \frac{1}{d^2} \Wg( \sigma,d),
\end{equation}
which we rewrite as a difference of differences
\begin{equation}
\left( \Wg( \fix^{\!2} \sigma,d) - \frac{1}{d} \Wg(\fix \sigma,d)\right)
-
\frac{1}{d} \left( \Wg(\fix \sigma,d)  - \frac{1}{d} \Wg( \sigma,d) \right).
\end{equation}
Each term can be simplified according to \eqref{one_step_down_cc}, yielding
\begin{equation}
- \frac{1}{d} \Wg( \cyc \sigma ,d) - \frac{1}{d} \sum_{i \in \sigma} \left( \Wg( \fix (i,\fixhat) \sigma ,d)  -  \frac{1}{d} \Wg( (i,\fixhat) \sigma ,d)  \right),
\end{equation}
and these last differences can be dealt with using \eqref{one_step_down_cc} again, yielding a double sum of terms with low contribution. The final form of the expression is
\begin{equation}
- \frac{1}{d} \Wg( \cyc \sigma ,d) 
+
\frac{1}{d^2} \sum_{i \in \sigma} \sum_{j \in (i,\fixhat) \sigma} \Wg( (j, \fix) (i,\fixhat) \sigma ,d)
=
d^{-2} \Wg(\fix^{\!2} \sigma,d) \left( 1 + \bigO \left( d^{-1}\right) \right),
\end{equation}
which comes from comparing leading order terms with \eqref{WgAsy}. This gives $\rho_2(\fix^{\! 2} \sigma) = 1$. Note that we have only lost a factor $d^{-2}$ in leading order, which is the same penalty as for $K=1$. We will see that this exponent then decreases every two steps, so that the the resulting exponent is $-2 \left\lfloor \frac{K+1}{2} \right\rfloor$. We proceed by strong induction on $K$, with the induction hypothesis that the claim holds for $K-2$ and $K-1$ (and for all permutations in any case). We will use the property of binomial coefficients through the formula
\begin{equation}
\sum_{k=0}^K \binom{K}{k} a_k = \sum_{k=0}^{K-1} \binom{K-1}{k} (a_{k} + a_{k+1}),
\end{equation}
which yields
$$
\sum_{k=0}^K \binom{K}{k}  \left(- \frac{1}{d} \right)^{k} \Wg (\fix^{\! K-k} \sigma, d )
= \sum_{k=0}^{K-1} \binom{K-1}{k} \left(- \frac{1}{d} \right)^{k}
\left(  \Wg (\fix^{\!K-k} \sigma, d ) - \frac{1}{d} \Wg (\fix^{\! K-k-1} \sigma, d ) \right) ,
$$
in which we replace the differences using \eqref{one_step_down_cc}, which creates two kinds of terms:
\begin{equation}\label{two_terms}
\sum_{k=0}^{K-1} \binom{K-1}{k} 
\left(- \frac{1}{d} \right)^{k+1}
\left( 
(K-k-1) \Wg ( \fix^{\! K-k-1} \cyc \sigma, d )
+
\sum_{i \in \sigma}   \Wg ( \fix^{\! K-k-1} (i, \fixhat) \sigma, d ) \right).
\end{equation}
For the first terms, we rewrite
\begin{equation}
    (K-k-1) \binom{K-1}{k} = (K-1) \binom{K-2}{k},
\end{equation}
and then use the induction hypotheses at step $K-2$ for the permutation $\cyc \sigma$. This part of \eqref{two_terms} becomes
\begin{multline}\label{first_term}
-\frac{1}{d} (K-1)
\sum_{k=0}^{K-2} \binom{K-2}{k} \left(- \frac{1}{d} \right)^{k} \Wg ( \fix^{\! K-2-k} \cyc \sigma, d ) \\
= d^{-2 \left\lfloor \frac{K+1}{2} \right\rfloor } (K-1) \rho_{K-2}(\fix^{\! K-2} \cyc \sigma) \Wg (\fix^{\! K} \sigma ,d) \left(1+\bigO\left(d^{-1}\right) \right).
%& =
% \bigO\left(d^{-2 \left\lfloor \frac{K-1}{2} \right\rfloor - 1} \Wg (\fix^{\! K-2} \cyc \sigma ,d) \right) \\
%  =
% \bigO\left(d^{-2 \left\lfloor \frac{K-1}{2} \right\rfloor - 2} \Wg (\fix^{\! K} \sigma ,d) \right) 
%&  =
% \bigO\left(d^{-2 \left\lfloor \frac{K+1}{2} \right\rfloor } \Wg (\fix^{\! K} \sigma ,d) \right) ,
\end{multline}
The second terms in \eqref{two_terms} are dealt with using the induction hypotheses at step $K-1$ for the permutations $(i, \fixhat) \sigma$, yielding
\begin{multline}\label{second_term}
-\frac{1}{d}
\sum_{i \in \sigma}
 \sum_{k=0}^{K-1}
\binom{K-1}{k} 
\left(- \frac{1}{d} \right)^{k}
\Wg ( \fix^{\! K-k-1} (i, \fixhat) \sigma, d ) \\
 =
-  d^{- 2 \left\lfloor \frac{K}{2} \right\rfloor -1} 
\sum_{i \in \sigma}
\rho_{K-1}(\fix^{\! K-1} (i, \fixhat)  \sigma)
\Wg (\fix^{\! K-1} (i,\fixhat) \sigma ,d) \left(1+\bigO\left(d^{-1}\right) \right), %\\
% = \bigO\left( d^{- 2 \left\lfloor \frac{K}{2} \right\rfloor - 2} \Wg (\fix^{\! K } \sigma ,d) \right) 
%& = \bigO\left( d^{- 2 \left\lfloor \frac{K+2}{2} \right\rfloor } \Wg (\fix^{\! K } \sigma ,d) \right) 
\end{multline}
which is a $ \bigO \left( d^{- 2 \left\lfloor \frac{K+2}{2} \right\rfloor} \Wg (\fix^{\! K} \sigma,d) \right)$, so either the same order as \eqref{first_term} if $K$ is odd, or strictly smaller if $K$ is even. Accordingly, the general formula for the coefficients $\rho_K$ is found by distinguishing the odd and even terms. When $K=2k$, only the first term \eqref{first_term} contributes, yielding
\begin{equation}\label{even_rec}
\rho_{2k} (\fix^{\! 2k} \sigma) = (2k-1) \rho_{2k-2} (\fix^{\! 2k-2} \cyc \sigma).
\end{equation}
Together with the fact that $\rho_2 \equiv 1$, it follows immediately that $\rho_{2k} (\fix^{\! 2k} \sigma) = (2k-1)!!$ and does not depend on the structure of $\sigma$, whence \eqref{def_rho_even}. When $K=2k+1$, both terms \eqref{first_term} and \eqref{second_term} contribute, which implies 
\begin{equation}
\rho_{2k+1} (\fix^{\! 2k+1} \sigma) 
= 2k \rho_{2k-1} (\fix^{\! 2k-1} \cyc \sigma) - \sum_{i \in \sigma} \rho_{2k} (\fix^{\! 2k} (i,\fixhat) \sigma) \frac{\Moeb((i,\fixhat) \sigma)}{\Moeb(\fix \sigma)}.
\end{equation}
Replacing the even terms $\rho_{2k}$ by their value, and then recognizing $\rho_1$ from \eqref{rho_1}, we write:
\begin{align*}
\rho_{2k+1} (\fix^{\! 2k+1} \sigma) 
& = 2k \rho_{2k-1} (\fix^{\! 2k-1} \cyc \sigma) - \sum_{i \in \sigma} (2k-1)!! \frac{\Moeb((i,\fixhat) \sigma)}{\Moeb(\fix \sigma)}\\
& = 2k \rho_{2k-1} (\fix^{\! 2k-1} \cyc \sigma) +  (2k-1)!! \rho_1 (\fix \sigma).
\end{align*}
The solution of this recurrence is given by the formula
\begin{equation}
    \rho_{2k+1}(\fix^{\! 2k+1} \sigma ) = (2k)!! \sum_{l=0}^k \frac{(2k-2l-1)!!}{(2k-2l)!!} \rho_1(\fix \cyc^{\! l} \sigma),
\end{equation}
and the definition \eqref{def_rho_odd} follows from the fact that $\rho_1 (\fix \cyc^{\!l} \sigma) = 4l+ \rho_1 (\fix  \sigma)$.
\end{proof}

We can now complete the proof of Theorem \ref{thm:Wg_penalty}, which follows directly from Lemma \ref{lem1} once the terms are regrouped appropriately.

\begin{proof}[Proof of Theorem \ref{thm:Wg_penalty}]
We illustrate the argument by treating the first few cases in details. For $L=1$, one may distinguish between covering pairs $\ccpair \in \CC(|\psi_{ab}|^2 \psi_G)$ that do not include the $2$-circuit $(a,b)$ and those that do. The latter can be decomposed as a circuit covering of $\Psi_G$ \textit{and} the $2$-circuit $(a,b)$ (which yields a fixed point that we denote as $\fix$, according to Remark \ref{cycle_notations}), and then compared to the centered terms:
\begin{align*}
\bbE \left[ \left( |\psi_{ab}|^2 - \frac{1}{d} \right) \Psi_G \right]
& = \bbE \left( |\psi_{ab}|^2 \Psi_G \right) - \frac{1}{d} \bbE \left( \Psi_G \right) \\
& =
\sum_{\ccpair \in \CC(|\psi_{ab}|^2 \Psi_G)}
\hspace{-.3in} \Wg(\cc,d)
- d^{-1} \hspace{-.2in}
\sum_{\ccpair \in \CC(\Psi_G)}
\hspace{-.25in} \Wg(\cc,d) \\
& =
\sum_{\substack{\ccpair \in \CC(|\psi_{ab}|^2 \Psi_G)\\ (a,b) \notin \At\ccpair}} \hspace{-.25in} \Wg(\cc,d)
 +
 \hspace{-.2in}
\sum_{\ccpair \in \CC(\Psi_G)} \hspace{-.05in}
\left( \Wg( \fix \cc,d)
- \frac{1}{d}  \Wg( \cc, d) \right)
\end{align*}
The first terms do not suffer a penalty and are kept untouched (which corresponds to $\rho_0 (\ccn) =1$).
The second terms suffer a multiplicative penalty of $ d^{-2} $ from Lemma \ref{lem1}, so that the total contribution is
\begin{equation}
\sum_{\substack{\ccpair \in \CC(|\psi_{ab}|^2 \Psi_G)\\ (a,b) \notin \At\ccpair}}
\hspace{-.2in}
\Wg( \cc, d)
 \ + \hspace{-.1in}
\sum_{\substack{\ccpair \in \CC(|\psi_{ab}|^2 \Psi_G)\\ (a,b) \in \At\ccpair}}
% d^{-2 \left\lfloor \frac{\at(\cc)+1}{2} \right\rfloor} \rho_1(\cc) 
d^{-2} \rho_1(\cc)  \Wg( \cc, d)  \left( 1 + \bigO \left( d^{-1}\right) \right),
\end{equation}
which is consistent with the claim, as $\at = 0$ in the first sum and $\at=1$ in the second. We now treat the case $L=2$, by first expanding:
\begin{align*}
&\bbE \left[ \left( |\psi_{a_1 b_1}|^2 - \frac{1}{d} \right)
 \left( |\psi_{a_2b_2}|^2 - \frac{1}{d} \right)
 \Psi_G \right]
 \\
 &\qquad=  
\bbE \left[ \left( |\psi_{a_1 b_1}|^2 |\psi_{a_2b_2}|^2 - \frac{1}{d} |\psi_{a_1 b_1}|^2 - \frac{1}{d} |\psi_{a_2b_2}|^2 + \frac{1}{d^2} \right)
 \Psi_G \right]
\end{align*}
 and then rearranging the subsequent Weingarten terms depending on whether the structure of the circuit covering explicitly \textit{excludes} the presence of such atoms. In other words, we consider a covering pair to include an atom $(a_l,b_l)$ if it either implies it explicitly, or allows it implicitly (i.e., the atom not being in the support). This leads to the following decomposition:
\begin{align*}
 & \sum_{ \substack{\ccpair \in \CC(|\psi_{a_1 b_1}|^2 |\psi_{a_2b_2}|^2 \Psi_G) \\ (a_1,b_1), (a_2,b_2) \notin \At\ccpair}} \Wg(\cc,d) 
 & \!\!\!\!\!\!\footnotesize{\text{(terms excluding both atoms)}} \\
 & +  \sum_{\substack{\ccpair \in \CC(|\psi_{a_2b_2}|^2 \Psi_G) \\ (a_2,b_2) \notin \At\ccpair}} \left( \Wg(\fix \cc,d) - \frac{1}{d} \Wg(\cc,d) \right) 
 & \!\!\!\!\!\!\footnotesize{\text{(terms excluding \((a_2,b_2)\) only)}}  \\
 & +  \sum_{\substack{\ccpair \in \CC(|\psi_{a_1 b_1}|^2 \Psi_G) \\ (a_1,b_1) \notin \At\ccpair}} \left( \Wg(\fix \cc,d) - \frac{1}{d} \Wg(\cc,d) \right) 
 & \!\!\!\!\!\!\footnotesize{\text{(terms excluding \((a_1,b_1)\) only) }} \\
 & +  \sum_{\ccpair \in \CC(\Psi_G)} \hspace{-.05in} \left( \Wg( \fix^{\!2} \cc,d) - \frac{2}{d} \Wg(\fix \cc,d)  + \frac{1}{d^2} \Wg( \cc,d) \right) 
 & \!\!\!\!\!\!\footnotesize{\text{(terms excluding no atom) }}
\end{align*}
Note that these four groups correspond to a number of atoms $\at\ccpair$ of $0$, $1$, $1$ and $2$ respectively. The first sum is kept unchanged. The second and third are treated as above, using Lemma \ref{lem1} for $K=\at\ccpair=1$. The last sum is treated with Lemma \ref{lem1} for $K=\at\ccpair=2$, which brings down the leading order contributions by a factor $d^{- \lfloor (K +1) /2 \rfloor} = d^{-2}$. More generally, for any choice of $K$ atoms (among $L$ possible atoms), the terms that do not strictly exclude the corresponding circuits but do exclude the $L-K$ others are grouped together and treated using Lemma \ref{lem1}, getting a multiplicative penalty of $d^{- \lfloor (K+1)/2  \rfloor} \rho_K(\cc)$ to the leading order term, where $K=\at\ccpair$, which is the claim. \end{proof}

For practical purposes, the results stated above can be summed up as the following method, which can be applied when it comes to identifying the leading order of an expression involving $\psi$ and $\overline{\psi}$ terms, as well as centered terms of order $2$:
\begin{enumerate}
    \item\label{rule1} The leading order term of $\bbE \psi_G$ is obtained by maximizing the number of circuits in a circuit-covering of $G$ (Theorem \ref{t:WG}, \eqref{WgAsy}, and Fact \ref{fact:correspondence}); an elementary bound is given by Proposition \ref{p:WGBD}.
    \item\label{rule2} If $\psi_G$ contains terms $|\psi_{ab}|^2$, it is always `safe' to take the corresponding $2$-circuits, in the sense that the total number of circuits could not be strictly increased by not taking them (Proposition \ref{banana_principle});
    \item\label{rule3} \textbf{but}, if the expression at stake includes atoms, i.e., centered terms $|\psi_{ab}|^2-1/d$, taking the corresponding $2$-circuits implies a multiplicative penalty of at least $d^{-1}$ for each such circuit taken (Theorem \ref{thm:Wg_penalty}).
    %whereas circuit coverings that do not contain any of the centered terms have their contribution unchanged.
\end{enumerate}

To put it in one sentence: in order to identify the leading order, one wants to maximize the number of circuits, taking as many non-centered $2$-circuits as possible, but avoiding atoms (unless it is impossible to do so). In principle, the remaining combinatorics in the choice of larger circuits can still be extremely complex. However, in the cases at stake here, these elementary rules will be sufficient to identify the leading order with certainty.

\subsection{Application to bounds on moments of $\gamma$ terms}\label{s:application_Wg}

We will now apply the general results of Section \ref{s:Wgcalc} to the particular expressions at stake here, namely, the expectation of products of terms such as
\begin{equation}\label{gamma_def2}
\gamma(k,l,m,n) := \sum_{j=1}^d \overline{\psi_{jk}} \psi_{jl} \overline{\psi_{jm}} \psi_{jn}
- \frac1d \delta_{kl} \delta_{mn}
- \frac1d \delta_{kn} \delta_{lm}.
\end{equation}
Note that, as $d$ goes to infinity, we can always assume we are in the stable range of Weingarten calculus and use the above asymptotics. For simplicity, we will refer to an expression such as \eqref{gamma_def2} as a gamma term `of type $klmn$', where the letters $k,l,m,n$ are chosen so as to indicate relevant identities of indices. \\

A preliminary remark is that any $2$-circuit that can be obtained within a single gamma term can be disregarded, as coming from an atom in the sense of section \ref{s:Wgcalc} (see Theorem \ref{thm:Wg_penalty}). Indeed, if at most one $2$-circuit can be formed, then the gamma term is of type $kkmn$ with $m \neq n$ (up to reordering and renaming indices), and therefore it can be written as
\begin{equation}\label{gamma_aacd}
\gamma(k,k,m,n) = \sum_{j=1}^d |\psi_{jk}|^2 \overline{\psi_{jm}} \psi_{jn}
= \sum_{j=1}^d \left( |\psi_{jk}|^2-\frac1d \right) \overline{\psi_{jm}} \psi_{jn}
\end{equation}
thanks to orthogonality of columns $m$ and $n$. Else, if two $2$-circuits can be formed, then (still up to a choice of indices) the gamma term is either of type $kkmm$ with $k \neq m$, or of type $kkkk$. In the first case,
\begin{equation}\label{gamma_aabb}
\gamma(k,k,m,m) = \sum_{j=1}^d 
|\psi_{jk}|^2 |\psi_{jm}|^2
- \frac1d
= \sum_{j=1}^d 
\left( |\psi_{jk}|^2 - \frac1d \right) \left(|\psi_{jm}|^2 -\frac1d \right),
\end{equation}
thanks to columns $k$ and $m$ having norm $1$. The last case is that of a term of type $kkkk$, in which case the term is not exactly, but still \textit{essentially} centered, in the sense that it is centered up to a small constant. Namely:
\begin{equation}\label{gamma_kkkk}
\gamma(k,k,k,k) = \sum_{j=1}^d  |\psi_{jk}|^4
- \frac2d
= \sum_{j=1}^d   \left( |\psi_{jk}|^4 - \bbE  |\psi_{jk}|^4 \right) - \frac{2}{d(d+1)}.
\end{equation}
For the sake of simplicity, we introduce the exactly centered terms
\begin{equation}\label{true_gamma_kkkk}
\widetilde{\gamma}(k,k,k,k) 
:= 
\sum_{j=1}^d   \left( |\psi_{jk}|^4 - \bbE  |\psi_{jk}|^4 \right).
\end{equation}
Our proof strategy regarding these centered $\widetilde{\gamma}$ terms can be summarized as follows.
\begin{remark}\label{strategy}
Expectation of products involving terms such as \eqref{true_gamma_kkkk} can be bounded in the same spirit thanks to the result of Bordenave and Collins (Proposition 8 and Theorem 10 in \cite{BordenaveCollins2020}). Similarly to Theorem \ref{thm:Wg_penalty}, the consequence is that a multiplicative penalty is inflicted to coverings pairs that isolate a part of the graph corresponding to a centered term $|\psi_{jk}|^4 - \bbE (|\psi_{jk}|^4)$ from the rest of the graph. Our proofs will proceed by first estimating expressions involving $\widetilde{\gamma}(k,k,k,k)$ terms, and then checking that replacing them by the actual value $\gamma(k,k,k,k)$ as in \eqref{gamma_kkkk} does not affect the result, owing to the small difference (of order $d^{-2}$) between the two terms.
\end{remark}

We first illustrate our method by evaluating moments of gamma terms of types $klmn$ and $kkmn$, with distinct $k,l,m,n$.

\begin{proposition}\label{prop:moments_klmn} Let \(\gamma\) be defined as in \eqref{gamma_def2}. Then, for distinct $k,l,m,n$ we have, when $d \rightarrow \infty$:
\begin{equation}%\label{gamma-2p}
\bbE \left( \gamma (k,l,m,n)^{p} \overline{\gamma (k,l,m,n)^{q}} \right) = 
\begin{cases} p! \ d^{-3p} + \bigO_p(d^{-3p-1})  &\qtq{if} p=q,\smallskip\\
0 & \qtq{otherwise.}\end{cases}
\end{equation}
and the same holds if $l=k$.
\end{proposition}
It follows in particular that $d^{3/2} \gamma (k,l,m,n)$ converges in distribution to a complex normal variable with mean $0$ and variance $1$, and likewise for $d^{3/2} \gamma (k,k,m,n)$.

\begin{proof}
We apply the method outlined in Section \ref{s:Wgcalc} to the expression
 \begin{equation}\notag
 \bbE \left( \gamma (k,l,m,n)^{p} \overline{\gamma (k,l,m,n)^{q}} \right)
 =
 \!\!\!\!\sum_{j_1,\dots,j_{p}, j_{1}', \dots, j_{q}'} \!\!\!\!
 \bbE \left[
 \prod_{i=1}^p
\overline{\psi_{j_i k}} {\psi_{j_i l}} \overline{\psi_{j_i m}} {\psi_{j_i n}}
 \times
 \prod_{i=1}^q
{\psi_{j_{i}' k}} \overline{\psi_{j_{i}' l}} {\psi_{j_{i}' m}} \overline{\psi_{j_{i}' n}}
 \right].
 \end{equation}
 The associated Weingarten graphs are represented in Figure \ref{f:moment_bound12}. First, note that if $p \neq q$ then this expectation is zero, for instance by considering that vertex $n$ has $p$ incoming edges and $q$ outgoing edges, which imposes $p=q$ for a circuit covering of the whole graph to exist.
 \begin{figure}[h!]
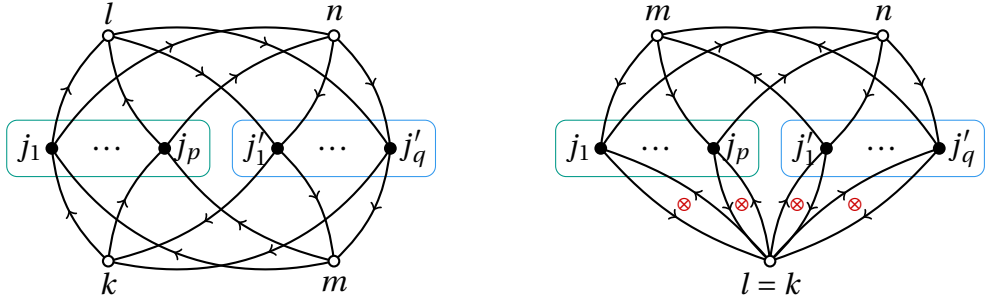

\tikz[scale=1.5]{
\draw[col3, rounded corners] (-1.9, -.25) rectangle (-.1, .25) {};
\draw[col5, rounded corners] (.1, -.25) rectangle (1.9, .25) {};
%Z\draw[col4, rounded corners] (-1.2, 1.12-.25) rectangle (1.2, 1.12+.25) {};
%\draw[col6, rounded corners] (-1.2, -1.12-.25) rectangle (1.2, -1.12+.25) {};
\draw[thick,->-] (-1,-1) to[bend left=20] (-1.5,0);
\draw[thick,->-] (-1.5,0) to[bend left=35] (1,1) ;
\draw[thick,->-] (-1.5,0) to[bend left=20] (-1,1);
\draw[thick,->-] (1,-1) to[bend left=35] (-1.5,0);
\draw[thick,->-] (-1,-1) to[bend left=20] (-.5,0);
\draw[thick,->-] (-.5,0) to[bend left=20] (1,1) ;
\draw[thick,->-] (-.5,0) to[bend left=20] (-1,1);
\draw[thick,->-] (1,-1) to[bend left=20] (-.5,0);
\draw[thick,->-] (.5,0) to[bend left=20] (-1,-1);
\draw[thick,->-] (1,1) to[bend left=20] (.5,0) ;
\draw[thick,->-] (-1,1) to[bend left=20] (.5,0);
\draw[thick,->-] (.5,0) to[bend left=20] (1,-1);
\draw[thick,->-] (1.5,0) to[bend left=35] (-1,-1);
\draw[thick,->-] (1,1) to[bend left=20] (1.5,0) ;
\draw[thick,->-] (-1,1) to[bend left=35] (1.5,0);
\draw[thick,->-] (1.5,0) to[bend left=20] (1,-1);
\node[ghostdot] at (-1,-1) {};
\node at (-1,-1.2) {$k$};
\node[ghostdot] at (-1,1) {};
\node at (-1,1.2) {$l$};
\node[ghostdot] at (1,-1) {};
\node at (1,-1.2) {$m$};
\node[ghostdot] at (1,1) {};
\node at (1,1.2) {$n$};
\node[dot] at (-.5,0) {};
\node at (-.3,0) {$j_p$};
\node at (-1,0) {$\dots$};
\node[dot] at (-1.5,0) {};
\node at (-1.7,0) {$j_1$};
\node[dot] at (.5,0) {};
\node at (.3,0) {$j_1'$};
\node at (1,0) {$\dots$};
\node[dot] at (1.5,0) {};
\node at (1.7,0) {$j_q'$};
}
\qquad \qquad
\tikz[scale=1.5]{
\draw[col3, rounded corners] (-1.9, -.25) rectangle (-.1, .25) {};
\draw[col5, rounded corners] (.1, -.25) rectangle (1.9, .25) {};
%\draw[col4, rounded corners] (.7, 1.12-.25) rectangle (1.3, 1.12+.25) {};
%\draw[col6, rounded corners] (-1.3, 1.12-.25) rectangle (-.7, 1.12+.25) {};
\draw[thick,->-] (-1,1) to[bend right=20] (-1.5,0);
\draw[thick,->-] (-1,1) to[bend right=20] (-.5,0);
\draw[thick,->-] (-1.5,0) to[bend left=35] (1,1);
\draw[thick,->-] (-.5,0) to[bend left=20] (1,1);
\draw[thick,->-] (1,1) to[bend left=20] (1.5,0);
\draw[thick,->-] (1,1) to[bend left=20] (.5,0);
\draw[thick,->-] (1.5,0) to[bend right=35] (-1,1);
\draw[thick,->-] (.5,0) to[bend right=20] (-1,1);
\draw[thick,->-] (-1.5,0) to[bend right=15] (0,-1);
\draw[thick,->-] (0,-1) to[bend right=15] (-1.5,0);
\draw[thick,->-] (-.5,0) to[bend right=20] (0,-1);
\draw[thick,->-] (0,-1) to[bend right=20] (-.5,0);
\draw[thick,->-] (.5,0) to[bend left=20] (0,-1);
\draw[thick,->-] (0,-1) to[bend left=20] (.5,0);
\draw[thick,->-] (1.5,0) to[bend left=15] (0,-1);
\draw[thick,->-] (0,-1) to[bend left=15] (1.5,0);
\node[ghostdot] at (-1,1) {};
\node at (-1,1.2) {$m$};
\node[ghostdot] at (1,1) {};
\node at (1,1.2) {$n$};
\node[ghostdot] at (0,-1) {};
\node at (0,-1.2) {$l=k$};
\node[dot] at (-.5,0) {};
\node at (-.3,0) {$j_p$};
\node at (-1,0) {$\dots$};
\node[dot] at (-1.5,0) {};
\node at (-1.7,0) {$j_1$};
\node[dot] at (.5,0) {};
\node at (.3,0) {$j_1'$};
\node at (1,0) {$\dots$};
\node[dot] at (1.5,0) {};
\node at (1.7,0) {$j_q'$};
\node at (-.73,-.5) {\color{red!80!black} \scalebox{.8}{$\otimes$} \color{black}};
\node at (-.22,-.5) {\color{red!80!black} \scalebox{.8}{$\otimes$} \color{black}};
\node at (.27,-.5) {\color{red!80!black} \scalebox{.8}{$\otimes$} \color{black}};
\node at (.78,-.5) {\color{red!80!black} \scalebox{.8}{$\otimes$} \color{black}};
}
\caption{Weingarten graphs considered in the proof of Proposition \ref{prop:moments_klmn}. The crossed out circles in the second graph indicate that the $2$-circuits corresponding to atoms involve a penalty, according to Theorem \ref{thm:Wg_penalty}.}
\label{f:moment_bound12}
\end{figure}
 Then, note that a $2$-circuit that is not an atom is only made possible by an equality $j_i = j_r'$, and each such equality enables four $2$-circuits in the corresponding terms. Proposition \ref{banana_principle} ensures that these can be taken first (and in any order) when building an optimal circuit covering. Let $\ell$ be the number of pairs $j=j'$ used in picking the $2$-circuits (yielding $4\ell$ of them). It follows that the number of distinct $j$-indices (corresponding to distinct vertices in the graph) is bounded by $2p-\ell$. The remaining edges can at best be used in $4$-circuits, and a quick count yields at most $2p-2l$ of these. 
% the number of $4$-circuits is then $(8p-8l)/4 = 2p-2l$. 
The total contribution is therefore bounded by:
\begin{equation}\notag
\underbrace{d^{2p-\ell}}_{\text{free indices } j}
\times
\underbrace{d^{-8p} }_{-2q}
\times
\underbrace{d^{4\ell} }_{2\text{-circuits}}
\times
\underbrace{d^{2p-2\ell} }_{\text{other circuits}}
= d^{-4p+\ell} \leq d^{-3p}.
\end{equation}
The only situation in which $d^{-3p}$ is reached is with $\ell=p$, exactly $p$ free $j$-indices, and the other edges all optimally used in $4$-circuits. This is the case if and only if the edges are matched in exact pairs, and such a configuration only depends on the way we pair the indices $j$ and $j'$, so that the overall contribution is
\begin{equation}\label{Gaussian_moment}
\bbE \left| \gamma \right|^{2p} = p! \ d^{-3p} + \bigO(d^{-3p - 1}),
\end{equation}
as claimed.
\end{proof}

 It is actually true of all gamma terms that their fluctuations are of order $\bigO(d^{-3/2})$, and it can be checked in the same way that once renormalized, these terms converge in distribution when $d \rightarrow \infty$ to either a real or a complex normal distribution, depending on the values of $k,l,m,n$, with a variance that also depends on the type. 
 
 Our purpose is now to bound the expectation of a product of gamma terms when it contains terms of different types. The general idea is that such an expectation will be higher when terms are combined in certain ways (typically: paired). This motivates the following definition of what we consider to be `smart combinations'.
 \begin{definition}\label{smart_def}
 Given a particular identification of $j$-indices (i.e. specifying which indices are equal and which are not) and a circuit covering of the resulting Weingarten graph, we say that
\begin{enumerate}
    \item a smart pair is a configuration where a pair of gamma terms form four $2$-circuits with one common index $j$;
    \item a smart triplet is when a triple $\gamma_1, \gamma_2, \gamma_3$ forms six $2$-circuits with one index $j$;
    \item a smart $2$-chain is when a pair $\gamma_1, \gamma_2$ that forms three circuits (two $2$-circuits and one $4$-circuits) with two distinct $j$'s.
    \item a smart $3$-chain is a triple $\gamma_1 \gamma_2 \gamma_3$ that allows $4$ circuits (two $2$-circuits and two $4$-circuits) with three distinct $j$'s.
\end{enumerate}
\end{definition}
By inspection, smart pairs only occur when pairing a $\gamma$ term with its corresponding $\overline{\gamma}$ term, and smart $2$-chains only occur with terms of type $kkkk$. Figure \ref{f:smart_guys} represents a possible configuration for a smart triplet and the only possible configurations for smart $2$-chain and smart $3$-chain respectively.

\begin{figure}[b!]
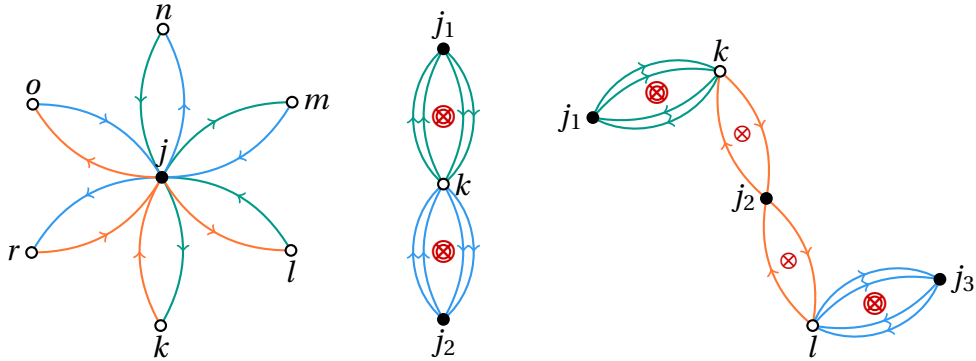

\tikz[scale=1,rotate=30]{
\draw[thick,->-, col3] (0,0) to[bend left=30] (-1,-1.7);
\draw[thick,->-,col3] (1,-1.7) to[bend right=30] (0,0);
\draw[thick,->-,col3] (0,0) to[bend left=30] (2,0);
\draw[thick,->-,col3] (1,1.7) to[bend right=30] (0,0);
\draw[thick,->-, col5] (-1,1.7) to[bend left=30] (0,0);
\draw[thick,->-,col5] (0,0) to[bend right=30] (-2,0);
\draw[thick,->-,col5] (2,0) to[bend left=30] (0,0);
\draw[thick,->-,col5] (0,0) to[bend right=30] (1,1.7);
\draw[thick,->-, col4] (0,0) to[bend left=30] (-1,1.7);
\draw[thick,->-,col4] (-2,0) to[bend right=30] (0,0);
\draw[thick,->-,col4] (0,0) to[bend right=30] (1,-1.7);
\draw[thick,->-,col4] (-1,-1.7) to[bend left=30] (0,0);
\node[ghostdot] at (-1,-1.7) {};
\node[below] at (-1,-1.7) {$k$};
\node[ghostdot] at (-1,1.7) {};
\node[above] at (-1,1.7) {$o$};
\node[ghostdot] at (2,0) {};
\node[right] at (2,0) {$m$};
\node[ghostdot] at (1,-1.7) {};
\node[below] at (1,-1.7) {$l$};
\node[ghostdot] at (1,1.7) {};
\node[above] at (1,1.7) {$n$};
\node[ghostdot] at (-2,0) {};
\node[left] at (-2,0) {$r$};
\node[dot] at (0,0) {};
\node[above] at (0,0) {$j$};
}
\qquad
\tikz[scale=1.8,rotate=0]{
\draw[thick,->-, col3] (0,0) to[bend left=30] (0,1);
\draw[thick,->-,col3] (0,0) to[bend left=50] (0,1);
\draw[thick,->-,col3] (0,1) to[bend left=30] (0,0);
\draw[thick,->-,col3] (0,1) to[bend left=50] (0,0);
\draw[thick,->-, col5] (0,0) to[bend left=30] (0,-1);
\draw[thick,->-,col5] (0,0) to[bend left=50] (0,-1);
\draw[thick,->-,col5] (0,-1) to[bend left=30] (0,0);
\draw[thick,->-,col5] (0,-1) to[bend left=50] (0,0);
\node[dot] at (0,1) {};
\node[above] at (0,1) {$j_1$};
\node[dot] at (0,-1) {};
\node[below] at (0,-1) {$j_2$};
\node[ghostdot] at (0,0) {};
\node[right] at (0,0) {$k$};
%\node at (0.025,.5) {\color{red!80!black} \scalebox{1}{$\otimes$} \color{black}};
\draw[red!80!black, thick] (0,.5) circle(.05);
\node at (0.025,.5) {\color{red!80!black} \scalebox{1.4}{$\otimes$} \color{black}};
%\node at (0.025,-.5) {\color{red!80!black} \scalebox{1}{$\otimes$} \color{black}};
\draw[red!80!black, thick] (0,-.5) circle(.05);
\node at (0.025,-.5) {\color{red!80!black} \scalebox{1.4}{$\otimes$} \color{black}};
}
\qquad
\tikz[scale=1.8,rotate=20]{
\draw[thick,->-, col3] (-1,1) to[bend left=30] (0,1);
\draw[thick,->-,col3] (-1,1) to[bend left=50] (0,1);
\draw[thick,->-,col3] (0,1) to[bend left=30] (-1,1);
\draw[thick,->-,col3] (0,1) to[bend left=50] (-1,1);
\draw[thick,->-, col4] (0,0) to[bend left=30] (0,1);
\draw[thick,->-,col4] (0,1) to[bend left=30] (0,0);
\draw[thick,->-,col4] (0,-1) to[bend left=30] (0,0);
\draw[thick,->-,col4] (0,0) to[bend left=30] (0,-1);
\draw[thick,->-, col5] (1,-1) to[bend left=30] (0,-1);
\draw[thick,->-,col5] (1,-1) to[bend left=50] (0,-1);
\draw[thick,->-,col5] (0,-1) to[bend left=30] (1,-1);
\draw[thick,->-,col5] (0,-1) to[bend left=50] (1,-1);
\node[dot] at (-1,1) {};
\node[left] at (-1,1) {$j_1$};
\node[ghostdot] at (0,1) {};
\node[above] at (0,1) {$k$};
\node[dot] at (0,0) {};
\node[left] at (0,0) {$j_2$};
\node[ghostdot] at (0,-1) {};
\node[below] at (0,-1) {$l$};
\node[dot] at (1,-1) {};
\node[right] at (1,-1) {$j_3$};
\node at (.02,.5) {\color{red!80!black} $\otimes$ \color{black}};
\node at (.02,-.5) {\color{red!80!black} $\otimes$ \color{black}};
%\node at (-.47,1) {\color{red!80!black} \scalebox{1}{$\otimes$} \color{black}};
\draw[red!80!black, thick] (-.4925,1.0075) circle(.05);
\node at (-.47,1) {\color{red!80!black} \scalebox{1.4}{$\otimes$} \color{black}};
%\node at (.51,-1.01) {\color{red!80!black} \scalebox{1}{$\otimes$} \color{black}};
\draw[red!80!black, thick] (.4875,-1.0025) circle(.05);
\node at (.51,-1.01) {\color{red!80!black} \scalebox{1.4}{$\otimes$} \color{black}};
}
\caption{From left to right: a smart triplet involving terms of type $klmn$, $nmor$ and $rolk$; a smart $2$-chain involving terms of type $kkkk$; a smart $3$-chain involving terms of type $kkkk$, $kkll$ and $llll$. 
%Both smart triplet and smart $3$-chains have a near-optimal contribution of $d^{-5}$, leading to the overall estimate of a $1/6$ penalty per term. 
Note that imposing $j_1=j_2=j_3$ transforms the smart $3$-chain into another smart triplet.}
\label{f:smart_guys}
\end{figure}

\begin{theorem}[Moment bound for gamma terms: general case]\label{thm:moment_bound}
We consider a product $\gamma_1 \cdots \gamma_{p} $ of gamma terms, denoting by $p_2^*$ the largest number of smart pairs that can be obtained, and by $p_3^*$ the largest number of smart triplets that can be obtained once $p_2^*$ smart pairs are removed. Then,
\begin{equation}\label{penalty_count}
\bbE \left[ \gamma_1 \cdots \gamma_{p} \right] = 
\bigO_{p} \left( d^{-\frac32 p - \frac16 (3p_3^*) - \frac14 (p-2p_2^* - 3p_3^*) } \right).
\end{equation}
\end{theorem}
\noindent This is to say that the best possible contribution (of order $d^{-\frac32 p}$) can be reached with exact pairs only; in other situations, smart triplets terms yield a $-\frac16$ penalty to the exponent, and all other terms a penalty of $-\frac14$.

\begin{proof}
As explained above (after the definition \eqref{true_gamma_kkkk}), we perform the computation with terms $\gamma(k,k,k,k)$ replaced by $\widetilde{\gamma}(k,k,k,k)$, and check in the end that this does not make a difference. After expanding the product as a sum over indices $j_1, \dots, j_{p}$ and applying Theorem \ref{thm:Wg_penalty}, we consider the order of the contribution of a given circuit-covering for a specific partition of the $j$-indices. For such configuration, we denote by $\ell_n$ the number of groups of $n$ equal $j$-indices, so that $\ell_1$ corresponds to isolated $j$'s, $\ell_2$ corresponds to terms being paired (but not necessarily optimally so), etc. As this is a partition of all $j$-indices, we have
\begin{equation}\label{partition}
\sum_{n \geq 1} n \ell_n = p.
\end{equation}
Note that a $2$-circuit can be created without penalty only if either two $j$'s  are equal, \textit{or} else in a term of type $kkkk$ --- these terms represent the main difficulty here and require careful treatment (see last bullet point below). We review the different ways in which $2$-circuits can appear, referring to smart pairs, smart triplets and smart chains as introduced in Definition \ref{smart_def}:
\begin{itemize}
    \item Among $\ell_2$ pairs of $j$-indices, we denote by $p_2$ those that correspond to smart pairs, i.e., two gamma terms yielding \textbf{four} $2$-circuits. The other $\ell_2 - p_2$ pairs allow at most \textbf{three} $2$-circuits each.
    \item Among $\ell_3$ triplets of $j$-indices, we denote by $p_3$ those that correspond to smart triplets, leading to \textbf{six} $2$-circuits for each smart triplet. The other $\ell_3 - p_3$ triplets allow at most \textbf{five} $2$-circuits each.
    \item Groups of $n$ $j$-indices with $n \geq 4$ induce at most $2 \ell_n$ circuits (this bound is achieved in the optimal situation where all circuits are $2$-circuits); the gain we expect here comes solely from the missing $j$-indices.
    \item The last possibility of a $2$-circuit is in a term of type $kkkk$ with isolated $j$-index. We call $m$ the number of such terms, and distinguish them according to what happens to the other two edges.
    \begin{itemize}
        \item If the remaining two edges also form a $2$-circuit, this is cancelled by a penalty, and so we say these two edges are `wasted'.
        \item If they form an $n$-circuit with a non-isolated term and $n \geq 4$ (or an isolated term and $n \geq 6$), the contribution is equivalent to that of an $(n-2)$-circuit and a (penalized) $2$-circuit. Again, we say these two edges are wasted.
        \item The only case in which these edges are not considered wasted is if they form a $4$-circuit with two other edges from another isolated term, which is necessarily a $kkmn$-type term; $2$-chains and $3$-chains are possible examples of such a configuration, which happen to be optimal and near-optimal, respectively. We denote by $m_2$ the number of $2$-chains and by $m_3$ the number of $3$-chains. By inspection, in any other possible case, such $kkkk$ terms cannot exceed a half of the total number of isolated terms.
    \end{itemize}
    We denote by $m_0$ the number of isolated terms of type $kkkk$ that lead, or are equivalent, to `wasting two edges', and by $m_1$ those that are linked to other isolated terms. The above considerations imply
    \begin{equation}\label{kkkk_bound}
    m= m_0 + m_1 + 2 m_2+ 2 m_3, \quad \text{ and } \quad m_1 \leq \frac12 (\ell_1-m_0- 2 m_2- 3 m_3).
    \end{equation}
\end{itemize}

We now compute the order of the contributions:
\begin{itemize}
\item The number of free $j$-indices is $\sum_{n \geq 1} \ell_n$, the number of parts in the partition. 
Using \eqref{partition}, we obtain
$$
\sum_{n \geq 4}  \ell_n \leq \frac14 \sum_{n \geq 4} n \ell_n = \frac14 (p - \ell_1 - 2 \ell_2 - 3\ell_3),
$$
and thus the bound
\begin{equation}
\# \{ \text{free } j \text{ indices} \}
 = \sum_{n \geq 1} \ell_n 
\leq \frac14 p + \frac34 \ell_1 + \frac12 \ell_2 + \frac14 \ell_3.
\end{equation}
\item The number of (unpenalized) $2$-circuits is bounded by 
$$
4p_2 + 3(\ell_2-p_2) 
+ 6 p_3 + 5(\ell_3 - p_3)
+ 2 \sum_{n \geq 4} n \ell_n + m,
$$
which yields after simplifications, again using \eqref{partition}:
\begin{equation}
    \# \{ 2\text{-circuits} \} \leq 2p + p_2 -\ell_2+ p_3 - \ell_3 - 2 \ell_1 + m.
\end{equation}
\item The edges that are neither wasted nor part of any $2$-circuits can at best form $4$-circuits. By inspection, starting with $4p$ edges and discounting $2m_0$ wasted edges,  the total number of circuits is bounded by 
\begin{equation}
\# \{ \text{circuits} \} 
\leq p + \frac12 \#\{2\text{-circuits}\} - \frac12 m_0
\leq 2p + \frac12 (p_2 - \ell_2) + \frac12 (p_3 - \ell_3) - \ell_1 + \frac12 m - \frac12 m_0.
\end{equation}
% (Former way to count) \color{red} Once $2$-circuits are discounted, the remaining (and unwasted) edges can at best form $4$-circuits. These edges come from unpaired terms (with the exception of isolated $kkkk$ terms, as well as half the edges from $m_1$ terms) and from unoptimally paired terms, yielding
%\begin{equation}
%    \# \{ n\text{-circuits with } n \geq 4 \} \leq \ell_1 - m_0 - \frac12 m_1 -\frac12 m_2 - \frac12 m_3 + \frac12 (\ell_2-p_2) + \frac12 (\ell_3 - p_3).
%\end{equation}
%\color{black}
\end{itemize}
According to the method outlined in Section \ref{s:Wgcalc}, the contribution of such a configuration is bounded by 
\begin{equation}\label{first_bound}
\underbrace{d^{\frac14 p + \frac34 \ell_1 + \frac12 \ell_2 + \frac14 \ell_3}}_{\text{free indices } j}
\times
\underbrace{d^{-4p} }_{-2q}
\times
\underbrace{d^{
2p + \frac12 (p_2 - \ell_2) + \frac12 (p_3 - \ell_3) - \ell_1 + \frac12 m - \frac12 m_0
} }_{\text{total number of circuits}}.
\end{equation}
% (Former way to count) \color{red}
%\begin{equation}\label{first_bound_?}
%\underbrace{d^{\frac14 p + \frac34 \ell_1 + \frac12 \ell_2 + \frac14 \ell_3}}_{\text{free indices } j}
%\times
%\underbrace{d^{-4p} }_{-2q}
%\times
%\underbrace{d^{2p + (p_2 -\ell_2) + (p_3 - \ell_3) - 2 \ell_1 + m} }_{2\text{-circuits}}
%\times
%\underbrace{d^{\ell_1 - m_0 - \frac12 m_1 - \frac12 m_2 - %\frac12 m_3 + \frac12 (\ell_2-p_2) + \frac12 %(\ell_3-p_3)} }_{\text{other circuits}}
%\end{equation}
%\color{black}
After rearranging, the exponent becomes
\begin{equation}\label{second_bound}
-\frac32 p 
- \frac14 (p- \ell_1 -2\ell_2-3\ell_3)
- \frac14 (2\ell_2 - 2p_2)
- \frac13 (3\ell_3 - 3p_3)
- \frac16 (3p_3)
+ \frac12 (m-m_0)
- \frac12 \ell_1.
\end{equation}
Using \eqref{kkkk_bound}, we get the following bound on the exponent:
\begin{align*}\label{last_bound}
-\frac32 p 
- \frac16 \hspace{-.1in} \underbrace{(3 p_3 + 3 m_3)}_{\text{ \tiny smart triplets and 3-chains}} \hspace{-.1in}
&- \frac14 \underbrace{(p- \ell_1 -2\ell_2-3\ell_3)}_{\text{ \tiny terms in groups of size } n \geq 4}
\\
&- \frac14 \underbrace{(2\ell_2 - 2p_2)}_{\text{\tiny non-optimal pairs}}
- \frac13 \underbrace{(3\ell_3 - 3p_3)}_{\text{\tiny non-optimal triplets}}
- \frac14 \hspace{-.1in} \underbrace{(\ell_1-2m_2-3m_3)}_{\text{\tiny isolated terms outside smart chains}}, \hspace{-.1in}
\end{align*}
which means, in concrete terms, the following penalties from the optimal $- \frac32 p$ exponent:
\begin{itemize}
    \item[--] no penalty for smart pairs and smart $2$-chains;
    \item[--] penalties of $-1/6$ for smart triplets and smart $3$-chains;
    \item[--] penalties of $-1/4$ for unpaired and untripled terms, as well as for non-optimally paired terms;
    \item[--] a larger penalty of $-1/3$ for terms in non-optimal triplets, which we simply bound by $-\frac14$.
\end{itemize}
The claim follows directly, with the following remarks: smart $2$-chains are automatically counted as smart pairs when maximizing $p_2$ (so that $p_2^*$ includes both $p_2$ and $m_2$); similarly, smart $3$-chains are counted as smart triplets when maximizing $p_3$, so that $p_3^*$ includes both $p_3$ and $m_3$.

The above argument proves the result holds when replacing $\gamma(k,k,k,k)$ terms by the corresponding $\widetilde{\gamma}(k,k,k,k)$. To conclude the proof (according to the strategy outlined in Remark \ref{strategy}), we check that this replacement does not affect the result, owing to the small constant difference (of order $d^{-2}$) between $\gamma$ and $\widetilde{\gamma}$. Indeed, replacing each $\widetilde{\gamma}(k,k,k,k)$ by the true value $\gamma(k,k,k,k)$ in the product $\gamma_1 \cdots \gamma_p$ amounts, after expanding, to adding a certain number of products where the same $\widetilde{\gamma}(k,k,k,k)$ terms are simply replaced by a constant of order $d^{-2}$. All of these terms (which number is bounded in a way that depends only on $p$) have smaller order than the estimate: indeed, the contribution of each gamma term to the exponent is between $-\frac32$ and $-\frac32-\frac14=-\frac{14}{8}$, so is always larger than $-2$.
\end{proof}

We note that Theorem \ref{thm:moment_bound} implies the following two estimates on the moments of any gamma term and correlation of two gamma terms, respectively:

\begin{corollary}\label{cor:moment_bound} For any values of $k,l,m,n$ we have the following moment estimates, when $d \rightarrow \infty$:
\begin{equation}%\label{gamma-2p}
\bbE |\gamma (k,l,m,n)|^{2p} = \bigO_p(d^{-3p}).
\end{equation}
\end{corollary}

\begin{corollary}\label{c:Special}
We have
\[
\bbE\left[\gamma(k,l,m,n)\gamma(k',l',m',n')\right] =\begin{cases} \sigma_{k,l,m,n}^2 d^{-3} + \bigO(d^{-4})&\qtq{if}\gamma(k',l',m',n')=\overline{\gamma(k,l,m,n)},\smallskip\\
\bigO(d^{-4})&\qtq{otherwise.}\end{cases}
\]
where $\sigma_{k,l,m,n}^2 > 0$, with \(\sigma_{k,l,m,n}^2 = 1\) if \(k,l,m,n\) are distinct.
\end{corollary}

\subsection{The semicircle law}\label{s:LSC} A fundamental result from random matrix theory is the \emph{semicircle law}, which is to say that, under generic assumptions, and in particular for GUE, the empirical distribution of the eigenvalues converges almost surely as \(N\to\infty\) to the semicircle distribution \eqref{Semicircle}.

In fact, we have some rather stronger results, the proof of which can be found in, e.g.,~\cite{benaychgeorges2018lectures}:

\begin{lemma}[Local semicircle law and eigenvalue rigidity]\label{l:LSC} Let the vector of eigenvalues, \(\lambda\), be distributed according to the density \eqref{Beta_ens}. Then:
\begin{itemize}[topsep=0pt, itemsep=\parskip]
\item We have the local semicircle law
\eq{LocalSC}{
\#\bigl\{\lambda_k\in I\bigr\} = (2N+1)\int_I\,d\sigma_{\mr{sc}} + \bigO_\prec(1),
}
uniformly for all intervals \(I\subseteq \R\).

\item If we define \(\nu(\kappa)\) for \(\kappa\in[-1,1]\) as in \eqref{nu} then we have the eigenvalue rigidity estimate
\eq{Rigid}{
\lambda_k = \nu\bigl(\tfrac kN\bigr) + \bigO_\prec \left(N^{-\frac23}\bigl(N+1-|k|\bigr)^{-\frac13}\right).
}
\end{itemize}
\end{lemma}

Our proof of Theorem~\ref{t:main} will rely on the following corollary of Lemma~\ref{l:LSC}:

\begin{corollary}\label{c:l}
Let \(\alpha\in \R\) and \(1\leq T\leq N\). Then,
\begin{align}
\sum_{k=-N}^N\frac1{\frac1T + |\lambda_k-\alpha|} &\prec N \log\bigl(16 + |\alpha|\bigr) ,\label{L1}
\end{align}
uniformly in \(\alpha\).
\end{corollary}
\bpf
Choose a partition of unity
\(
1 = \sum_{j\in \Z}\eta(\omega - j),
\)
where \(\eta\in \Test(\R)\) is a non-negative function supported in \([-2,2]\). We then set \(\eta_{j,N}(x) = \eta(Nx-j)\), which localizes to the region \(x = \frac jN + \bigO(\frac1N)\). Setting \(x_k = \lambda_k-\alpha\), we estimate the small, medium, and large choices of \(|j|\) as
\[
\frac1{\frac1T + |x_k|}\lesssim \sum_{|j|\leq 4\frac NT}T\,\eta_{j,N}(x_k) + \sum_{4\frac NT< |j|\leq (16 +|\alpha|)N}\frac N{|j|}\,\eta_{j,N}(x_k) + \sum_{|j|> (16+|\alpha|)N}\eta_{j,N}(x_k).
\]

Let us note that
\[
\sum_k \eta_{j,N}(x_k) \lesssim \#\Bigl\{\lambda_k\in \bigl[\alpha + \tfrac {j-2}N,\alpha + \tfrac{j+2}N\bigr]\Bigr\},
\]
From \eqref{LocalSC} and the fact that \(\sigma_{\mr{sc}}\) is supported on \([-2,2]\), we see that the contribution of the large choices of \(|j|\) should be negligible. Precisely,
\[
\sum_k\sum_{|j|> (16+|\alpha|)N}\eta_{j,N}(x_k) \lesssim \#\bigl\{|\lambda_k|\geq 8\bigr\}\prec 1.
\]
For the small and medium choices of \(|j|\), where \(|j|\leq (16+|\alpha|)N\), we again apply \eqref{LocalSC} to estimate
\[
\sum_k \eta_{j,N}(x_k)\prec 1.
\]

Combining these bounds, we may sum to obtain
\[
\sum_k\frac1{\frac1T + |x_k|}\prec N\Bigl[1 + \log T + \log\bigl(16 + |\alpha|\bigr)\Bigr],
\]
and the estimate \eqref{L1} now follows from the fact that, as \(1\leq T\leq N\), we have \(\log T\prec 1\).
\epf

\section{The leading order terms}\label{s:LOT}

In this section we investigate the leading order terms in our expansion and prove Proposition~\ref{p:LOT}. We subsequently denote the conditional expectation \(\bbE_\lambda[\,\cdot\,] = \bbE[\,\cdot\,|\,\lambda\,]\).

Recalling the definition of the profile \eqref{Profile}, the modulation \eqref{Osc}, and the interaction coefficient \eqref{Mom}, we may use the Duhamel formula \eqref{Duhamel} to obtain the expressions
\begin{align}
f_k\sbrack 0(t) &= a_k\sbrack 0 = A\left(\tfrac kN\right),\\
f_k\sbrack 1(t) &= -i\frac{\mu^2}N\sum_{\ell,m,n}\left(\int_0^t e^{is\Omega(k,\ell,m,n)}\,ds\right)\gamma(k,\ell,m,n)a_\ell\sbrack 0\overline{a_m\sbrack 0}a_n\sbrack 0,\label{fk1}\\
f_k\sbrack 2(t)&=-2\frac{\mu^4}{N^2}\sum_{\ell,m,n,o,p,q}\left(\int_0^t\int_0^se^{is\Omega(k,\ell,m,n)}e^{i\sigma\Omega(\ell,o,p,q)}\,d\sigma\,ds\right)\\
&\qquad\qquad\qquad\qquad\qquad\qquad\qquad\qquad\times\gamma(k,\ell,m,n)\gamma(\ell,o,p,q)a_o\sbrack 0\overline{a_p\sbrack 0}a_q\sbrack 0\overline{a_m\sbrack 0}a_n\sbrack 0\notag\\
&\quad + \frac{\mu^4}{N^2}\sum_{\ell,m,n,o,p,q}\left(\int_0^t\int_0^s e^{is\Omega(k,\ell,m,n)}e^{-i\sigma\Omega(m,o,p,q)}\,d\sigma\,\,ds\right)\notag\\
&\qquad\qquad\qquad\qquad\qquad\qquad\qquad\qquad\times\gamma(k,\ell,m,n)\overline{\gamma(m,o,p,q)}a_\ell\sbrack 0\overline{a_o\sbrack 0}a_p\sbrack 0\overline{a_q\sbrack 0}a_n\sbrack 0.\notag
\end{align}
We may then write
\begin{align}\label{ExpansionLOT}
&\bbE_\lambda\bigl|a_k\sbrack 0 + a_k\sbrack 1\bigr|^2 + 2\Re\bbE_\lambda\Bigl[\overline{a_k\sbrack 0}a_k\sbrack2\Bigr]\\
&\qquad = \bbE_\lambda \bigl|f_k\sbrack 0\bigr|^2 + \bbE_\lambda\bigl|f_k\sbrack 1\bigr|^2 + 2\Re\bbE_\lambda\Bigl[\overline{f_k\sbrack 0}f_k\sbrack 1\Bigr] + 2\Re\bbE_\lambda\Bigl[\overline{f_k\sbrack 0}f_k\sbrack 2\Bigr].\notag
\end{align}
Our first lemma rules out the quartic (in the initial data) terms in \eqref{ExpansionLOT}:

\begin{lemma}\label{l:mu2}
We have
\eq{mu2cancels}{
2\Re\bbE_\lambda\Bigl[\overline{f_k\sbrack 0}f_k\sbrack 1\Bigr] = 0.
}
\end{lemma}
\bpf
Observe that
\[
\overline{f_k\sbrack 0}f_k\sbrack 1 = -i\frac{\mu^2}N\sum_{\ell,m,n}\left(\int_0^t e^{is\Omega(k,\ell,m,n)}\,ds\right)\gamma(k,\ell,m,n)\overline{a_k\sbrack 0}a_\ell\sbrack 0\overline{a_m\sbrack 0}a_n\sbrack 0.
\]
Recalling \eqref{Weingarten}, we see that
\[
\bbE\gamma(k,\ell,m,n) = 0,
\]
unless \(k=\ell,m=n\) or \(k=n,\ell=m\). This leads to the simplification
\begin{align*}
\bbE_\lambda\Bigl[\overline{f_k\sbrack 0}f_k\sbrack 1\Bigr] &= -2i\frac{\mu^2 t}N\sum_\ell\bbE\gamma(k,k,\ell,\ell)\bigl|a_k\sbrack 0\bigr|^2\bigl|a_\ell\sbrack 0\bigr|^2+i\frac{\mu^2 t}N\bbE\gamma(k,k,k,k)\bigl|a_k\sbrack 0\bigr|^4.
\end{align*}
However, this term is clearly imaginary, so taking the real part we obtain \eqref{mu2cancels}.
\epf

Next, we apply Corollary~\ref{c:Special} to obtain the following:

\begin{lemma}\label{l:mu4}
We have
\begin{align}
\bbE_\lambda\Bigl[\bigl|f_k\sbrack 1\bigr|^2\Bigr]
&= \frac{2\mu^4}{N^2(2N+1)^3} \sum_{\substack{\ell,m,n}} \left|\int_0^t e^{is\Omega(k,\ell,m,n)}\,ds\right|^2\bigl|a_\ell\sbrack 0\bigr|^2\bigl|a_m\sbrack 0\bigr|^2\bigl|a_n\sbrack 0\bigr|^2 + \bigO\left(\tfrac{t}{T_\kin} \tfrac tN\right),\label{mu4-1}\\
2\Re\bbE_\lambda\Bigl[\overline{f_k\sbrack 0}f_k\sbrack 2\Bigr]&
= - \frac{8\mu^4}{N^2(2N+1)^3}\Re\sum_{\substack{\ell,m,n}}\left(\int_0^t\int_0^s e^{i(s-\sigma)\Omega(k,\ell,m,n)}\,d\sigma\,ds\right)\bigl|a_k\sbrack 0\bigr|^2\bigl|a_n\sbrack 0\bigr|^2\bigl|a_m\sbrack 0\bigr|^2\notag\\
&\quad + \frac{4\mu^4}{N^2(2N+1)^3}\Re\sum_{\substack{\ell,m,n}}\left(\int_0^t\int_0^s e^{i(s-\sigma)\Omega(k,\ell,m,n)}\,d\sigma\,ds\right)\bigl|a_k\sbrack 0\bigr|^2\bigl|a_\ell\sbrack 0\bigr|^2\bigl|a_n\sbrack 0\bigr|^2\label{mu4-2} \\
& \quad + \bigO\left(\tfrac{t}{T_\kin} \tfrac tN\right).\notag
\end{align}
\end{lemma}
\bpf
We present the computation in detail for \eqref{mu4-1}. The argument for \eqref{mu4-2} is identical.

Using the expression \eqref{fk1}, we compute that
\begin{align*}
\bigl|f_k\sbrack 1\bigr|^2 &= \frac{\mu^4}{N^2}\sum_{\ell,m,n,o,p,q}\left(\int_0^t e^{is\Omega(k,\ell,m,n)}\,ds\right)\left(\int_0^t e^{-is\Omega(k,o,p,q)}\,ds\right)\\
&\quad\qquad\qquad\qquad\qquad\qquad\qquad\qquad\times\gamma(k,\ell,m,n)\overline{\gamma(k,o,p,q)}a_\ell\sbrack 0\overline{a_m\sbrack 0}a_n\sbrack 0\overline{a_o\sbrack 0}a_p\sbrack 0\overline{a_q\sbrack 0}.
\end{align*}
Recalling \eqref{Weingarten}, we see that after taking the conditional expectation we only obtain a non-zero contribution when indices are paired, which implies that $\#\{k,\ell,m,n,o,p,q\} \leq 4$. In particular, we may decompose
\[
\bbE_\lambda\Bigl[\bigl|f_k\sbrack 1\bigr|^2\Bigr] = E_1 + E_2 + E_3 + E_4,
\]
where
\begin{align*}
E_r &= \frac{\mu^4}{N^2}\sum_{r=1}^4 \sum_{\substack{\ell,m,n,o,p,q\\\#\{k,\ell,m,n,o,p,q\} = r}}\left(\int_0^t e^{is\Omega(k,\ell,m,n)}\,ds\right)\left(\int_0^t e^{-is\Omega(k,o,p,q)}\,ds\right)\\
&\quad\qquad\qquad\qquad\qquad\qquad\qquad\qquad\times\bbE\Bigl[\gamma(k,\ell,m,n)\overline{\gamma(k,o,p,q)}\Bigr]a_\ell\sbrack 0\overline{a_m\sbrack 0}a_n\sbrack 0\overline{a_o\sbrack 0}a_p\sbrack 0\overline{a_q\sbrack 0}.
\end{align*}

In the cases \(r=1,2,3\), Corollary~\ref{c:Special} ensures that
\[
\bbE\Bigl[\gamma(k,\ell,m,n)\overline{\gamma(k,o,p,q)}\Bigr] = \bigO(N^{-3}).
\]
We then use the trivial bound
\[
\left|\left(\int_0^t e^{is\Omega(k,\ell,m,n)}\,ds\right)\left(\int_0^t e^{-is\Omega(k,o,p,q)}\,ds\right)\right|\leq t^2
\]
to estimate
\[
|E_r|\lesssim \frac t{T_{\kin}}\frac t N,
\]
where we note that as \(k\) is fixed the sum over \(\ell,m,n,o,p,q\) has dimension \(r-1\).

When \(r=4\) we have
\begin{align*}
E_4 &= \frac{2\mu^4}{N^2}\!\!\!\!\!\sum_{\substack{\ell,m,n\\\#\{k,\ell,m,n\} = 4}}\!\!\!\!\!\left|\int_0^t e^{is\Omega(k,\ell,m,n)}\,ds\right|^2\bbE|\gamma(k,\ell,m,n)|^2\bigl|a_\ell\sbrack 0\bigr|^2\bigl|a_m\sbrack 0\bigr|^2\bigl|a_n\sbrack 0\bigr|^2\\
&\quad + \frac{4\mu^4}{N^2}\!\!\!\!\!\sum_{\substack{\ell,m,p\\\#\{k,\ell,m,p\}=4}}\!\!\!\!\!\left(\int_0^t e^{is\Omega(k,\ell,m,m)}\,ds\right)\left(\int_0^t e^{-is\Omega(k,\ell,p,p)}\,ds\right)\\
&\qquad\qquad\qquad\qquad\qquad\qquad\qquad\qquad\times\bbE\left[\gamma(k,\ell,m,m)\overline{\gamma(k,\ell,p,p)}\right]\bigl|a_\ell\sbrack 0\bigr|^2\bigl|a_m\sbrack 0\bigr|^2\bigl|a_p\sbrack 0\bigr|^2.
\end{align*}
For distinct \(k,\ell,m,n\), Corollary~\ref{c:Special} gives us
\[
\bbE|\gamma(k,\ell,m,n)|^2 = (2N+1)^{-3} + \bigO(N^{-4}),
\]
whereas, for distinct \(k,\ell,m,p\), we have
\[
\bbE\left[\gamma(k,\ell,m,m)\overline{\gamma(k,\ell,p,p)}\right] = \bigO(N^{-4}).
\]
These yield
$$
E_4 = \frac{2\mu^4}{N^2(2N+1)^3} \!\!\!\!\!\sum_{\substack{\ell,m,n\\\#\{k,\ell,m,n\} = 4}}\!\!\!\!\! \left|\int_0^t e^{is\Omega(k,\ell,m,n)}\,ds\right|^2\bigl|a_\ell\sbrack 0\bigr|^2\bigl|a_m\sbrack 0\bigr|^2\bigl|a_n\sbrack 0\bigr|^2 + \bigO\left(\frac{t}{T_\kin} \frac tN\right),
$$
where we have used that
\[
\left|\int_0^t e^{is\Omega(k,\ell,m,n)}\,ds\right|^2\leq t^2
\]
to bound the remainder terms. Finally, we observe that the condition $\#\{k,\ell,m,n\} = 4$ in the sum can be relaxed, up to an error term of size $\bigO\left(\tfrac{t}{T_\kin} \tfrac tN\right)$. This gives the desired estimate.
\epf

We now apply Lemmas~\ref{l:mu2} and \ref{l:mu4} to the identity \eqref{ExpansionLOT} to obtain
\begin{align}\label{HalfWay}
\bbE_\lambda\bigl|a_k\sbrack 0 + a_k\sbrack 1\bigr|^2 + 2\Re\bbE_\lambda\Bigl[\overline{a_k\sbrack 0}a_k\sbrack2\Bigr] &=  \left|A\left( \tfrac{k}{N} \right)\right|^2 + I_k + \bigO\left(\tfrac{t}{T_\kin} \tfrac tN\right),
\end{align}
where, after integrating in time and symmetrizing the middle term,
\begin{align*}
I_k :&= \frac{2\mu^4}{N^2(2N+1)^3}\sum_{\ell,m,n}\left|\int_0^t e^{is\Omega(k,\ell,m,n)}\,ds\right|^2\bigl|a_\ell\sbrack 0\bigr|^2\bigl|a_m\sbrack 0\bigr|^2\bigl|a_n\sbrack 0\bigr|^2\\
&\quad - \frac{8\mu^4}{N^2(2N+1)^3}\Re\sum_{\ell,m,n}\left(\int_0^t\int_0^s e^{i(s-\sigma)\Omega(k,\ell,m,n)}\,d\sigma\,ds\right)\bigl|a_k\sbrack 0\bigr|^2\bigl|a_m\sbrack 0\bigr|^2\bigl|a_n\sbrack 0\bigr|^2\\
&\quad + \frac{4\mu^4}{N^2(2N+1)^3}\Re\sum_{\ell,m,n}\left(\int_0^t\int_0^s e^{i(s-\sigma)\Omega(k,\ell,m,n)}\,d\sigma\,ds\right)\bigl|a_k\sbrack 0\bigr|^2\bigl|a_\ell\sbrack 0\bigr|^2\bigl|a_n\sbrack 0\bigr|^2\\
&= \frac{8t^2\mu^4}{N^2(2N+1)^3}\sum_{\ell,m,n}\phi\bigl(t\Omega(k,\ell,m,n)\bigr)f\bigl(\tfrac kN,\tfrac\ell N,\tfrac mN,\tfrac nN\bigr),
\end{align*}
with
\begin{align*}
\phi(x) &=\begin{cases} \dfrac{\sin^2(x/2)}{x^2} &\qtq{if}x\neq 0,\smallskip\\ 1/4 &\qtq{if}x=0,\end{cases}
\end{align*}
and
\begin{align*}
f(k,\ell,m,n) &= |A(k)|^2|A(\ell)|^2|A(m)|^2|A(n)|^2\Bigl[\tfrac1{|A(k)|^2} - \tfrac1{|A(\ell)|^2} + \tfrac1{|A(m)|^2} - \tfrac1{|A(n)|^2}\Bigr].
\end{align*}

We now approximate the eigenvalues by their deterministic locations using the following:
\begin{lemma}\label{l:OmegaCancels}
For \(1\leq t\leq N\) we have
\eq{OmegaCancels}{
I_k = \frac{8t^2\mu^4}{N^2(2N+1)^3}\sum_{\ell,m,n}\phi\left(t\Theta\bigl(\tfrac kN,\tfrac \ell N,\tfrac mN,\tfrac nN\bigr)\right)f\bigl(\tfrac kN,\tfrac\ell N,\tfrac mN,\tfrac nN\bigr) + I_k^{\err,1},
}
where \(\Theta\) is defined as in \eqref{DeterministicPhase} and
\eq{err1}{
\tfrac1N\sum_k|I_k^{\err,1}|\prec \tfrac t{T_{\kin}}\tfrac tN.
}
\end{lemma}
\bpf
It suffices to prove that
\eq{Enough}{
\sup_{\alpha\in \R}\sum_k\left|\phi\left(\alpha + t\lambda_k\right) - \phi\left(\alpha + t\nu\bigl(\tfrac kN\bigr)\right)\right|\prec 1.
}
The estimate \eqref{err1} then follows from applying \eqref{Enough} to \(I^{\err,1}\).

Given \(\delta>0\) and \(N\geq 1\), let \(E_N\) be the event that for all \(k\in \llbracket-N,N\rrbracket\) we have
\eq{lk approx}{
\left|\lambda_k - \nu\bigl(\tfrac kN\bigr)\right|\leq N^{\frac \delta4-\frac23}\bigl(N+1 - |k|\bigr)^{-\frac13},
}
and recall from \eqref{Rigid} that for all \(K\geq 1\) and sufficiently large \(N\gg1\) we have
\eq{complement of this thing}{
\bbP[E_N^c]\lesssim_{\delta,K} N^{-K}.
}
Taking \(\omega\in E_N\) and \(\lambda = \lambda(\omega)\), we will prove that
\eq{calculus}{
\sup_\alpha\sum_k \left|\phi\left(\alpha + t\lambda_k\right) - \phi\left(\alpha + t\nu\bigl(\tfrac kN\bigr)\right)\right|\lesssim N^{\frac{3\delta} 4}.
}
This shows that for sufficiently large \(N\gg_\delta1\), the event that
\[
\sup_\alpha\sum_n\left|\phi\left(\alpha + t\lambda_k\right) - \phi\left(\alpha + t\nu\bigl(\tfrac kN\bigr)\right)\right|\leq N^\delta
\]
is a superset of \(E_N\), and the estimate \eqref{Enough} now follows from \eqref{complement of this thing}.

To prove \eqref{calculus}, we first use \eqref{nu} to show that \(\nu\colon[-1,1]\to[-2,2]\) is a continuous, strictly increasing, odd function, which is smooth on \((-1,1)\) and has graph between the lines \(\nu = \pi \kappa/2\) and \(\nu=2\kappa\). Consequently,
\eq{nu-sim}{
\nu(\kappa)\sim \kappa\qtq{for}|\kappa|\leq 1.
}
Moreover, from \eqref{nudiff} we have the asymptotics
\eq{nu-asy}{
2 - \nu(\kappa) = \bigl(\tfrac{3\pi}4(1 - \kappa)\bigr)^{\frac23} + \bigO\bigl((1-\kappa)^{\frac53}\bigr)\qtq{as}\kappa\nearrow 1.
}

We now consider the bulk of the spectrum, where \(|k|\leq N/2\). If \(|\alpha + t\nu(k/N)|\geq \frac12 t|\nu(k/N)|\) we apply the Mean Value Theorem to bound
\[
\left|\phi\left(\alpha + t\lambda_k\right) - \phi\left(\alpha + t\nu\bigl(\tfrac kN\bigr)\right)\right|\lesssim 
\frac{t N^{\frac{3\delta}4-1}}{\left(1 + \frac tN|k|\right)^2},
\]
where we have used \eqref{nu-sim} to show that
\[
|\phi'(x)|\lesssim \frac1{1 + |x|^2} \lesssim \frac{N^{\frac\delta 2}}{\left(1 + \tfrac tN|k|\right)^2}\qtq{whenever}x = \alpha+t\nu\left(\tfrac kN\right) + \bigO\left( t N^{\frac\delta4-1}\right).
\]

Conversely, if \(|\alpha + t\nu(k/N)| < \frac12 t|\nu(k/N)|\) we simply bound
\[
\left|\phi\left(\alpha + t\lambda_k\right) - \phi\left(\alpha + t\nu\bigl(\tfrac kN\bigr)\right)\right|\lesssim 
t N^{\frac\delta4-1}.
\]
However, from \eqref{nu-sim} we see that there are at most \(\bigO(N/t)\) choices of \(|k|\leq N/2\) for which \(|\alpha + t\nu(k/N)|< \frac12t|\nu(k/N)|\).

Combining these estimates, we may sum to obtain
\eq{case1}{
\sup_\alpha\sum_{|k|\leq \frac N2}\left|\phi\left(\alpha + t\lambda_k\right) - \phi\left(\alpha + t\nu\bigl(\tfrac kN\bigr)\right)\right|\lesssim N^{\frac{3\delta} 4}.
}

By symmetry, it remains to consider the right-hand edge of the spectrum, where \(\ell = N-k\) satisfies \(0\leq  \ell <N/2\). Proceeding as before, if \(|\alpha + t\nu(1-\ell/N)|\geq \frac12 t|\nu(1-\ell/N)|\) we may bound
\[
\left|\phi\left(\alpha + t\lambda_{N-\ell}\right) - \phi\left(\alpha + t\nu\bigl(1-\tfrac \ell N\bigr)\right)\right|\lesssim 
\frac{t N^{\frac{3\delta}4-\frac23}(1 + \ell)^{-\frac13}}{\left(1 + \frac{t^{3/2}}N\ell\right)^{\frac43} },
\]
where we have used \eqref{nu-asy} to show that
\[
|\phi'(x)|\lesssim \frac{N^{\frac\delta 2}}{\left(1 + \frac{t^{3/2}}N\ell\right)^{\frac43}}\qtq{whenever}x = \alpha + t\nu\left(1-\tfrac \ell N\right) + \bigO\left( t N^{\frac\delta4-\frac23}(1+\ell)^{-\frac13}\right).
\]
Conversely, another application of \eqref{nu-asy} shows there are at most \(\bigO(N/t^{3/2})\) choices of \(\ell\) for which \(|\alpha + t\nu(1-\ell/N)|< \frac12 t|\nu(1-\ell/N)|\), where we have
\[
\left|\phi\left(\alpha + t\lambda_{N-\ell}\right) - \phi\left(\alpha + t\nu\bigl(1-\tfrac \ell N\bigr)\right)\right|\lesssim 
t N^{\frac\delta2-\frac23}(1 + \ell)^{-\frac13}.
\]
Combining these estimates, we may sum to obtain
\eq{case2}{
\sup_\alpha\sum_{0\leq \ell<N/2}\left|\phi\left(\alpha + t\lambda_{N-\ell}\right) - \phi\left(\alpha + t\nu\bigl(1-\tfrac \ell N\bigr)\right)\right|\lesssim N^{\frac{3\delta} 4}.
}

The estimate \eqref{calculus} now follows from combining \eqref{case1} and \eqref{case2}.
\epf

Our next task is to pass from the Riemann sum to an integral:

\begin{lemma}\label{l:Rsum}
For \(1\leq t\leq N\) we have
\begin{align}
I_k &= \frac{8t^2N\mu^4}{(2N+1)^3}\int_{[-1,1]^3}\phi\bigl(t\Theta(\tfrac kN,\ell,m,n)\bigr)f\bigl(\tfrac kN,\ell,m,n\bigr)\,d\ell\,dm\,dn + I^{\err,1}_k + I^{\err,2}_k,\label{Rsum}
\end{align}
where \(I^{\err,1}\) is defined as in \eqref{OmegaCancels} and we have the estimate
\eq{err2}{
|I_k^{\err,2}|\lesssim \tfrac t{T_{\kin}}\tfrac t N.
}
\end{lemma}
\bpf
Let us first show that if \(h\in \Cont[-1,1]\cap\Cont^1(-1,1)\) has finite norm
\[
\|h\|_X := \sup_{|\kappa|\leq 1}\left\{|h(\kappa)| + \bigl(1 - |\kappa|\bigr)^{\frac13}|h'(\kappa)|\right\},
\]
then we may bound
\eq{Riemann}{
\left|\tfrac1{N}\sum_{k=-N}^N h\left(\tfrac kN\right) - \int_{-1}^1 h(\kappa)\,d\kappa\right|\lesssim \tfrac1N\|h\|_X.
}

Denoting the interval \(J_k = [\frac kN-\frac1{2N},\frac kN + \frac1{2N}]\), we apply the Mean Value Theorem to estimate
\[
|h(\kappa) - h(\tfrac kN)|\lesssim \|h\|_X \left(1-\tfrac1{2N} - |\kappa|\right)^{-\frac13} \tfrac 1N\qtq{for}\kappa\in J_k\qtq{with}|k|\leq N-1.
\]
As a consequence,
\[
\int_{-1}^1 h(\kappa)\,d\kappa = \sum_{k=1-N}^{N-1} \int_{-1}^1h(\kappa)\bbo_{J_k}(\kappa)\,d\kappa + \bigO\left(\tfrac1N\|h\|_X\right) = \tfrac1N\sum_{k=1-N}^{N-1}h\bigl(\tfrac kN\bigr) + \bigO\left(\tfrac1N\|h\|_X\right),
\]
from which we obtain \eqref{Riemann}.

Now consider
\[
h(\ell):= \sum_{m,n}\phi\bigl(t\Theta(\tfrac kN,\ell,\tfrac mN,\tfrac nN)\bigr)f\bigl(\tfrac kN,\ell,\tfrac mN,\tfrac nN\bigr).
\]
From \eqref{nudiff} and \eqref{nu-asy}, we see that
\[
|\nu'(\kappa)|\lesssim (1 - |\kappa|)^{-\frac13}.
\]
As \(\phi,\phi'\) are bounded, it is then clear that
\[
\|h\|_X\lesssim N^2
\]
and hence we may apply \eqref{Riemann} to obtain
\begin{align*}
&\frac{8t^2\mu^4}{N^2(2N+1)^3}\sum_{\ell,m,n} \phi\bigl(t\Theta(\tfrac kN,\tfrac\ell N,\tfrac mN,\tfrac nN)\bigr)f\bigl(\tfrac kN,\tfrac\ell N,\tfrac mN,\tfrac nN\bigr)\\
&\qquad= \frac{8t^2\mu^4}{N(2N+1)^3}\sum_{m,n}\int_{-1}^1 \phi\bigl(t\Theta(\tfrac kN,\ell,\tfrac mN,\tfrac nN)\bigr)f\bigl(\tfrac kN,\ell,\tfrac mN,\tfrac nN\bigr)\,d\ell + \bigO\left(\tfrac t{T_{\kin}}\tfrac tN\right).
\end{align*}

The remaining terms in \(I^{\err,2}\) are estimated similarly to obtain \eqref{err2}.
\epf

Finally, we take the limit as \(t\to\infty\) using the following lemma:
\begin{lemma}\label{l:Delta}
Let \(h\in \Cont^{0,\frac12}[-2,2]\). Then for \(t>0\) we have
\eq{Delta}{
\left|\int_{-2}^2 \phi(tx) h(x)\,dx - \frac \pi{2t} h(0)\right|\lesssim  t^{-\frac32}\|h\|_{\Cont^{0,\frac12}}.
}
\end{lemma}
\bpf
We write
\begin{align*}
\int_{-2}^2 \phi(tx) h(x)\,dx = \frac 1t\int_\R \phi(x) h(0)\,dx + \frac 1t\int_{|x|\leq 2t} \!\!\!\phi(x) \left[h(\tfrac xt) - h(0)\right]\,dx - \frac1t\int_{|x|>2t}\!\!\!\phi(x)h(0)\,dx.
\end{align*}
For the first term, we compute that
\[
\frac 1t\int_\R \phi(x) h(0)\,dx = \frac \pi{2t} h(0).
\]
For the second and third terms, we estimate
\[
\left|h(\tfrac xt) - h(0)\right|\lesssim \|h\|_{\Cont^{0,\frac12}}\sqrt{\tfrac{|x|}t}\qtq{and}|h(0)|\lesssim  \|h\|_{\Cont^{0,\frac12}}.
\]
Using that
\[
|\phi(x)|\lesssim \frac1{1 + |x|^2},
\]
we may bound
\begin{align*}
\left|\frac 1t\int_{|x|\leq 2t} \phi(x) \left[h(\tfrac xt) - h(0)\right]\,dx\right| \lesssim t^{-\frac32}\|h\|_{\Cont^{0,\frac12}}\qtq{and}\left|\frac1t\int_{|x|>2t}\phi(x)h(0)\,dx\right| \lesssim t^{-2}\|h\|_{\Cont^{0,\frac12}},
\end{align*}
which completes the proof of \eqref{Delta}.
\epf

Finally, we complete the:

\bpf[Proof of Proposition~\ref{p:LOT}]
Using \eqref{nudiff}, we see that \(\nu^{-1}\in \Cont^1[-2,2]\) satisfies
\[
\frac d{dx}\nu^{-1}(x) = \tfrac1\pi\sqrt{4-x^2}\in \Cont^{0,\frac12}[-2,2].
\]
As a consequence, we may apply Lemma~\ref{l:Delta} to bound
\begin{align*}
&\int_{-1}^1 \phi\Bigl(t\Theta\bigl(\tfrac kN,\ell,m,n\bigr)\Bigr)f\bigl(\tfrac kN,\ell,m,n\bigr)\,dn\\
&\qquad= \frac1\pi\int_{-2}^2\phi\Bigl(t\bigl(\nu(\tfrac kN) - \nu(\ell) + \nu(m) - x\bigr)\Bigr)f\bigl(\tfrac kN,\ell,m,\nu^{-1}(x)\bigr)\sqrt{4 - x^2}\,dx\\
&\qquad= \frac1{2 t}\int_{-2}^2\delta\bigl(\nu(\tfrac kN) - \nu(\ell) + \nu(m) - x\bigr)f\bigl(\tfrac kN,\ell,m,\nu^{-1}(x)\bigr)\sqrt{4 - x^2}\,dx + \bigO\left(t^{-\frac32}\right)\\
&\qquad = \frac \pi{2t}\int_{-1}^1\delta\Bigl(\Theta\bigl(\tfrac kN,\ell,m,n\bigr)\Bigr)f\bigl(\tfrac kN,\ell,m,n\bigr)\,dn + \bigO\left(t^{-\frac32}\right).
\end{align*}
Combining this with Lemmas~\ref{l:OmegaCancels} and~\ref{l:Rsum}, we have
\[
I_k = \tfrac t{T_\kin}\cC\bigl[|A|^2\bigr]\bigl(\tfrac kN\bigr) + I_k^{\err,1} + I_k^{\err,2} + I_k^{\err,3},
\]
where \(I_k^{\err,1}\) is defined as in \eqref{OmegaCancels}, \(I_k^{\err,2}\) as in \eqref{Rsum}, and the final remainder term satisfies
\[
|I_k^{\err,3}|\lesssim \frac t{T_{\kin}} \left(\frac1{\sqrt t} + \frac1N\right).
\]
In particular, the estimates \eqref{err1} and \eqref{err2} show that for any \(\delta>0\) we have
\eq{err3}{
\sum_{j=1}^3 \frac1N\|\bbE I^{\err,j}\|_{\ell^1}\lesssim_\delta \frac t{T_{\kin}}\left(\frac1{\sqrt t} + N^\delta\frac tN\right).
}

To complete the proof, we take the expectation of \eqref{HalfWay} to get
\begin{align*}
\bbE\bigl|a_k\sbrack 0 + a_k\sbrack 1\bigr|^2 + 2\Re\bbE\Bigl[\overline{a_k\sbrack 0}a_k\sbrack 2\Bigr] &=  \left|A\left( \tfrac{k}{N} \right)\right|^2 + \bbE I_k + \bigO\left(\tfrac{t}{T_\kin} \tfrac tN \right).
\end{align*}
We then apply the estimate \eqref{err3} to complete the proof of \eqref{LOT}.
\epf

\section{The approximate solution}\label{s:app}

In this section we estimate the Duhamel iterates \(a\sbrack m\) and prove Proposition~\ref{p:app}. Applying Lemma~\ref{l:Profile} and taking \(\phi_{T,\pm}(t) = \frac{e^{\pm t/T}}{2\pi}\chi_T(t)\), we compute that
\begin{align}\label{Nala}
\|\chi_T\bbo_{\R_\pm }a\sbrack m\|_{X^b_T}^2 &= \frac{\mu^{4m}}{N^{2m}}\sum_{\ell',\cK',\ell,\cK}\left[\sum_k\Delta(k,\cK',\ell')\Delta(k,\cK,\ell)\right]\prod_{r=1}^{m}\overline{\gamma_{r,\ell_r'}(\cK')}\gamma_{r,\ell_r}(\cK)\\
&\quad \qquad\times\int_{\R^3} \frac{\left(\tfrac1T + |\tau|\right)^{2b}\overline{\widehat\phi_{T,\pm}(\tau - \alpha')}\widehat\phi_{T,\pm}(\tau - \alpha)}{\prod_{r=0}^m \left(\alpha' + \omega_r(\cK',\ell') \mp \frac{i}{T}\right) \left(\alpha + \omega_r(\cK,\ell) \pm \frac{i}{T}\right)}  \,d\alpha'\,d\alpha\,d\tau\notag\\
&\quad \qquad\qquad\qquad\qquad\qquad\times\overline{\bA\bigl(k_{0,1}',\dots,k_{0,2m+1}'\bigr)}\bA\bigl(k_{0,1},\dots,k_{0,2m+1}\bigr) .\notag
\end{align}
We subsequently consider estimates for only the ``\(+\)'' case of \eqref{Nala} and write \(\phi_T = \phi_{T,+}\). The corresponding estimates in the ``\(-\)'' case are identical.

To simplify the subsequent expressions, we denote \(\ttL = (\ell',\ell)\), \(\ttK = (\cK',\cK)\), and
\begin{align*}
\Identifications(\ttL,\ttK) &= \sum_k\Delta(k,\cK',\ell')\Delta(k,\cK,\ell),\\
\Interactions(\ttL,\ttK) &= \frac{\mu^{4m}}{N^{2m}}\prod_{r=1}^{m}\overline{\gamma_{r,\ell_r'}(\cK')}\gamma_{r,\ell_r}(\cK),\\
\Oscillations(\ttL,\ttK) & =\int_{\R^3}\frac{\left(\tfrac1T + |\tau|\right)^{2b} \overline{\widehat\phi_T(\tau - \alpha')}\widehat\phi_T(\tau-\alpha)}{\prod_{r=0}^m \left(\alpha' + \omega_r(\cK',\ell') - \frac{i}{T} \right)\left(\alpha + \omega_r(\cK,\ell) + \frac{i}{T}\right)}  \,d\alpha'\,d\alpha\,d\tau,\\
\Data(\ttK) &=  \overline{\bA\bigl(k_{0,1}',\dots,k_{0,2m+1}'\bigr)}\bA\bigl(k_{0,1},\dots,k_{0,2m+1}\bigr).
\end{align*}
This allows us to write \eqref{Nala} in the compact form
\[
\|\chi_T\bbo_{\R_+ }a\sbrack m\|_{X^b_T}^2 = \sum_{\ttL,\ttK}\Identifications\cdot \Interactions \cdot \Oscillations\cdot \Data.
\]

Let us now fix the choice of interaction history \(\ttL\) and consider the dimension of the sum over the frequencies \(\ttK\). On the support of \(\Identifications(\ttL,\ttK)\) there can be at most \(6m+1\) distinct frequencies:
\begin{itemize}
\item The \(4m+2\) input frequencies \(k_{0,1}',\dots,k_{0,2m+1}',k_{0,1},\dots,k_{0,2m+1}\);
\item The \(2m-2\) internal frequencies \(k_{1,\ell_1'}',\dots,k_{m-1,\ell_{m-1}'}',k_{1,\ell_1},\dots,k_{m-1,\ell_{m-1}}\);
\item The output frequency \(k_{m,1}'=k_{m,1} \).
\end{itemize}

This is best understood by introducing an \emph{interaction graph} \(G = G(\ttL)\) obtained from the Feynman diagrams (as defined in Section~\ref{s:Feyn}) for \(a\sbrack m\), corresponding to the interaction history \(\ell\), and \(\overline{a\sbrack m}\), corresponding to the interaction history \(\ell'\). To construct \(G\), let us assume that the Feynman diagram for \(a\sbrack m\) is on the right and the Feynman diagram for \(\overline{a\sbrack m}\) is on the left. Objects in the graph of \(\overline{a\sbrack m}\) are denoted with primes, so the vertices are \(v_{r,j}'\), etc. We now construct the interaction graph \(G\) by first identifying the output vertices \(v_{m+1,1}'\) and \(v_{m+1,1}\) (denoted by \tikz[baseline=-2.4,scale=0.15]{\node[outdot] {};} in each Feynman diagram), and then removing all vertices connected to exactly two edges. The graph \(G\) will then have \(6m+1\) edges, each corresponding to one of the \(6m+1\) possible distinct frequencies. An example is given in Figure~\ref{f:deg1}.

\begin{figure}[h!]
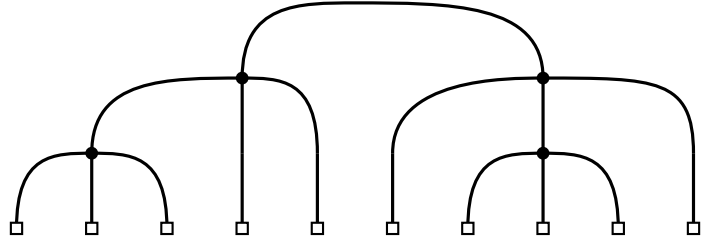

\begin{center}
\tikz{
\draw[very thick] (1,0)--(1,1) {};
\draw[very thick] (1,1)--(1,2) {};
\draw[very thick] (-1,0)--(-1,1) {};
\draw[very thick] (2,0)--(2,1) {};
\draw[very thick] (-2,0) .. controls (-2,1) and (-1.5,1) .. (-1,1);
\draw[very thick] (0,0) .. controls (0,1) and (-.5,1) .. (-1,1);
\draw[very thick] (2,1) .. controls (2,2) and (1.5,2) .. (1,2);
\draw[very thick] (-1,1) .. controls (-1,2) and (0,2) .. (1,2);
\draw[very thick] (5,0)--(5,1) {};
\draw[very thick] (3,0)--(3,1) {};
\draw[very thick] (7,0)--(7,1) {};
\draw[very thick] (5,1)--(5,2) {};
\draw[very thick] (4,0) .. controls (4,1) and (4.5,1) .. (5,1);
\draw[very thick] (6,0) .. controls (6,1) and (5.5,1) .. (5,1);
\draw[very thick] (7,1) .. controls (7,2) and (6.5,2) .. (5,2);
\draw[very thick] (3,1) .. controls (3,2) and (4.5,2) .. (5,2);
\draw[very thick] (5,2) .. controls (5,3) and (3.75,3) .. (2.5,3) {};
\draw[very thick] (1,2) .. controls (1,3) and (1.75,3) .. (2.5,3) {};
\node[indot] at (-2,0) {};
\node[indot] at (-1,0) {};
\node[indot] at (0,0) {};
\node[indot] at (1,0) {};
\node[indot] at (2,0) {};
\node[indot] at (3,0) {};
\node[indot] at (4,0) {};
\node[indot] at (5,0) {};
\node[indot] at (6,0) {};
\node[indot] at (7,0) {};
\node[dot] at (-1,1) {};
\node[dot] at (1,2) {};
\node[dot] at (5,1) {};
\node[dot] at (5,2) {};
}
\end{center}
\caption{The interaction graph \(G\) in the case \(m=2\) with interaction history \(\ttL = (\ell',\ell) = \bigl((1,1),(2,1)\bigr)\).}\label{f:deg1}
\end{figure}

We will prove Proposition~\ref{p:app} using Markov's inequality. As a consequence, we consider the expected value of 
\[
\|\chi_T\bbo_{\R_+ }a\sbrack m\|_{X^b_T}^{2\ttn}\qtq{for an integer}\ttn\geq 1.
\]
To write this sum in a compact form, we denote the interaction history and frequencies of the \(\tti\)\textsuperscript{th} copy of \(\|\chi_T\bbo_{\R_+ }a\sbrack m\|_{X^b_T}^2\) in this product by \(\ttL\sprack\tti\) and \(\ttK\sprack\tti\), respectively. We then take \(\vec\ttL = (\ttL\sprack1,\dots,\ttL\sprack\ttn)\) and \(\vec\ttK = (\ttL\sprack1,\dots,\ttK\sprack\ttn)\).

Setting \(\Identifications\sprack\tti = \Identifications(\ttL\sprack\tti,\ttK\sprack\tti)\), etc., we obtain the expression
\eq{Crostata}{
\|\chi_T\bbo_{\R_+ }a\sbrack m\|_{X^b_T}^{2\ttn} = \sum_{\vec \ttL,\vec\ttK} \prod_{\tti=1}^\ttn \Identifications\sprack\tti\cdot\Interactions\sprack \tti\cdot\Oscillations\sprack\tti\cdot\Data\sprack\tti.
}

Let \(G\sprack\tti\) to be the interaction graph for the \(\tti\)\textsuperscript{th} copy of \(\|\chi_T\bbo_{\R_+ }a\sbrack m\|_{X^b_T}^2\) in this product, with interaction history \(\ttL\sprack\tti\). For each choice of interaction histories \(\vec\ttL\) we take the interaction graph \(G\) for the product \(\|\chi_T\bbo_{\R_+ }a\sbrack m\|_{X^b_T}^{2n}\) to be the union of the interaction graphs \(G\sprack\tti\), which we assume are ordered so that \(G\sprack1\) is on the left and \(G\sprack\ttn\) is on the right: see Figure~\ref{f:union} for an example. The graph \(G = \bigcup_{\tti=1}^\ttn G\sprack\tti\) will then have \(\ttn(6m+1)\) edges, each corresponding to one of the \(\ttn(6m+1)\) possible distinct frequencies.

\begin{figure}[h!]
\begin{center}
\tikz[scale=.75]{
\draw[very thick] (1,0)--(1,1) {};
\draw[very thick] (1,1)--(1,2) {};
\draw[very thick] (-1,0)--(-1,1) {};
\draw[very thick] (2,0)--(2,1) {};
\draw[very thick] (-2,0) .. controls (-2,1) and (-1.5,1) .. (-1,1);
\draw[very thick] (0,0) .. controls (0,1) and (-.5,1) .. (-1,1);
\draw[very thick] (2,1) .. controls (2,2) and (1.5,2) .. (1,2);
\draw[very thick] (-1,1) .. controls (-1,2) and (0,2) .. (1,2);
\draw[very thick] (5,0)--(5,1) {};
\draw[very thick] (3,0)--(3,1) {};
\draw[very thick] (7,0)--(7,1) {};
\draw[very thick] (5,1)--(5,2) {};
\draw[very thick] (4,0) .. controls (4,1) and (4.5,1) .. (5,1);
\draw[very thick] (6,0) .. controls (6,1) and (5.5,1) .. (5,1);
\draw[very thick] (7,1) .. controls (7,2) and (6.5,2) .. (5,2);
\draw[very thick] (3,1) .. controls (3,2) and (4.5,2) .. (5,2);
\draw[very thick] (5,2) .. controls (5,3) and (3.75,3) .. (2.5,3) {};
\draw[very thick] (1,2) .. controls (1,3) and (1.75,3) .. (2.5,3) {};
\node[indot] at (-2,0) {};
\node[indot] at (-1,0) {};
\node[indot] at (0,0) {};
\node[indot] at (1,0) {};
\node[indot] at (2,0) {};
\node[indot] at (3,0) {};
\node[indot] at (4,0) {};
\node[indot] at (5,0) {};
\node[indot] at (6,0) {};
\node[indot] at (7,0) {};
\node[dot] at (-1,1) {};
\node[dot] at (1,2) {};
\node[dot] at (5,1) {};
\node[dot] at (5,2) {};
}
\quad
\tikz[scale=.75]{
\begin{scope}[xscale=-1]
\draw[very thick] (1,0)--(1,1) {};
\draw[very thick] (1,1)--(1,2) {};
\draw[very thick] (-1,0)--(-1,1) {};
\draw[very thick] (2,0)--(2,1) {};
\draw[very thick] (-2,0) .. controls (-2,1) and (-1.5,1) .. (-1,1);
\draw[very thick] (0,0) .. controls (0,1) and (-.5,1) .. (-1,1);
\draw[very thick] (2,1) .. controls (2,2) and (1.5,2) .. (1,2);
\draw[very thick] (-1,1) .. controls (-1,2) and (0,2) .. (1,2);
\draw[very thick] (5,0)--(5,1) {};
\draw[very thick] (3,0)--(3,1) {};
\draw[very thick] (7,0)--(7,1) {};
\draw[very thick] (5,1)--(5,2) {};
\draw[very thick] (4,0) .. controls (4,1) and (4.5,1) .. (5,1);
\draw[very thick] (6,0) .. controls (6,1) and (5.5,1) .. (5,1);
\draw[very thick] (7,1) .. controls (7,2) and (6.5,2) .. (5,2);
\draw[very thick] (3,1) .. controls (3,2) and (4.5,2) .. (5,2);
\draw[very thick] (5,2) .. controls (5,3) and (3.75,3) .. (2.5,3) {};
\draw[very thick] (1,2) .. controls (1,3) and (1.75,3) .. (2.5,3) {};
\node[indot] at (-2,0) {};
\node[indot] at (-1,0) {};
\node[indot] at (0,0) {};
\node[indot] at (1,0) {};
\node[indot] at (2,0) {};
\node[indot] at (3,0) {};
\node[indot] at (4,0) {};
\node[indot] at (5,0) {};
\node[indot] at (6,0) {};
\node[indot] at (7,0) {};
\node[dot] at (-1,1) {};
\node[dot] at (1,2) {};
\node[dot] at (5,1) {};
\node[dot] at (5,2) {};
\end{scope}
}
\end{center}
\caption{The interaction graph for \(m=2\), \(\ttn=2\), with \(\vec \ttL = (\ttL\sprack1,\ttL\sprack2) =\bigl( ({\ell\sprack1}',{\ell\sprack 1}),({\ell\sprack2}',{\ell\sprack2})\bigr) = \Bigl(\bigl((1,1), (2,1)\bigr),\bigl((2,1),(3,1)\bigr)\Bigr)\).}\label{f:union}
\end{figure}

Given a fixed choice of \(\vec\ttL\), we define a \emph{coloring} of the corresponding interaction graph \(G = G(\vec\ttL)\) to be an assignment of a color to each edge. We allow adjacent edges to have the same color and consider two colorings to be the same if they are identical up to permuting the colors themselves. Note that any coloring may have at most \(\ttn(6m+1)\) distinct colors.

\begin{figure}[h!]
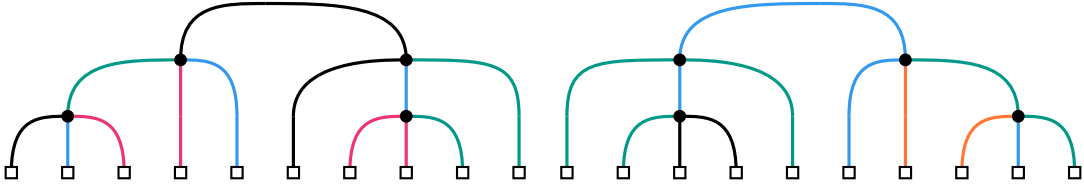

\begin{center}
\tikz[scale=.75]{
\draw[very thick,col6] (1,0)--(1,1) {};
\draw[very thick,col6] (1,1)--(1,2) {};
\draw[very thick,col5] (-1,0)--(-1,1) {};
\draw[very thick,col5] (2,0)--(2,1) {};
\draw[very thick,col1] (-2,0) .. controls (-2,1) and (-1.5,1) .. (-1,1);
\draw[very thick,col6] (0,0) .. controls (0,1) and (-.5,1) .. (-1,1);
\draw[very thick,col5] (2,1) .. controls (2,2) and (1.5,2) .. (1,2);
\draw[very thick,col3] (-1,1) .. controls (-1,2) and (0,2) .. (1,2);
\draw[very thick,col6] (5,0)--(5,1) {};
\draw[very thick,col1] (3,0)--(3,1) {};
\draw[very thick,col3] (7,0)--(7,1) {};
\draw[very thick,col5] (5,1)--(5,2) {};
\draw[very thick,col6] (4,0) .. controls (4,1) and (4.5,1) .. (5,1);
\draw[very thick,col3] (6,0) .. controls (6,1) and (5.5,1) .. (5,1);
\draw[very thick,col3] (7,1) .. controls (7,2) and (6.5,2) .. (5,2);
\draw[very thick,col1] (3,1) .. controls (3,2) and (4.5,2) .. (5,2);
\draw[very thick,col1] (5,2) .. controls (5,3) and (3.75,3) .. (2.5,3) {};
\draw[very thick,col1] (1,2) .. controls (1,3) and (1.75,3) .. (2.5,3) {};
\node[indot] at (-2,0) {};
\node[indot] at (-1,0) {};
\node[indot] at (0,0) {};
\node[indot] at (1,0) {};
\node[indot] at (2,0) {};
\node[indot] at (3,0) {};
\node[indot] at (4,0) {};
\node[indot] at (5,0) {};
\node[indot] at (6,0) {};
\node[indot] at (7,0) {};
\node[dot] at (-1,1) {};
\node[dot] at (1,2) {};
\node[dot] at (5,1) {};
\node[dot] at (5,2) {};
}
\quad
\tikz[scale=.75]{
\begin{scope}[xscale=-1]
\draw[very thick,col4] (1,0)--(1,1) {};
\draw[very thick,col4] (1,1)--(1,2) {};
\draw[very thick,col5] (-1,0)--(-1,1) {};
\draw[very thick,col5] (2,0)--(2,1) {};
\draw[very thick,col3] (-2,0) .. controls (-2,1) and (-1.5,1) .. (-1,1);
\draw[very thick,col4] (0,0) .. controls (0,1) and (-.5,1) .. (-1,1);
\draw[very thick,col5] (2,1) .. controls (2,2) and (1.5,2) .. (1,2);
\draw[very thick,col3] (-1,1) .. controls (-1,2) and (0,2) .. (1,2);
\draw[very thick,col1] (5,0)--(5,1) {};
\draw[very thick,col3] (3,0)--(3,1) {};
\draw[very thick,col3] (7,0)--(7,1) {};
\draw[very thick,col5] (5,1)--(5,2) {};
\draw[very thick,col1] (4,0) .. controls (4,1) and (4.5,1) .. (5,1);
\draw[very thick,col3] (6,0) .. controls (6,1) and (5.5,1) .. (5,1);
\draw[very thick,col3] (7,1) .. controls (7,2) and (6.5,2) .. (5,2);
\draw[very thick,col3] (3,1) .. controls (3,2) and (4.5,2) .. (5,2);
\draw[very thick,col5] (5,2) .. controls (5,3) and (3.75,3) .. (2.5,3) {};
\draw[very thick,col5] (1,2) .. controls (1,3) and (1.75,3) .. (2.5,3) {};
\node[indot] at (-2,0) {};
\node[indot] at (-1,0) {};
\node[indot] at (0,0) {};
\node[indot] at (1,0) {};
\node[indot] at (2,0) {};
\node[indot] at (3,0) {};
\node[indot] at (4,0) {};
\node[indot] at (5,0) {};
\node[indot] at (6,0) {};
\node[indot] at (7,0) {};
\node[dot] at (-1,1) {};
\node[dot] at (1,2) {};
\node[dot] at (5,1) {};
\node[dot] at (5,2) {};
\end{scope}
}
\end{center}
\caption{A coloring of the interaction graph from Figure~\ref{f:union} by \(5\) distinct colors.}\label{f:deg2}
\end{figure}

We identify each coloring \(C\) of \(G\) with the set of all frequencies \(\vec\ttK\) for which edges with the same color have the same frequency and edges with distinct colors have distinct frequencies. In this way, we may treat each distinct color as a distinct frequency. Moreover, we have
\[
\left\{\vec\ttK:\prod_{\tti=1}^\ttn\Identifications\sprack\tti\neq 0\right\} = \bigcup\Bigl\{C:C\text{ is a coloring of }G(\vec\ttL)\Bigr\}.
\]

Now fix a coloring \(C\) of \(G\) by \(p\in \llbracket1,\ttn(6m+1)\rrbracket\) distinct colors. Enumerate the \(p\) distinct frequencies (each corresponding to a distinct color) by \(K_1,\dots,K_p\) and enumerate the \(2m\ttn\) interaction vertices (denoted by \tikz[baseline=-2.4,scale=0.15]{\node[dot] {};}) in \(G\) in the order
\[
(V_1,\dots,V_{2mn}) = \left(v_{1,\ell_1}\sprack\ttn,\dots,v_{m,\ell_m}\sprack\ttn,{v_{1,\ell_1'}\sprack\ttn\!}',\dots,{v_{m,\ell_m'}\sprack\ttn\!\!\!\!}'\,\,\,,\dots,v_{1,\ell_1}\sprack1,\dots,v_{m,\ell_m}\sprack1,{v_{1,\ell_1'}\sprack1\!}',\dots,{v_{m,\ell_m'}\sprack1\!\!\!\!}'\,\,\,\right),
\]
which corresponds to working up the rightmost tree from bottom to top, then the second tree from the right from bottom to top, and so on. Now review each interaction vertex in the order \(V_1,V_2,\dots,V_{2m\ttn}\) and examine the frequencies of the three edges immediately below it. Having examined all such edges, examine the frequencies of the remaining edges from right to left. We write \(K_j\leq V_s\) if \(K_j\) does not appear in this process after we have examined the three edges immediately below \(V_s\). Otherwise, we write \(K_j>V_s\).

Mirroring our enumeration of the interaction vertices, we enumerate the modulations in the order
\begin{align*}
(\varpi_1,\dots,&\varpi_{2m\ttn})\\
= \biggl( & \omega_0(\cK\sprack\ttn,\ell\sprack\ttn),\dots,\omega_{m-1}(\cK\sprack\ttn,\ell\sprack\ttn),\omega_0({\cK\sprack\ttn}',{\ell\sprack\ttn}'),\dots,\omega_{m-1}({\cK\sprack\ttn}',{\ell\sprack\ttn}'),\dots\\
&\quad\dots,\omega_0(\cK\sprack1,\ell\sprack1),\dots,\omega_{m-1}(\cK\sprack1,\ell\sprack1),\omega_0({\cK\sprack1}',{\ell\sprack1}'),\dots,\omega_{m-1}({\cK\sprack1}',{\ell\sprack1}')\biggr).
\end{align*}
We recall here that by \eqref{thetaj} the modulations \(\omega_m = 0\) and hence we do not include these in our list.

From \eqref{thetaj}, we see that each \(\varpi_r\) can be written as a linear combination of the frequencies \(K_j\). By inspection of \eqref{thetaj}, we immediately obtain the following consequence of our enumeration procedure:

\begin{lemma}\label{l:independence}
Fix an interaction history \(\vec\ttL\) and coloring \(C\) of the associated interaction graph \(G(\vec\ttL)\). Enumerating the distinct frequencies, vertices, and modulations as above, if \(K_j\leq V_s\) and \(r>s\) then the modulation \(\varpi_r\) does not depend on the frequency \(K_j\).
\end{lemma}

We define the \emph{degree} of the vertex \(V_s\) to be the number of distinct frequencies \(K_j\) satisfying \(V_{s-1}<K_j\leq V_s\), where all frequencies are \(> V_0\) by convention. Informally, the degree of a vertex is the number of frequencies immediately below it that do not appear later in our reviewing process. For \(d=0,1,2,3\), let $n_d$ denote the number of vertices of degree $d$. As there are \(2m\ttn\) interaction vertices, we have
\eq{Dee}{
n_0 + n_1 + n_2 + n_3 = 2m\ttn.
}

Denote the number of distinct output frequencies by \(q = \#\{k_{m,1}\sprack\tti\}\in \llbracket1,\ttn\rrbracket\). As every output frequency is \(>V_{2m\ttn}\), and there are \(p\) distinct frequencies in total, we have
\eq{Aye}{
n_1 + 2n_2 + 3 n_3 = p-q.
}

An example of this counting procedure is given in Figure~\ref{f:deg3}.

\begin{figure}[h!]
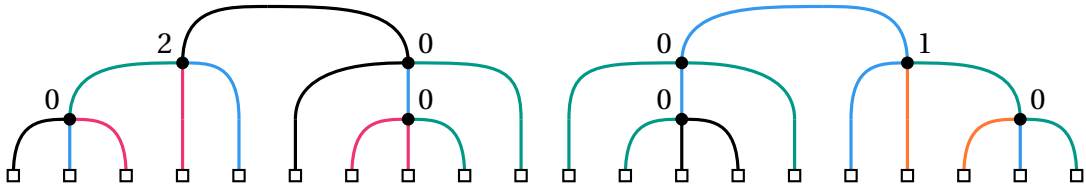

\begin{center}
\tikz[scale=.75]{
\draw[very thick,col6] (1,0)--(1,1) {};
\draw[very thick,col6] (1,1)--(1,2) {};
\draw[very thick,col5] (-1,0)--(-1,1) {};
\draw[very thick,col5] (2,0)--(2,1) {};
\draw[very thick,col1] (-2,0) .. controls (-2,1) and (-1.5,1) .. (-1,1);
\draw[very thick,col6] (0,0) .. controls (0,1) and (-.5,1) .. (-1,1);
\draw[very thick,col5] (2,1) .. controls (2,2) and (1.5,2) .. (1,2);
\draw[very thick,col3] (-1,1) .. controls (-1,2) and (0,2) .. (1,2);
\draw[very thick,col6] (5,0)--(5,1) {};
\draw[very thick,col1] (3,0)--(3,1) {};
\draw[very thick,col3] (7,0)--(7,1) {};
\draw[very thick,col5] (5,1)--(5,2) {};
\draw[very thick,col6] (4,0) .. controls (4,1) and (4.5,1) .. (5,1);
\draw[very thick,col3] (6,0) .. controls (6,1) and (5.5,1) .. (5,1);
\draw[very thick,col3] (7,1) .. controls (7,2) and (6.5,2) .. (5,2);
\draw[very thick,col1] (3,1) .. controls (3,2) and (4.5,2) .. (5,2);
\draw[very thick,col1] (5,2) .. controls (5,3) and (3.75,3) .. (2.5,3) {};
\draw[very thick,col1] (1,2) .. controls (1,3) and (1.75,3) .. (2.5,3) {};
\node[indot] at (-2,0) {};
\node[indot] at (-1,0) {};
\node[indot] at (0,0) {};
\node[indot] at (1,0) {};
\node[indot] at (2,0) {};
\node[indot] at (3,0) {};
\node[indot] at (4,0) {};
\node[indot] at (5,0) {};
\node[indot] at (6,0) {};
\node[indot] at (7,0) {};
\node[dot] at (-1,1) {};
\node[dot] at (1,2) {};
\node[dot] at (5,1) {};
\node[dot] at (5,2) {};
\node[above left] at (-1,1) {$0$};
\node[above left] at (1,2) {$2$};
\node[above right] at (5,1) {$0$};
\node[above right] at (5,2) {$0$};
}
\quad
\tikz[scale=.75]{
\begin{scope}[xscale=-1]
\draw[very thick,col4] (1,0)--(1,1) {};
\draw[very thick,col4] (1,1)--(1,2) {};
\draw[very thick,col5] (-1,0)--(-1,1) {};
\draw[very thick,col5] (2,0)--(2,1) {};
\draw[very thick,col3] (-2,0) .. controls (-2,1) and (-1.5,1) .. (-1,1);
\draw[very thick,col4] (0,0) .. controls (0,1) and (-.5,1) .. (-1,1);
\draw[very thick,col5] (2,1) .. controls (2,2) and (1.5,2) .. (1,2);
\draw[very thick,col3] (-1,1) .. controls (-1,2) and (0,2) .. (1,2);
\draw[very thick,col1] (5,0)--(5,1) {};
\draw[very thick,col3] (3,0)--(3,1) {};
\draw[very thick,col3] (7,0)--(7,1) {};
\draw[very thick,col5] (5,1)--(5,2) {};
\draw[very thick,col1] (4,0) .. controls (4,1) and (4.5,1) .. (5,1);
\draw[very thick,col3] (6,0) .. controls (6,1) and (5.5,1) .. (5,1);
\draw[very thick,col3] (7,1) .. controls (7,2) and (6.5,2) .. (5,2);
\draw[very thick,col3] (3,1) .. controls (3,2) and (4.5,2) .. (5,2);
\draw[very thick,col5] (5,2) .. controls (5,3) and (3.75,3) .. (2.5,3) {};
\draw[very thick,col5] (1,2) .. controls (1,3) and (1.75,3) .. (2.5,3) {};
\node[indot] at (-2,0) {};
\node[indot] at (-1,0) {};
\node[indot] at (0,0) {};
\node[indot] at (1,0) {};
\node[indot] at (2,0) {};
\node[indot] at (3,0) {};
\node[indot] at (4,0) {};
\node[indot] at (5,0) {};
\node[indot] at (6,0) {};
\node[indot] at (7,0) {};
\node[dot] at (-1,1) {};
\node[dot] at (1,2) {};
\node[dot] at (5,1) {};
\node[dot] at (5,2) {};
\node[above right] at (-1,1) {$0$};
\node[above right] at (1,2) {$1$};
\node[above left] at (5,1) {$0$};
\node[above left] at (5,2) {$0$};
\end{scope}
}
\end{center}
\caption{The degrees of each vertex in the graph from Figure~\ref{f:deg2}. Here, we have \(p=5\), \(q=2\), and \((n_0,n_1,n_2,n_3) = (6,1,1,0)\).}\label{f:deg3}
\end{figure}

With these conventions set, we now proceed to estimate each term appearing in the summand on $\RHS{Crostata}$. For the data, we simply use that \(A\) is bounded to obtain
\eq{DataBound}{
\sup_{\vec\ttK}\left|\prod_{\tti=1}^\ttn\Data\sprack\tti\right|\lesssim_{A,m,\ttn} 1.
}

For the interactions, we have the following:
\begin{lemma}\label{l:Osci}
Let \(m,n\geq 1\). Fix an interaction history \(\vec\ttL\) and coloring \(C\) of the associated interaction graph \(G(\vec\ttL)\). Then, for each choice of \(\vec\ttK\in C\) we have the estimate
\eq{OsciBound}{
\left|\bbE\left[\prod_{\tti=1}^\ttn\Interactions\sprack\tti\right]\right| \lesssim_{m,\ttn}\frac1{T_\kin^{m\ttn}}\min\bigl\{N^{q-p},N^{-3m\ttn},N^{-\frac{10}3m\ttn+\frac13n_0}\bigr\}.
}
Further, \(\LHS{OsciBound}\) is only nonzero if each of the input frequencies
\eq{IPF1}{
k_{0,1}\sprack 1,\dots,k_{0,2m+1}\sprack 1,{k_{0,2}\sprack 1}',\dots,{k_{0,2m}\sprack 1\!\!\!\!}'\,\,\,,\dots,k_{0,1}\sprack\ttn,\dots,k_{0,2m+1}\sprack\ttn,{k_{0,2}\sprack \ttn}',\dots,{k_{0,2m}\sprack\ttn\!\!\!\!}'\,\,\,,
}
is paired with one of the input frequencies
\eq{IPF2}{
k_{0,2}\sprack 1,\dots,k_{0,2m}\sprack1,{k_{0,1}\sprack 1}',\dots,{k_{0,2m+1}\sprack 1\!\!\!\!\!\!\!\!\!\!}'\,\,\,\,\,\,\,\,\,,\dots,k_{0,2}\sprack \ttn,\dots,k_{0,2m}\sprack \ttn,{k_{0,1}\sprack \ttn}',\dots,{k_{0,2m+1}\sprack \ttn\!\!\!\!\!\!\!\!\!\!}'\,\,\,\,\,\,\,\,\,,
}
with paired frequencies being equal. In particular, we must have that \(p\leq 4m\ttn\).
\end{lemma}
\bpf
We proceed by estimating \(\LHS{OsciBound}\) using Theorem~\ref{thm:moment_bound}, Corollary~\ref{cor:moment_bound}, and Proposition~\ref{p:WGBD}.

\underline{Application of Corollary~\ref{cor:moment_bound}.} Applying H\"older's inequality, we estimate
\[
\left|\bbE\left[\prod_{\tti=1}^\ttn\Interactions\sprack \tti\right]\right|\leq \frac1{T_{\kin}^{m\ttn}}\prod_{\tti=1}^\ttn \prod_{r=1}^m \left(\bbE|\gamma_{r,{\ell_r\sprack \tti}'}|^{2m\ttn}\right)^{\frac1{2m\ttn}}\left(\bbE|\gamma_{r,\ell_r\sprack \tti}|^{2m\ttn}\right)^{\frac1{2m\ttn}}.
\]
Corollary~\ref{cor:moment_bound} then gives us
\eq{CS bound}{
\left|\bbE\left[\prod_{\tti=1}^\ttn\Interactions\sprack\tti\right]\right|\lesssim_{m,\ttn}\frac1{T_{\kin}^{m\ttn}} N^{-3m\ttn}.
}

\underline{Application of Theorem~\ref{thm:moment_bound}.} Let us denote the maximal number of smart pairs that can be obtained from the \(2m\ttn\) choices of \(\gamma\) terms by \(p_2^*\). Having removed these smart pairs, denote the maximal number of smart triplets that can be obtained from the remaining terms by \(p_3^*\). Theorem~\ref{thm:moment_bound} then gives us the bound
\[
\left|\bbE\left[\prod_{\tti=1}^\ttn\Interactions\sprack\tti\right]\right|\lesssim \frac1{T_{\kin}^{m\ttn}}N^{-\frac72m\ttn+\frac12p_2^* + \frac14p_3^*}\lesssim \frac1{T_{\kin}^{m\ttn}}N^{-\frac{10}3m\ttn+\frac13p_2^*},
\]
where the second inequality follows from the trivial bound \(2p_2^* + 3p_3^*\leq 2m\ttn\). Now observe that at least one \(\gamma\) term in each smart pair corresponds to a degree zero vertex, which gives us \(p_2^*\leq n_0\), and hence
\eq{CS bound SC}{
\left|\bbE\left[\prod_{\tti=1}^\ttn\Interactions\sprack\tti\right]\right|\lesssim \frac1{T_{\kin}^{m\ttn}}N^{-\frac{10}3m\ttn+\frac13n_0}.
}

\underline{Application of Proposition~\ref{p:WGBD}.} Enumerating the positions and frequencies, we obtain the expression
\[
\bbE\left[\prod_{\tti=1}^\ttn\Interactions\sprack\tti\right] = \frac\sigma{T_{\kin}^{m\ttn}} \sum_{j_1,\dots,j_{2m\ttn}}\bbE\left[\prod_{s=1}^{2m\ttn}\overline{\psi_{j_s\kappa_{2s-1}'}}\,\psi_{j_s\kappa_{2s-1}}\,\overline{\psi_{j_s\kappa_{2s}'}}\,\psi_{j_s\kappa_{2s}}\right],
\]
where \(\sigma = \sigma(\vec \ttL)\in \{\pm 1\}\); the frequencies \(\kappa_1,\dots,\kappa_{4m\ttn}\) consist of the input frequencies \eqref{IPF1}, every internal frequency, and every output frequency; the frequencies \(\kappa_1',\dots,\kappa_{4m\ttn}'\) consist of the input frequencies \eqref{IPF2}, every internal frequency, and every output frequency.

First, note that Proposition~\ref{p:WGBD} gives us
\[
\bbE\left[\prod_{\tti=1}^\ttn\Interactions\sprack \tti\right] = 0
\]
unless each of the frequencies \(\kappa_1,\dots,\kappa_{4m\ttn}\) is paired with one of the frequencies \linebreak\(\kappa_1',\dots,\kappa_{4m\ttn}'\), with paired frequencies being equal. This occurs only if each of the input frequencies in \eqref{IPF1} is paired with one of the input frequencies in \eqref{IPF2}, and hence we must have \(p\leq 4m\ttn\) to obtain a non-zero contribution.

Ignoring the direction of each edge, which is unimportant for our application of Proposition~\ref{p:WGBD}, we now construct the Weingarten graph \(W\) for the expression
\[
\prod_{s=1}^{2m\ttn}\overline{\psi_{j_s\kappa_{2s-1}'}}\,\psi_{j_s\kappa_{2s-1}}\,\overline{\psi_{j_s\kappa_{2s}'}}\,\psi_{j_s\kappa_{2s}}
\]
with fixed \(j_1,\dots,j_{2m\ttn}\) from the interaction graph \(G\) by the following algorithm:
\begin{itemize}
\item Replace each input vertex \tikz[baseline=-2.4,scale=0.15]{\node[indot] {};} by a frequency vertex \tikz[baseline=-2.4,scale=0.15]{\node[ghostdot] {};}.
\item Place a frequency vertex \tikz[baseline=-2.4,scale=0.15]{\node[ghostdot] {};} at the center of each of the edge that connects two interaction vertices \tikz[baseline=-2.4,scale=0.15]{\node[dot] {};}. (Subsequently, we consider the interaction vertices to be position vertices in \(W\).)
\item Label each position vertex \tikz[baseline=-2.4,scale=0.15]{\node[dot] {};} with one of the \(j_s\)'s and identify any two vertices for which the corresponding \(j_s\)'s coincide.
\item Label each frequency vertex \tikz[baseline=-2.4,scale=0.15]{\node[ghostdot] {};} with the corresponding frequency \(\kappa_s\) or \(\kappa_s'\) and identify any two vertices for which the corresponding frequencies coincide.
\end{itemize}

Our original interaction graph \(G\) has \(\ttn\) connected components. As there are only \(q\) distinct output frequencies, this procedure leaves us with a Weingarten graph \(W\) with at most \(q\) connected components. Taking \(d = \#\{j_1,\dots,j_{2m\ttn}\}\) to be the number of distinct positions, this procedure shows that the Weingarten graph \(W\) has \(d\) position vertices, \(p\) frequency vertices, and \(8m\ttn\) edges. Applying Proposition~\ref{p:WGBD}, we obtain
\[
\bbE\left[\prod_{s=1}^{2m\ttn}\overline{\psi_{j_s\kappa_{2s-1}'}}\,\psi_{j_s\kappa_{2s-1}}\,\overline{\psi_{j_s\kappa_{2s}'}}\,\psi_{j_s\kappa_{2s}}\right] \lesssim_{m,\ttn} N^{q-d-p}.
\]

Using this bound we may then estimate
\begin{align*}
\left|\bbE\left[\prod_{\tti=1}^\ttn\Interactions\sprack \tti\right]\right| &\lesssim_{m,\ttn} \frac1{T_{\kin}^{m\ttn}}\sum_{d=1}^{2m\ttn}\sum_{\substack{j_1,\dots,j_{2m\ttn}\\\#\{j_1,\dots,j_{2m\ttn}\} = d}}N^{q-d-p}\\
&\lesssim_{m,\ttn} \frac1{T_{\kin}^{m\ttn}}N^{q-p}.
\end{align*}
Taking the minimum of this bound, \eqref{CS bound}, and \eqref{CS bound SC}, we obtain \eqref{OsciBound}.
\epf

It remains to bound the oscillations. Fixing the interaction histories \(\vec\ttL\) and a coloring \(C\) of the interaction graph \(G(\vec\ttL)\), enumerate the frequencies and vertices as above. We say that a vertex \(V_s\) is \emph{non-degenerate} if there exists a frequency \(K_J\) so that \(V_{s-1}<K_J\leq V_s\) and we can write
\eq{nondeg}{
\varpi_s = \eta \lambda_{K_J} + \beta,
}
where \(\eta\neq 0\) and \(\beta\) is a linear combination of \(\lambda_{K_j}\)'s that does not depend on \(\lambda_{K_J}\). If the vertex \(V_s\) is not non-degenerate, we say it is \emph{degenerate}. We denote the number of degenerate vertices by \(D\in \llbracket0,2m\ttn\rrbracket\) and have the following lemma:
\begin{lemma}\label{l:counting degeneracy}
All degree \(0\) vertices are degenerate. All degree \(2\) and \(3\) vertices are non-\linebreak degenerate.
\end{lemma}
\bpf
The fact that all degree \(0\) vertices are degenerate is vacuous. The fact that all degree \(2\) and \(3\) vertices are non-degenerate follows from inspection of the expressions \eqref{Osc} and \eqref{thetaj}. 
\epf

Before stating our main estimates for the oscillations, let us introduce a notational convention. Fix and interaction history \(\vec\ttL\) and a coloring \(C\) of the interaction graph \(G(\vec\ttL)\). If \(\vec\ttK\in C\), a function \(f(\vec\ttK)\) can be thought of as a function \(f(K_1,\dots,K_p)\) of the \(p\) distinct frequencies \(K_1,\dots,K_p\). Now consider vertices \(V_r\), \(V_s\) with \(r<s\). If there are no frequencies \(K_j\) with \(V_r<K_j\leq V_s\) we define
\[
\sum_{V_r<K_j\leq V_s}f(\vec\ttK) := 0.
\]
Otherwise, up to re-indexing the frequencies, we may assume that \(V_r<K_1,\dots,K_d\leq V_s\) for some \(d\geq 1\). We then define
\[
\sum_{V_r<K_j\leq V_s}f(\vec\ttK):=\max_{K_{d+1},\dots,K_p}\sum_{K_1,\dots,K_d}f(K_1,\dots,K_p),
\]
where the maximum and sum are taken over all choices of \(K_1,\dots,K_p\in \llbracket-N,N\rrbracket\). We note that we have relaxed the restriction that \(K_1,\dots,K_p\) are distinct here: this will have no bearing on the subsequent estimates. 

With this convention in hand, our first estimate for the oscillations is the following:

\begin{lemma}\label{l:reciprocal bound} Let \(m,n\geq 1\). Fix an interaction history \(\vec\ttL\) and a coloring \(C\) of the interaction graph \(G(\vec\ttL)\).
Let \(1\leq T\leq N\), \(\nu>0\) and \(|\alpha|\leq N^\nu\). Then, if the vertex \(V_s\) has degree \(d\in\{0,1,2,3\}\) we have the estimate
\eq{DegreeBound}{
\sum_{V_{s-1}<K_j\leq V_s} \frac1{|\alpha + \varpi_s| + \frac1T}\prec_{\nu,m}\begin{cases}TN^d&\qtq{if}V_s\text{ is degenerate},\smallskip\\ N^d &\qtq{if}V_s\text{ is non-degenerate}.\end{cases}
}
\end{lemma}
\bpf
For the degenerate case, we simply bound
\[
\frac1{|\alpha + \varpi_s| + \frac1T}\leq T.
\]

For the non-degenerate case, we choose \(\eta\), \(K_J\), and \(\beta\) as in \eqref{nondeg}. From \eqref{Rigid}, we have
\[
|\beta|\prec_{m}1.
\]
We may then apply \eqref{L1} to bound
\[
\sum_{K_J} \frac1{|\alpha+\varpi_s| + \frac1T}\prec_{\nu,m} N.
\]
The estimate \eqref{DegreeBound} now follows from summing over remaining \(d-1\) frequencies satisfying \(V_{s-1}<K_j\leq V_s\) and taking the maximum over all other frequencies.
\epf

Using this lemma, we obtain the following:

\begin{lemma}\label{l:sprocket} Let \(m,\ttn\geq 1\), \(\tfrac12<b<\tfrac12+m \), and \(1\leq T\leq N\). Fix an interaction history \(\vec\ttL\) and a coloring \(C\) of the interaction graph \(G(\vec\ttL)\). Then, for each \(\tti=1,\dots,\ttn\) we have
\eq{sprocket}{
\sum_{V_{2m(\ttn - \tti)}<K_j\leq V_{2m(\ttn+1 - \tti)}}\left|\Oscillations\sprack\tti\right|\prec_{b,m} T^{D\sprack \tti}N^{n_1\sprack \tti+2n_2\sprack \tti+3n_3\sprack \tti},
}
where \(D\sprack \tti\) is the number of degenerate vertices and \(n_d\sprack \tti\) is the number of degree \(d\) vertices in the subgraph \(G\sprack\tti\) of \(G\).
\end{lemma}
\bpf
We consider the case that \(\tti = \ttn\). The proof for other values of \(\tti\) is identical.

We recall from \eqref{thetaj} that each \(\omega_m=0\) and from Lemma~\ref{l:independence}, the modulation \(\varpi_s\) does not depend on any frequencies \(K_j\leq V_{s-1}\). Applying the Cauchy--Schwarz inequality in \(\tau\) and then Young's convolution inequality, we may then bound
\begin{align*}
&\sum_{V_0<K_j\leq V_{2m}}\left|\Oscillations\sprack n\right|\\
&\qquad\leq  \int_{\R^3} \frac{\left(\tfrac1T + |\tau|\right)^{2b}|\widehat \phi_T(\tau-\alpha)||\widehat\phi_T(\tau - \alpha')|}{\left(\frac 1T + |\alpha|\right)\left(\frac 1T + |\alpha'|\right)}\\
&\qquad\qquad\qquad\times\prod_{s=1}^m\left[\sum_{V_{s-1}<K_j\leq V_s} \frac1{\frac1T + |\alpha + \varpi_s|}\right]\cdot \prod_{s=m+1}^{2m}\left[\sum_{V_{s-1}<K_j\leq V_s} \frac1{\frac1T + |\alpha' + \varpi_s|}\right]\,d\alpha\,d\alpha'\,d\tau\\
&\qquad\lesssim_b \left\| \bigl(1 + T|\tau|\bigr)^{2b}\,\widehat\phi_T(\tau)\right\|_{L^1_\tau}^2 \left\|\left(\tfrac1T + |\alpha|\right)^{b-1}\prod_{s=1}^m\left[\sum_{V_{s-1}<K_j\leq V_s} \frac1{\frac1T + |\alpha + \varpi_s|}\right]\right\|_{L^2_\alpha}\\
&\qquad\qquad\qquad\times\left\|\left(\tfrac1T + |\alpha'|\right)^{b-1}\prod_{s=m+1}^{2m}\left[\sum_{V_{s-1}<K_j\leq V_s} \frac1{\frac1T + |\alpha' + \varpi_s|}\right]\right\|_{L^2_{\alpha'}}\\
&\qquad\lesssim_b \left\|\left(\tfrac1T + |\alpha|\right)^{b-1}\prod_{s=1}^m\left[\sum_{V_{s-1}<K_j\leq V_s} \frac1{\frac1T + |\alpha + \varpi_s|}\right]\right\|_{L^2_\alpha}\\
&\qquad\qquad\qquad\times\left\|\left(\tfrac1T + |\alpha'|\right)^{b-1}\prod_{s=m+1}^{2m}\left[\sum_{V_{s-1}<K_j\leq V_s} \frac1{\frac1T + |\alpha' + \varpi_s|}\right]\right\|_{L^2_{\alpha'}}.
\end{align*}

By symmetry, it suffices to bound the first norm in this expression. By a slight abuse of notation, denote the number of degenerate vertices and the number of degree \(d\) vertices in the set \(\{V_1,\dots,V_m\}\) by \(D\) and \(n_d\), respectively. We will prove that
\eq{recall this later}{
\left\|\left(\tfrac1T + |\alpha|\right)^{b-1}\prod_{s=1}^m\left[\sum_{V_{s-1}<K_j\leq V_s} \frac1{\frac1T + |\alpha + \varpi_s|}\right]\right\|_{L^2_\alpha}\prec_{b,m} T^{D}N^{n_1 + 2n_2 + 3n_3}.
}
from which the estimate \eqref{sprocket} follows.

To prove \eqref{recall this later}, we first use \eqref{Rigid} to obtain
\[
|\varpi_s|\prec_m 1.
\]
As a consequence, for any \(\nu>0\) we may bound the integral over the region where \(|\alpha|>N^\nu\) by
\begin{align*}
&\left\|\left(\tfrac1T + |\alpha|\right)^{b-1}\prod_{s=1}^m\left[\sum_{V_{s-1}<K_j\leq V_s} \frac1{\frac1T + |\alpha + \varpi_s|}\right]\right\|_{L^2(|\alpha|>N^\nu)}\\
&\qquad\qquad\qquad\qquad\qquad\prec_m N^{n_1 + 2n_2 + 3n_3}\left\|\left(\tfrac1T + |\alpha|\right)^{b-1-m}\right\|_{L^2(|\alpha|>N^\nu)}\\
&\qquad\qquad\qquad\qquad\qquad\prec_{b,m} N^{\nu(b-\frac12-m)}N^{n_1 + 2n_2 + 3n_3}.
\end{align*}
To bound the region where \(|\alpha|\leq N^\nu\), we instead apply Lemma~\ref{l:reciprocal bound} to estimate
\begin{align*}
&\left\|\left(\tfrac1T + |\alpha|\right)^{b-1}\prod_{s=1}^m\left[\sum_{V_{s-1}<K_j\leq V_s} \frac1{\frac1T + |\alpha + \varpi_s|}\right]\right\|_{L^2(|\alpha|\leq N^\nu)}\\
&\qquad\qquad\qquad\qquad\qquad\prec_{\nu,m} T^{D}N^{n_1 + 2n_2 + 3n_3}\left\|\left(\tfrac1T + |\alpha|\right)^{b-1}\right\|_{L^2(|\alpha|\leq N^\nu)}\\
&\qquad\qquad\qquad\qquad\qquad\prec_{b,\nu,m} T^{D}N^{n_1 + 2n_2 + 3n_3+\nu(b-\frac12)}.
\end{align*}
The estimate \eqref{recall this later} now follows from the fact that \(\nu>0\) can be chosen to be arbitrarily small.
\epf

Applying this estimate, we may finally bound the oscillations:

\begin{corollary}\label{c:TB}
Let \(m,\ttn\geq 1\), \(\frac12<b<\frac12+m \), and \(1\leq T\leq N\). Fix an interaction history \(\vec\ttL\) and a coloring \(C\) of the interaction graph \(G(\vec\ttL)\). Then we have the estimate
\eq{TB}{
\sum_{\vec\ttK\in C}\left|\prod_{\tti=1}^\ttn\Oscillations\sprack\tti\right| \prec_{b,m,\ttn} T^DN^p.
}
\end{corollary}
\bpf
Recall the our vertices are ordered so that \(V_1,\dots,V_{2m}\) lie in the rightmost connected component of the interaction graph \(G\), corresponding to \(\tti=\ttn\) in \(\LHS{TB}\), \(V_{2m+1},\dots,V_{4m}\) lie in the second connected component from the right, corresponding to \(\tti=\ttn-1\), etc. By Lemma~\ref{l:independence}, we may bound
\[
\LHS{TB}\lesssim N^q \prod_{\tti=1}^\ttn \left[\sum_{V_{2m(\ttn-\tti)}<K_j\leq V_{2m(\ttn+1-\tti)}}\left|\Oscillations\sprack\tti\right|\right],
\]
where we recall that \(q\) is the number of distinct output frequencies. Applying Lemma~\ref{l:sprocket} to each term in the product, the estimate \eqref{TB} now follows from \eqref{Aye}.
\epf

Having obtained estimates for each part of \(\RHS{Crostata}\), we now turn to the:

\bpf[Proof of Proposition~\ref{p:app}]
First, recall that the eigenvectors and eigenvalues are independent, and hence
\begin{align*}
&\bbE \|\chi_T\bbo_{\R_+}a\sbrack m\|_{X^b_T}^{2n}\\
&\qquad\leq \sum_{\vec\ttL,\vec\ttK} \bbE\left[\prod_{\tti=1}^\ttn\Interactions\sprack \tti\right]\cdot\bbE\left[\prod_{i=1}^n\Oscillations\sprack \tti\right]\cdot \left[\prod_{i=1}^n\Data\sprack \tti\right].
\end{align*}

Fix \(\vec\ttL\) and a coloring \(C\) of the interaction graph \(G(\vec\ttL)\). Combining \eqref{DataBound}, \eqref{OsciBound}, and \eqref{TB}, for any \(\delta>0\) we may bound
\begin{align}\label{Kbd}
&\left|\sum_{\vec\ttK\in C}\bbE\left[\prod_{\tti=1}^\ttn\Interactions\sprack\tti\right]\cdot\bbE\left[\prod_{\tti=1}^\ttn\Oscillations\sprack\tti\right]\cdot \left[\prod_{\tti=1}^\ttn\Data\sprack\tti\right]\right|\\
&\qquad\lesssim_{\delta,m,\ttn,b} \frac{T^D}{T_{\kin}^{m\ttn}}\min\bigl\{N^{q -p},N^{-3m\ttn},N^{-\frac{10}3m\ttn+\frac13n_0}\bigr\} N^{p+\delta}.\notag
\end{align}

From \eqref{Dee} and \eqref{Aye}, we have
\[
n_0 + n_1 = 2m\ttn - n_2 - n_3\leq \lfloor 2m\ttn - \tfrac13(p-q-n_1)\rfloor,
\]
and hence by Lemma~\ref{l:counting degeneracy}
\[
D\leq n_0 + n_1 \leq \tfrac32\left(\lfloor 2m\ttn - \tfrac13(p-q-n_1)\rfloor - \tfrac13n_1\right) - \tfrac12 n_0.
\]
We now apply this with \eqref{Kbd}, considering each of the following cases:

\smallskip

\underline{Case 1: \(n_0\geq m\ttn\) and \(p-q\geq 3m\ttn\).} Here, we may bound
\[
\RHS{Kbd}\lesssim_{\delta,m,\ttn,b}\frac{T^{\frac32m\ttn-\frac12n_0}}{T_{\kin}^{m\ttn}} N^{q+\delta}\lesssim_{\delta,m,\ttn,b}\left(\frac {T}{T_\kin}\right)^{m\ttn} N^{\ttn+\delta},
\]
where the second inequality follows from the fact that \(q\leq \ttn\).

\smallskip

\underline{Case 2: \(n_0\geq m\ttn\) and \(p-q< 3m\ttn\).} Here, we have
\[
\RHS{Kbd}\lesssim_{\delta,m,\ttn,b} \frac{T^{\frac32m\ttn - \frac12n_0}}{T_{\kin}^{m\ttn}}\left(\frac{\sqrt T}N\right)^{3m\ttn-(p-q)} N^{q+\delta}\lesssim_{\delta,m,\ttn,b}\left(\frac {T}{T_\kin}\right)^{m\ttn} N^{\ttn+\delta},
\]
where we have again used that \(q\leq \ttn\) as well as that \(T\leq N\).

\smallskip

\underline{Case 3: \(n_0 < m\ttn\) and \(p-q\geq \frac{10}3m\ttn - \frac13n_0\).} In this case, we get
\[
\RHS{Kbd}\lesssim_{\delta,m,\ttn,b} \frac{T^{\frac43m\ttn - \frac13n_0}}{T_{\kin}^{m\ttn}} N^{q+\delta}\lesssim_{\delta,m,\ttn,b}\left(\frac {T^{\frac43}}{T_\kin}\right)^{m\ttn} N^{\ttn+\delta},
\]
where the second inequality uses that \(n_0\geq 0\) and \(q\leq \ttn\). 

\smallskip

\underline{Case 4: \(n_0 < m\ttn\) and \(p-q< \frac{10}3m\ttn - \frac13n_0\).} For this final case, we obtain
\[
\RHS{Kbd}\lesssim_{\delta,m,\ttn,b} \frac{T^{\frac43m\ttn - \frac13n_0}}{T_{\kin}^{m\ttn}}\left(\frac{\sqrt T}N\right)^{\frac{10}3m\ttn - \frac13 n_0 - (p-q)} N^{q+\delta}\lesssim_{\delta,m,\ttn,b} \left(\frac {T^{\frac43}}{T_\kin}\right)^{m\ttn} N^{\ttn+\delta},
\]
where we have once again used that \(q\leq \ttn\) and \(T\leq N\).

\smallskip

Summing over all colorings \(C\) of the graph \(G(\vec\ttL)\) and then over all choices of interaction history \(\vec\ttL\), we arrive at the bound
\[
\bbE \|\chi_T\bbo_{\R_+}a\sbrack m\|_{X^b_T}^{2\ttn}\lesssim_{\delta,m,\ttn,b}\left(\frac {T^{\frac43}}{T_\kin}\right)^{m\ttn} N^{\ttn+\delta},
\]
where we have used that \(T\geq 1\). The estimate \eqref{AppXb} now follows from Markov's inequality.

It remains to consider the special cases \eqref{AppXb-1} and \eqref{AppXb-2}, where \(\ttn=1\) and \(m=1,2\), respectively. Here, we have
\[
\RHS{Kbd}\lesssim_{\delta,b}\frac{T^D}{T_\kin^m}\min\{N^{1-p},N^{-3m},N^{-\frac{10}3m+\frac13n_0}\} N^{p+\delta}.
\]

First consider the case that \(m=1\). When \(D\leq n_0 + n_1\leq 1\), we may bound \linebreak\(\min\{N^{1-p},N^{-3},N^{-\frac{10}3+\frac13n_0}\}\leq N^{1-p}\) to obtain
\[
\RHS{Kbd}\lesssim_{\delta,b}\frac T{T_{\kin}}N^{1+\delta}.
\]
When \(D = n_0 + n_1=2\) we have \(p\leq 3\) and so we may estimate \(\min\{N^{1-p},N^{-3},N^{-\frac{10}3+\frac13n_0}\}\leq N^{-3}\) to obtain\[\RHS{Kbd}\lesssim_{\delta,b}\frac T{T_{\kin}}\frac TN N^{1+\delta},\]which is acceptable.

When \(m=2\) and \(D\leq 2\), we again use that \(\min\{N^{1-p},N^{-6},N^{-\frac{20}3+\frac13n_0}\}\leq N^{1-p}\) to obtain
\[
\RHS{Kbd}\lesssim_{\delta,b} \frac{T^2}{T_{\kin}^2}N^{1+\delta}.
\]
It remains to check the cases where \(D\geq 3\). Using that \(D\leq n_0+n_1\), we have the following scenarios:
\begin{center}
\begin{tabular}{c | c c c c}
\(p\)&\(n_0\)&\(n_1\)&\(n_2\)&\(n_3\)\\\hline
$1$&$4$&$0$&$0$&$0$\\
$2$&$3$&$1$&$0$&$0$\\
$3$&$2$&$2$&$0$&$0$\\
$4$&$1$&$3$&$0$&$0$\\
$5$&$0$&$4$&$0$&$0$\\
$3$&$3$&$0$&$1$&$0$\\
$4$&$3$&$0$&$0$&$1$\\
$4$&$2$&$1$&$1$&$0$\\
$5$&$2$&$1$&$0$&$1$\\
$5$&$1$&$2$&$1$&$0$\\
$6$&$1$&$2$&$0$&$1$\\
$6$&$0$&$3$&$1$&$0$\\
$7$&$0$&$3$&$0$&$1$
\end{tabular}
\end{center}
Bounding \[D\leq n_0+n_1\qtq{and}\min\{N^{1-p},N^{-6},N^{-\frac{20}3+\frac13n_0}\}\leq N^{-6},\] we may readily check that each case leads to an acceptable contribution, except for the final one.

For the final case, recall from Lemma~\ref{l:Osci} that the initial frequencies must be paired. This means that, counting from left to right, each odd input edge is paired with an even input edge with paired edges having the same color. Further, by definition, a degree one vertex \(V_s\) can only be degenerate if the unique frequency \(K_j\) satisfying \(V_{s-1}<K_j\leq V_s\) appears in a consecutive pair of inputs and the other input frequency is a distinct color. This means that \(3\) of the \(4\) interaction vertices in our graph must have a repeated color on two consecutive inputs that does not appear later in our reviewing process, and a distinct color on the third edge that does appear later in our reviewing process. The remaining vertex has degree \(3\) and hence must have three distinctly colored inputs.

Now let us consider all colorings of \(m=2, \ttn=1\) interaction graphs that satisfy these properties. First, note that \(V_1\) cannot have degree \(3\). Indeed, the pairing of the input frequencies means that \(V_1\) can only have degree \(0\) or \(1\) for \emph{any} coloring of the graph. Taking \(V_1\) to be a degenerate degree one vertex, we next consider \(V_2\). Two of the input frequencies for \(V_2\) are initial frequencies. If these two frequencies are distinct, they must both be paired with an input frequency for \(V_1\): This is evidently impossible as \(V_1\) only has one unpaired input frequency. This not only shows that \(V_2\) is a degenerate degree \(1\) vertex but also that the interaction history for the left hand graph, \(\ell\) is either \((1,1)\) or \((3,1)\). By symmetry, we may assume we are in the latter case and hence we currently have the graph shown in Figure~\ref{f:bad case a}.

\begin{figure}[h!]
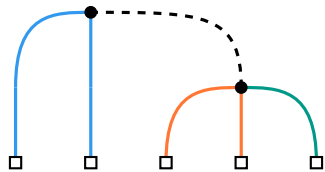

\begin{center}
\tikz{
\draw[very thick,col4] (6,0)--(6,1) {};
\draw[very thick,col5] (3,0)--(3,1) {};
\draw[very thick,col5] (4,0)--(4,1) {};
\draw[very thick,col5] (4,1)--(4,2) {};
\draw[very thick,col4] (5,0) .. controls (5,1) and (5.5,1) .. (6,1);
\draw[very thick,col3] (7,0) .. controls (7,1) and (6.5,1) .. (6,1);
\draw[very thick,dashed] (6,1) .. controls (6,2) and (5.5,2) .. (4,2);
\draw[very thick,col5] (3,1) .. controls (3,2) and (3.5,2) .. (4,2);
\node[indot] at (3,0) {};
\node[indot] at (4,0) {};
\node[indot] at (5,0) {};
\node[indot] at (6,0) {};
\node[indot] at (7,0) {};
\node[dot] at (6,1) {};
\node[dot] at (4,2) {};
}
\end{center}
\caption{The right hand interaction graph in the case \(m=2\), \(D=3\), \((n_0,n_1,n_2,n_3) = (0,3,0,1)\). As before, distinct colors represent distinct frequencies. The dashed frequency can be any color except orange or blue.}\label{f:bad case a}
\end{figure}

Turning to the left hand graph, let us note that we only have one free intial frequency (shown in green in Figure~\ref{f:bad case a}) in the right hand graph. This shows that \(V_3\) cannot be a degree \(3\) vertex and hence \(V_4\) must be degree \(3\). Further, the input edge of \(V_3\) that is not paired with another input edge of \(V_3\) must be paired with an initial edge that is an input to \(V_4\). Putting this together we are left with only two possible graphs, shown in Figure~\ref{f:bad case}.

\begin{figure}[h!]
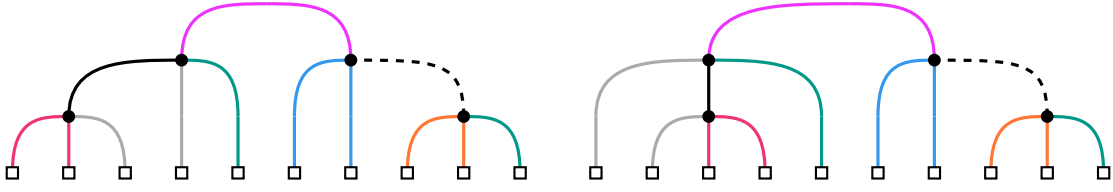

\begin{center}
\tikz[scale=.75]{
\draw[very thick,col7] (1,0)--(1,1) {};
\draw[very thick,col7] (1,1)--(1,2) {};
\draw[very thick,col6] (-1,0)--(-1,1) {};
\draw[very thick,col3] (2,0)--(2,1) {};
\draw[very thick,col6] (-2,0) .. controls (-2,1) and (-1.5,1) .. (-1,1);
\draw[very thick,col7] (0,0) .. controls (0,1) and (-.5,1) .. (-1,1);
\draw[very thick,col3] (2,1) .. controls (2,2) and (1.5,2) .. (1,2);
\draw[very thick,col1] (-1,1) .. controls (-1,2) and (0,2) .. (1,2);
\draw[very thick,col4] (6,0)--(6,1) {};
\draw[very thick,col5] (3,0)--(3,1) {};
\draw[very thick,col5] (4,0)--(4,1) {};
\draw[very thick,col5] (4,1)--(4,2) {};
\draw[very thick,col4] (5,0) .. controls (5,1) and (5.5,1) .. (6,1);
\draw[very thick,col3] (7,0) .. controls (7,1) and (6.5,1) .. (6,1);
\draw[very thick,dashed] (6,1) .. controls (6,2) and (5.5,2) .. (4,2);
\draw[very thick,col5] (3,1) .. controls (3,2) and (3.5,2) .. (4,2);
\draw[very thick,col2] (4,2) .. controls (4,3) and (3.25,3) .. (2.5,3) {};
\draw[very thick,col2] (1,2) .. controls (1,3) and (1.75,3) .. (2.5,3) {};
\node[indot] at (-2,0) {};
\node[indot] at (-1,0) {};
\node[indot] at (0,0) {};
\node[indot] at (1,0) {};
\node[indot] at (2,0) {};
\node[indot] at (3,0) {};
\node[indot] at (4,0) {};
\node[indot] at (5,0) {};
\node[indot] at (6,0) {};
\node[indot] at (7,0) {};
\node[dot] at (-1,1) {};
\node[dot] at (1,2) {};
\node[dot] at (6,1) {};
\node[dot] at (4,2) {};
}
\quad\quad
\tikz[scale=.75]{
\draw[very thick,col7] (-2,0)--(-2,1) {};
\draw[very thick,col7] (-2,1) .. controls (-2,2) and (-1,2) .. (0,2);
\draw[very thick,col6] (0,0)--(0,1) {};
\draw[very thick,col3] (2,0)--(2,1) {};
\draw[very thick,col7] (-1,0) .. controls (-1,1) and (-0.5,1) .. (0,1);
\draw[very thick,col6] (1,0) .. controls (1,1) and (.5,1) .. (0,1);
\draw[very thick,col3] (2,1) .. controls (2,2) and (1,2) .. (0,2);
\draw[very thick,col1] (0,1)--(0,2);
\draw[very thick,col4] (6,0)--(6,1) {};
\draw[very thick,col5] (3,0)--(3,1) {};
\draw[very thick,col5] (4,0)--(4,1) {};
\draw[very thick,col5] (4,1)--(4,2) {};
\draw[very thick,col4] (5,0) .. controls (5,1) and (5.5,1) .. (6,1);
\draw[very thick,col3] (7,0) .. controls (7,1) and (6.5,1) .. (6,1);
\draw[very thick,dashed] (6,1) .. controls (6,2) and (5.5,2) .. (4,2);
\draw[very thick,col5] (3,1) .. controls (3,2) and (3.5,2) .. (4,2);
\draw[very thick,col2] (4,2) .. controls (4,3) and (3.25,3) .. (2.5,3) {};
\draw[very thick,col2] (0,2) .. controls (0,3) and (1.25,3) .. (2.5,3) {};
\node[indot] at (-2,0) {};
\node[indot] at (-1,0) {};
\node[indot] at (0,0) {};
\node[indot] at (1,0) {};
\node[indot] at (2,0) {};
\node[indot] at (3,0) {};
\node[indot] at (4,0) {};
\node[indot] at (5,0) {};
\node[indot] at (6,0) {};
\node[indot] at (7,0) {};
\node[dot] at (0,1) {};
\node[dot] at (0,2) {};
\node[dot] at (6,1) {};
\node[dot] at (4,2) {};
}
\end{center}
\caption{The only possible interaction graphs in the case \(m=2\), \(D=3\), \((n_0,n_1,n_2,n_3) = (0,3,0,1)\). As before, distinct colors represent distinct frequencies. The dashed frequency can be any color except orange or blue.}\label{f:bad case}
\end{figure}

However, if we reflect either of the graphs in Figure~\ref{f:bad case} horizontally about its line of symmetry, we obtain a graph with either \((n_0,n_1,n_2,n_3) = (0,2,2,0)\) or \((n_0,n_1,n_2,n_3) = (1,0,3,0)\), depending on the color of the dashed frequency in Figure~\ref{f:bad case}. Each of these gives an acceptable contribution as \(D\leq n_0 + n_1\leq 2\). However, reflecting our original graph corresponds to reversing the order in which we enumerated the vertices, or equivalently considering the complex conjugate of our coloring. Either way, it is clear that the original and reflected graphs must satisfy identical bounds, and hence we obtain an acceptable contribution to \eqref{AppXb-2}.
\epf

\section{The linearized operator around the approximate solution}\label{s:lin}

In this section we consider the real linear operator \(\cL\), as defined in \eqref{Linearization}, and prove Proposition~\ref{p:Lin}.

Using \eqref{app-sum} and \eqref{Linearization} we may decompose
\begin{align*}
\cL a &= \sum_{\mfi,\mfj=0}^M 2 \tfrac{\mu^2}{N} \Bigl[ \Psi^*\bigl( \Psi a^{(\mfi)} \odot \overline{ \Psi a^{(\mfj)}} \odot \Psi a \bigr) - \tfrac1{2N+1}\<a\sbrack \mfj,a\sbrack \mfi\>a - \tfrac1{2N+1}\<a\sbrack \mfj,a\>a\sbrack \mfi\Bigr]\\
&\quad +  \sum_{\mfi,\mfj=0}^M\tfrac{\mu^2}{N} \Bigl[\Psi^*\bigl( \Psi a^{(\mfi)}\odot  \overline{ \Psi a } \odot \Psi a^{(\mfj)} \bigr) - \tfrac1{2N+1}\<a,a\sbrack \mfi\>a\sbrack \mfj - \tfrac1{2N+1}\<a,a\sbrack \mfj\>a\sbrack \mfi\Bigr] \\
& = \sum_{\mfi,\mfj=0}^M \Bigl[2 \cL_{\mfi\mfj} a + \overline{\cL_{\mfi\mfj}' a}\Bigr],
\end{align*}
where the operators \(\cL_{\mfi\mfj}\), \(\cL_{\mfi\mfj}'\) are complex linear.

Our main estimate for the operators \(\cL_{\mfi\mfj}\), \(\cL_{\mfi\mfj}'\) is stated in the following lemma, the proof of which we momentarily delay:

\begin{lemma}\label{l:target} Let \(0\leq \mfi,\mfj\leq M\) and \(1\leq T\leq N\). Then we have the estimate
\begin{equation}
    \label{targetestimate}
\left\|\chi_T\, \cL_{\mfi\mfj} \right\|_{X^{1/2}_T \to X^{-1/2}_T} + \left\|\chi_T\, \cL_{\mfi\mfj}' \right\|_{X^{1/2}_T \to X^{-1/2}_T} \prec_{\mfi,\mfj} \left( \frac{T^{4/3}}{T_{\kin}} \right)^{\frac{\mfi+\mfj+1}{2}}.
\end{equation}
\end{lemma}

We may then give the

\bpf[Proof of Proposition~\ref{p:Lin}] It suffices to prove that for each \(0\leq \mfi,\mfj\leq M\) we have the estimate
\eq{targetcons}{
\|\chi_T\, \cL_{\mfi\mfj} \|_{X^{b}_T \to X^{b-1}_T} + \|\chi_T \,\cL_{\mfi\mfj}' \|_{X^{b}_T \to X^{b-1}_T} \prec_{b,i,j} (NT)^{2b-1} \left( \frac{T^{4/3}}{T_{\kin}} \right)^{\frac{\mfi+\mfj+1}{2}}.
}
For simplicity, we only consider the bound for \(\cL_{\mfi\mfj}\). The estimate for \(\cL_{\mfi\mfj}'\) is identical.
 
Recalling that $X^0_T=L^2_t \ell^2$, we argue as in the proof of Lemma~\ref{l:Nonlinear}, first using the embedding \(\ell^2\subset\ell^\infty\) and \eqref{Cont}, followed by \eqref{free bird} when \(\mfi=0\) or \(\mfj = 0\) and \eqref{AppXb} otherwise to obtain
\begin{align*}
\left\| \chi_T\, \cL_{\mfi\mfj} a \right\|_{X^0_T} & \lesssim \tfrac{\mu^2}{N} \|\chi_{2T}\, a^{(\mfi)} \|_{L^\infty_t \ell^2}\|\chi_{2T}\, a^{(\mfj)} \|_{L^\infty_t \ell^2}\|a\|_{L^2_t\ell^2}\\
&\lesssim_b  \tfrac{\mu^2}NT^{2b-1}\|\chi_{4T}\,a^{(\mfi)} \|_{X^b_T}\|\chi_{4T}\,a^{(\mfj)} \|_{X^b_T}\|a\|_{X_T^0}\prec_{b,\mfi,\mfj} N\left(\frac{T^{4/3}}{T_{\kin}}\right)^{\frac{\mfi+\mfj+1}2}\|a\|_{X_T^0},
\end{align*}
where the final inequality uses hypotheses that \(b\leq 5/6\) and \(T\geq 1\).

We now interpolate between this bound and \eqref{targetestimate} to yield
\begin{align*}
\|\chi_T\,\cL_{\mfi\mfj}\|_{X_T^b\to X_T^{b-1}}&\lesssim_b T^{2b-1}\|\chi_T\,\cL_{\mfi\mfj}\|_{X_T^{1-b}\to X_T^{b-1}}\\
&\lesssim_bT^{2b-1} \|\chi_T\,\cL_{\mfi\mfj}\|_{X_T^{1/2}\to X_T^{-1/2}}^{1-(2b-1)}\|\chi_T\,\cL_{\mfi\mfj}\|_{X_T^0\to X_T^0}^{2b-1}\prec_{b,\mfi,\mfj} (NT)^{2b-1}\left(\frac{T^{4/3}}{T_{\kin}}\right)^{\frac{\mfi+\mfj+1}2},
\end{align*}
which completes the proof of \eqref{targetcons}.
\epf

We now turn to the proof of Lemma~\ref{l:target}. Here, we confine our attention to the operator \(\cL_{\mfi\mfj}\) restricted to positive times. The corresponding estimate for negative times and for the operator \(\cL_{\mfi\mfj}'\) follow from a similar argument.

Denoting the Fourier transform by \(\cF\), we introduce the operator
\[
\fL_{\mfi\mfj} = \bigl(\tfrac1T + |\cdot|\bigr)^{-\frac12}\cF e^{it\Lambda}\chi_T\,\bbo_{\R_+}\,\cL_{\mfi\mfj}e^{-it\Lambda}\cF^*\bigl(\tfrac1T + |\cdot|\bigr)^{-\frac12},
\]
so that
\[
\| \chi_T\,\bbo_{\R_+}\,\cL_{\mfi\mfj}\|_{X_T^{1/2}\to X_T^{-1/2}} = \|\fL_{\mfi\mfj}\|_{L_t^2\ell^2\to L_t^2\ell^2}.
\]
We then compute that for any \(f\in L^2_t\ell^2\) we have
\[
\bigl(\fL_{\mfi\mfj} f\bigr)_k(\tau) = \int_\R\sum_{\ell=-N}^N\fL_{\mfi\mfj}(\tau,k;\sigma,\ell)f_\ell(\sigma)\,d\sigma,
\]
where the kernel
\begin{align*}
\fL_{\mfi\mfj} (\tau_0,\vk_0;\tau_3,\vk_3) & = \frac2{(2\pi)^{3/2}}\sum_{\vk_1,\vk_2}\int_{\R^2}\frac{\mu^2}{N} \frac{\widehat\chi_T\bigl(\tau_0 - \tau_1 + \tau_2 - \tau_3 -  \Omega(\vk_0,\vk_1,\vk_2,\vk_3)\bigr)}{\sqrt{\frac1T + |\tau_0|}\sqrt{\frac1T + |\tau_3|}} \gamma(\vk_0,\vk_1,\vk_2,\vk_3)\\
&\qquad\qquad\qquad\qquad\times\cF\bigl(\chi_{2T} \bbo_{\R_+}f^{(\mfi)}\bigr)_{\vk_1}(\tau_1)\, \overline{\cF\bigl(\chi_{2T}\bbo_{\R_+} f^{(\mfj)}\bigr)_{\vk_2} (\tau_2)} \, d\tau_1 \,d\tau_2,
\end{align*}
and the profiles \(f\sbrack \mfi\), \(f\sbrack \mfj\) are defined as in \eqref{Profile}.

We will bound the operator norm of \(\fL_{\mfi\mfj}\) by considering the expected value of
\[
\tr\Bigl[\bigl(\fL_{\mfi\mfj}^*\fL_{\mfi\mfj}\bigr)^\ttn\Bigr].
\]
To do this, we first consider the operator \(\mf M_{\mfi\mfj} = \fL_{\mfi\mfj}^*\fL_{\mfi\mfj}\) with kernel
\[
\mf M_{\mfi\mfj}(\tau_0',\vk_0';\tau_3,\vk_3) = \sum_{\vk_0,\vk_3'}\int_{\R^2}\overline{\fL_{\mfi\mfj}(\tau_3',\vk_3';\tau_0',\vk_0')}\fL_{\mfi\mfj}(\tau_0,\vk_0;\tau_3,\vk_3)\bbo_{\{\vk_3'=\vk_0\}}\delta(\tau_3'-\tau_0)\,d\tau_0\,d\tau_3'
\]

Mirroring our approach from Section~\ref{s:app}, we introduce interaction histories \(\ell,\ell''\) and frequencies \(\cK,\cK''\) for \(f\sbrack \mfi\), \(\overline{f\sbrack \mfi}\), respectively, as well as interaction histories \(\ell',\ell'''\) and frequencies \(\cK',\cK'''\) for \(\overline{f\sbrack \mfj}\), \(f\sbrack \mfj\), respectively. We then take \(\ttL = (\ell''',\ell'',\ell',\ell)\), \linebreak\(\ttK = (\vk_0,\vk_1,\vk_2,\vk_1',\vk_2',\vk_3',\cK''',\cK'',\cK',\cK)\) and use Lemma~\ref{l:Profile} to obtain (c.f. \eqref{Nala}):
\begin{align}\label{Jafar}
\mf M_{\mfi\mfj}(\tau_0',\vk_0';\tau_3,\vk_3)
&= \sum_{\ttL,\ttK}\Identifications\cdot\Interactions\cdot\Oscillations\cdot\Data,
\end{align}
where we denote
\begin{align*}
&\Identifications(\ttL,\ttK) = \Delta(\vk_2',\cK''',\ell''')\Delta(\vk_1',\cK'',\ell'')\Delta(\vk_2,\cK',\ell')\Delta(\vk_1,\cK,\ell)\bbo_{\{\vk_3'=\vk_0\}},\\
&\Interactions(\ttL,\ttK)\!=\! \tfrac{\mu^{4m}}{N^{2m}}\gamma(\vk_0,\vk_1,\vk_2,\vk_3)\overline{\gamma(\vk_0',\vk_1',\vk_2',\vk_3')}\\&\qquad\qquad\qquad\qquad\qquad\times\prod_{r=1}^m\gamma_{r,\ell_r'''}(\cK''')\,\overline{\gamma_{r,\ell_r''}(\cK'')}\,\overline{\gamma_{r,\ell_r'}(\cK')}\,\gamma_{r,\ell_r}(\cK),\\
&\Oscillations(\ttL,\ttK)\\
&\qquad\qquad\qquad = \tfrac1{2\pi^3}\int_{\R^{10}}\tfrac{\widehat\chi_T\left(\tau_0 - \tau_1 + \tau_2 - \tau_3 - \Omega(\vk_0,\vk_1,\vk_2,\vk_3)\right)}{\sqrt{\frac1T + |\tau_0|}\sqrt{\frac1T + |\tau_3|}}\tfrac{\overline{\widehat\chi_T\left(\tau_0' - \tau_1' + \tau_2' - \tau_3' - \Omega(\vk_0',\vk_1',\vk_2',\vk_3')\right)}}{\sqrt{\frac1T + |\tau_0'|}\sqrt{\frac1T + |\tau_3'|}}\delta(\tau_0 - \tau_3')\\
&\qquad\qquad\qquad\quad\times\tfrac{\widehat\phi_T(\tau_2'-\alpha''')\,\overline{\widehat\phi_T(\tau_1'-\alpha'')}\,\overline{\widehat\phi_T(\tau_2-\alpha')}\,\widehat\phi_T(\tau_1-\alpha)\,d\alpha'''\,d\alpha''\,d\alpha'\,d\alpha\,d\tau_1'\,d\tau_2'\,d\tau_1\,d\tau_2\,d\tau_3'\,d\tau_0}{\prod_{r=0}^\mfj\left(\alpha''' + \omega_r(\cK''',\ell''') - \frac{i}{T}\right) \left(\alpha' + \omega_r(\cK',\ell') - \frac{i}{T}\right)\prod_{r=0}^\mfi\left(\alpha'' + \omega_r(\cK'',\ell'') + \frac{i}{T}\right) \left(\alpha + \omega_r(\cK,\ell) + \frac{i}{T}\right) },\\  
&\Data(\ttK) = \bA\bigl(k_{0,1}''',\dots,k_{0,2\mfj+1}'''\bigr)\,\overline{\bA\bigl(k_{0,1}'',\dots,k_{0,2\mfi+1}''\bigr)}\,\overline{\bA\bigl(k_{0,1}',\dots,k_{0,2\mfj+1}'\bigr)}\,\bA\bigl(k_{0,1},\dots,k_{0,2\mfi+1}\bigr),
\end{align*}
with \(m = \mfi+\mfj+1\) and \(\phi_T(t) = \frac{e^{t/T}}{2\pi}\chi_{2T}(t) \).

To obtain bounds, it will again be useful to introduce an interaction graph associated to this expression. To the authors' knowledge, graphs of this form were first used to prove similar estimates for the linearized operator in~\cite{DH1,CG1}. For a fixed choice of \(\ttL\), our the interaction graph is formed as follows:
\begin{itemize}
\item Take two copies, \(G\) and \(G'\), of the (unique) Feynman diagram for \(a\sbrack 1\).
\item Identify the output vertex in the diagram for \(a\sbrack \mfi\) with the left input vertex in \(G\).
\item Identify the output vertex in the diagram for \(\overline{a\sbrack \mfj}\) with the middle input vertex in \(G\).
\item Identify the output vertex in the diagram for \(\overline{a\sbrack \mfi}\) with the left input vertex in \(G'\).
\item Identify the output vertex in the diagram for \(a\sbrack \mfj\) with the middle input vertex in \(G'\).
\item Identify the output vertex in \(G\) with the right input vertex in \(G'\).
\item Finally, remove all vertices connected to exactly two edges.
\end{itemize}
This construction is shown in Figure~\ref{f:hilbert_schmidt}, with a concrete example given in Figure~\ref{f:hilbert_schmidt_example}.

\begin{figure}[h!]
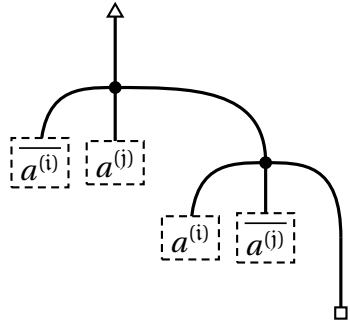

\begin{center}
\tikz{
\draw[very thick] (1,1)--(1,3) {};

\draw[very thick] (0,1) .. controls (0,2) and (0.5,2) .. (1,2);
\draw[very thick] (3,1) .. controls (3,2) and (2,2) .. (1,2);

\draw[very thick] (2,0) .. controls (2,1) and (2.5,1) .. (3,1);
\draw[very thick] (3,0)--(3,1) {};
\draw[very thick] (4,0) .. controls (4,1) and (3.5,1) .. (3,1);

\draw[very thick] (4,0)--(4,-1);

\node[outdot] at (1,3) {};
\node[dot] at (1,2) {};
\node[dot] at (3,1) {};
\node[indot] at (4,-1) {};
\node[gluedot] at (0,1) {$\overline{a\sbrack \mfi}$};
\node[gluedot] at (1,1) {$a\sbrack \mfj$};
\node[gluedot] at (3,0) {$\overline{a\sbrack \mfj}$};
\node[gluedot] at (2,0) {$a\sbrack \mfi$};

}
\end{center}
\caption{The interaction graph for \eqref{Jafar}. The Feynman diagrams for \(a\sbrack \mfi\), etc. appear in the dashed boxes.}\label{f:hilbert_schmidt}
\end{figure}

\begin{figure}[h!]
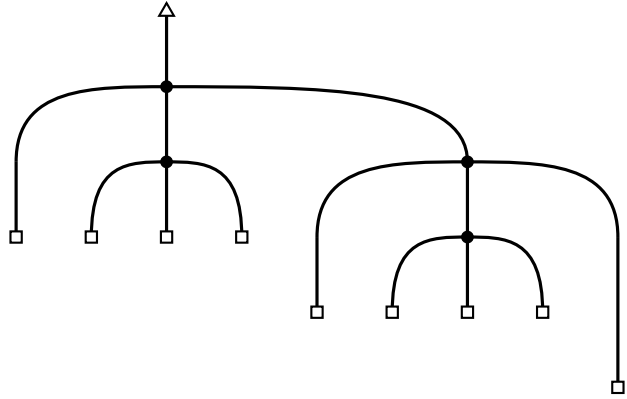

\begin{center}
\tikz{

\draw[very thick] (2,1)--(2,4);
\draw[very thick] (0,1)--(0,2);

\draw[very thick] (0,2) .. controls (0,3) and (1,3) .. (2,3);

\draw[very thick] (1,1) .. controls (1,2) and (1.5,2) .. (2,2);
\draw[very thick] (3,1) .. controls (3,2) and (2.5,2) .. (2,2);

\draw[very thick] (6,0)--(6,2) {};
\draw[very thick] (4,0)--(4,1) {};
\draw[very thick] (8,-1)--(8,1) {};
\draw[very thick] (4,1) .. controls (4,2) and (5,2) .. (6,2);
\draw[very thick] (8,1) .. controls (8,2) and (7,2) .. (6,2);
\draw[very thick] (5,0) .. controls (5,1) and (5.5,1) .. (6,1);
\draw[very thick] (7,0) .. controls (7,1) and (6.5,1) .. (6,1);

\draw[very thick] (6,2) .. controls (6,3) and (4,3) .. (2,3);

\node[outdot] at (2,4) {};

\node[dot] at (2,3) {};
\node[dot] at (2,2) {};
\node[dot] at (6,2) {};
\node[dot] at (6,1) {};

\node[indot] at (0,1) {};
\node[indot] at (1,1) {};
\node[indot] at (2,1) {};
\node[indot] at (3,1) {};
\node[indot] at (4,0) {};
\node[indot] at (5,0) {};
\node[indot] at (6,0) {};
\node[indot] at (7,0) {};
\node[indot] at (8,-1) {};
}
\end{center}
\caption{The interaction graph for \eqref{Jafar} from Figure~\ref{f:hilbert_schmidt} in the case \(\mfi=0\) and \(\mfj=1\).}\label{f:hilbert_schmidt_example}
\end{figure}

We now consider the the operator \((\fL_{\mfi\mfj}^*\fL_{\mfi\mfj})^\ttn = \mf M_{\mfi\mfj}^n\). Mirroring our approach from Section~\ref{s:app}, we denote the interaction history and frequencies of the \(\tti\)\textsuperscript{th} copy of \(\mf M_{\mfi\mfj}\) by \(\ttL\sprack \tti\), \(\ttK\sprack \tti\), respectively, and similarly denote the input and output variables by \((\tau_3\sprack\tti,\vk_3\sprack \tti)\) and \(({\tau_0\sprack \tti}',{\vk_0\sprack \tti}')\). We again take \(\vec\ttL = (\ttL\sprack 1,\dots,\ttL\sprack n)\), but for the frequencies, we make the minor modification \[\vec\ttK = (\ttK\sprack 1,\dots,\ttK\sprack n,{\vk_0\sprack 1}',\dots,{\vk_0\sprack \ttn}', \vk_3\sprack1,\dots,\vk_3\sprack\ttn).\] Setting \(\Identifications\sprack \tti = \Identifications(\ttL\sprack \tti,\ttK\sprack \tti)\), etc., we arrive at the following expression for the trace of \(\mf M_{\mfi\mfj}^\ttn\)
\begin{align}
\tr\Bigl[\mf M_{\mfi\mfj}^\ttn\Bigr] =  \sum_{\vec\ttL,\vec\ttK}\int_{\R^{2n}}\prod_{\tti=1}^\ttn&\Identifications\sprack \tti\bbo_{\{{\vk_0\sprack \tti}' = \vk_3\sprack{\tti+1}\}}\cdot\Interactions\sprack\tti\label{Chewbacca}\\
&\times\Oscillations\sprack \tti\delta({\tau_0\sprack \tti}' - \tau_3\sprack{\tti+1})\cdot\Data\sprack \tti\,d\vec\tau_0'\,d\vec\tau_3,\notag
\end{align}
where we take \(\vk_3\sprack{\ttn+1} = \vk_3\sprack1\), \(\tau_3\sprack{\ttn+1} = \tau_3\sprack 1\), and denote \(\vec\tau_0' = ({\tau_0\sprack 1}',\dots,{\tau_0\sprack \ttn}')\), \(\vec\tau_3 = (\tau_3\sprack1,\dots,\tau_3\sprack\ttn)\).

Taking \(G\sprack \tti\) to be the interaction graph for the \(\tti\)\textsuperscript{th} copy of \(\mf M_{\mfi\mfj}\) in \(\mf M_{\mfi\mfj}^n\), we form an interaction graph for the product by first identifying the output vertex of \(G\sprack{\tti+1}\) with the rightmost input vertex in \(G\sprack \tti\) for each \(\tti=1,\dots,n-1\), and then removing all vertices connected to exactly two edges. This is illustrated in Figure~\ref{f:schatten}.

\begin{figure}[h!]
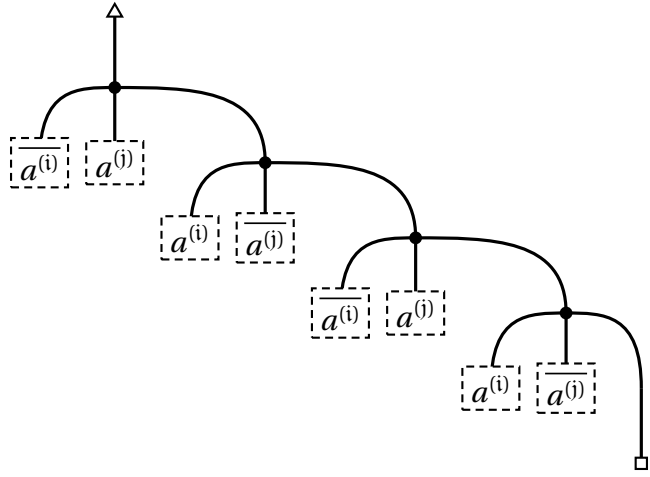

\begin{center}
\tikz{
\begin{scope}
\draw[very thick] (1,1)--(1,3) {};
\draw[very thick] (0,1) .. controls (0,2) and (0.5,2) .. (1,2);
\draw[very thick] (3,1) .. controls (3,2) and (2,2) .. (1,2);
\draw[very thick] (2,0) .. controls (2,1) and (2.5,1) .. (3,1);
\draw[very thick] (3,0)--(3,1) {};
\draw[very thick] (5,0) .. controls (5,1) and (4,1) .. (3,1);
\node[outdot] at (1,3) {};
\node[dot] at (1,2) {};
\node[dot] at (3,1) {};
\node[gluedot] at (0,1) {$\overline{a\sbrack \mfi}$};
\node[gluedot] at (1,1) {$a\sbrack \mfj$};
\node[gluedot] at (3,0) {$\overline{a\sbrack \mfj}$};
\node[gluedot] at (2,0) {$a\sbrack \mfi$};
\end{scope}
\begin{scope}[shift={(4,-2)}]
\draw[very thick] (1,1)--(1,2) {};
\draw[very thick] (0,1) .. controls (0,2) and (0.5,2) .. (1,2);
\draw[very thick] (3,1) .. controls (3,2) and (2,2) .. (1,2);
\draw[very thick] (2,0) .. controls (2,1) and (2.5,1) .. (3,1);
\draw[very thick] (3,0)--(3,1) {};
\draw[very thick] (4,0) .. controls (4,1) and (3.5,1) .. (3,1);
\draw[very thick] (4,0)--(4,-1);
\node[dot] at (1,2) {};
\node[dot] at (3,1) {};
\node[indot] at (4,-1) {};
\node[gluedot] at (0,1) {$\overline{a\sbrack \mfi}$};
\node[gluedot] at (1,1) {$a\sbrack \mfj$};
\node[gluedot] at (3,0) {$\overline{a\sbrack \mfj}$};
\node[gluedot] at (2,0) {$a\sbrack \mfi$};
\end{scope}

}
\end{center}
\caption{The interaction graph for \eqref{Chewbacca} in the case \(n=2\).}\label{f:schatten}
\end{figure}

Continuing to mirror the approach of Section~\ref{s:app}, we now enumerate the vertices in \(G\), working from the rightmost subgraph, \(G\sprack \ttn\), to the leftmost subgraph, \(G\sprack 1\), and enumerating each subgraph as follows:
\begin{itemize}
\item Enumerate the interaction vertices for \(\overline{a\sbrack \mfj}\) from bottom to top.
\item Enumerate the interaction vertices for \(a\sbrack \mfi\) from bottom to top.
\item Enumerate the interaction vertices for \(a\sbrack \mfj\) from bottom to top.
\item Enumerate the interaction vertices for \(\overline{a\sbrack \mfi}\) from bottom to top.
\item Enumerate the two remaining interaction vertices from right to left.
\end{itemize}

For each vertex \(V_s\), we take \(\varpi_s\) to be the associated modulation as defined in Section~\ref{s:app}, with the exception of the vertices \(V_{2m\tti-1},V_{2m\tti}\) for \(\tti=1,\dots,n\), i.e., the last two vertices enumerated in the subgraph \(G\sprack \tti\). For these vertices, we take
\begin{alignat*}{3}
&\varpi_{2m-1} = \Omega(\vk_0\sprack n,\vk_1\sprack n,\vk_2\sprack n,\vk_3\sprack n),\qquad
&\varpi_{2m} &= \Omega(\vk_0\sprack n,{\vk_1\sprack n}',{\vk_2\sprack n}',\vk_3\sprack n),\\
&\varpi_{4m-1} = \Omega(\vk_0\sprack{n-1},\vk_1\sprack{n-1},\vk_2\sprack{n-1},\vk_3\sprack{n-1}),\qquad
&\varpi_{4m} &= \Omega(\vk_0\sprack{n-1},{\vk_1\sprack{n-1}}',{\vk_2\sprack{n-1}}',\vk_3\sprack{n-1}),\\
&\dots\\
&\varpi_{2m\ttn-1} = \Omega(\vk_0\sprack 1,\vk_1\sprack 1,\vk_2\sprack 1,\vk_3\sprack 1),\qquad
&\varpi_{2m\ttn} &= \Omega(\vk_0\sprack 1,{\vk_1\sprack 1}',{\vk_2\sprack 1}',\vk_3\sprack 1).
\end{alignat*}

Now we have enumerated the vertices and defined the associated modulations, we consider a fixed coloring \(C\) of the graph \(G\) by \(p\in \llbracket1,\ttn(6m+1)\rrbracket\) colors. We then proceed as in Section~\ref{s:app} to make identical definitions of degree and degeneracy for each vertex. Further, we may readily verify that this construction ensures that Lemmas~\ref{l:independence},~\ref{l:counting degeneracy}, and~\ref{l:reciprocal bound} remain valid.

To bound the interactions, we have the following analog of Lemma~\ref{l:Osci}:
\begin{lemma}\label{l:Osci2}
Let \(n\geq 1\) and \(\mfi,\mfj\geq 0\). Fix an interaction history \(\vec\ttL\) and coloring \(C\) of the associated interaction graph \(G(\vec\ttL)\). Then, for each choice of \(\vec\ttK\in C\) we have the estimate
\eq{OsciBound2}{
\left|\bbE\left[\prod_{\tti=1}^\ttn\Interactions\sprack\tti\bbo_{\{{\vk_0\sprack \tti}' = \vk_3\sprack{\tti+1}\}}\right]\right| \lesssim_{m,\ttn}\frac1{T_\kin^{m\ttn}}\min\bigl\{N^{1-p},N^{-3m\ttn},N^{-\frac{10}3m\ttn+\frac13n_0}\bigr\},
}
where \(m = \mfi+\mfj+1\).
\end{lemma}
\bpf
By inspecting the associated expressions, we may readily verify that $\LHS{OsciBound2}$ is identical to $\LHS{OsciBound}$ with \(q=\ttn=1\) and \(m\) replaced by \(m\ttn\). Consequently, the estimate \eqref{OsciBound2} is an immediate corollary of \eqref{OsciBound}.
\epf

For the oscillations, we first prove an analog of Lemma~\ref{l:sprocket}:

\begin{lemma}\label{l:sprocket2} Let \(\ttn\geq 1\), \(\mfi,\mfj\geq 0\), and \(1\leq T\leq N\). Fix an interaction history \(\vec\ttL\) and a coloring \(C\) of the interaction graph \(G(\vec\ttL)\). Then, for each \(\tti=1,\dots,\ttn\) we have
\eq{sprocket2}{
\left\|\sum_{V_{2m(\ttn - \tti)}<K_j\leq V_{2m(\ttn+1 - \tti)}}\left|\Oscillations\sprack\tti\right|\right\|_{L^2_{{\tau_0\sprack\tti}',\tau_3\sprack\tti}}\prec_{\mfi,\mfj,n} T^{D\sprack \tti}N^{n_1\sprack \tti+2n_2\sprack \tti+3n_3\sprack \tti},
}
where \(D\sprack \tti\) is the number of degenerate vertices and \(n_d\sprack \tti\) is the number of degree \(d\) vertices in the subgraph \(G\sprack\tti\) of \(G\).
\end{lemma}
\begin{proof}
We consider the case \(\tti = \ttn\). The proof for other values of \(\tti\) is identical.

Let us first consider the case that \(\mfi,\mfj\geq 1\). Taking \(\frac12<b<\frac7{12}\) we denote
\begin{align*}
h_1(\tau_2) & = \left(\tfrac1T + |\tau_2|\right)^{b-\frac12}\int_\R \frac{|\widehat \phi_T(\tau_2 - \alpha')|}{\frac1T + |\alpha'|}\prod_{s=1}^\mfj\left[\sum_{V_{s-1}<K_j\leq V_s} \frac1{\frac1T + |\alpha' + \varpi_s|}\right]\,d\alpha',\\
h_2(\tau_1) &= \left(\tfrac1T + |\tau_1|\right)^{b-\frac12}\int_\R \frac{|\widehat \phi_T(\tau_1 - \alpha)|}{\frac1T + |\alpha|}\prod_{s=\mfj+1}^{\mfi+\mfj}\left[\sum_{V_{s-1}<K_j\leq V_s} \frac1{\frac1T + |\alpha + \varpi_s|}\right]\,d\alpha,\\
h_3(\tau_1') &=\left(\tfrac1T + |\tau_1'|\right)^{b-\frac12}\int_\R \frac{|\widehat \phi_T(\tau_1' - \alpha'')|}{\frac1T + |\alpha''|}\prod_{s=\mfi+\mf j+1}^{\mfi+2\mfj}\left[\sum_{V_{s-1}<K_j\leq V_s} \frac1{\frac1T + |\alpha'' + \varpi_s|}\right]\,d\alpha'',\\
h_4(\tau_2') &= \left(\tfrac1T + |\tau_2'|\right)^{b-\frac12}\int_\R \frac{|\widehat \phi_T(\tau_2' - \alpha''')|}{\frac1T + |\alpha'''|}\prod_{s=\mfi+2\mfj+1}^{m-2}\left[\sum_{V_{s-1}<K_j\leq V_s} \frac1{\frac1T + |\alpha''' + \varpi_s|}\right]\,d\alpha''',\\
g_1(\sigma) &= \left(\tfrac1T + |\sigma|\right)^{b-\frac12}\sum_{V_{2m-2}<K_j\leq V_{2m-1}}|\widehat \chi_T(\sigma - \varpi_{2m-1})|,\\
g_2(\sigma) &=  \left(\tfrac1T + |\sigma|\right)^{b-\frac12}\sum_{V_{2m-1}<K_j\leq V_{2m}}|\widehat \chi_T(\sigma - \varpi_{2m})|.
\end{align*}
We may then bound
\begin{align}
&\bigl(\tfrac1T+|\tau_3|\bigr)^b\bigl(\tfrac1T+|\tau_0'|\bigr)^b\sum_{V_0<K_j\leq V_{2m}}\left|\Oscillations\sprack \ttn(\tau_3,\tau_0') \right|\label{1701}\\
&\qquad\lesssim\int_{\R^5} T^{6b-3}\bigl(\tfrac1T+|\tau_0|\bigr)^{2b-2}g_1(\tau_0 - \tau_1 + \tau_2 - \tau_3)g_2(\tau_0' -\tau_1' + \tau_2' - \tau_0)\notag\\
&\qquad\qquad\qquad\qquad\times h_1(\tau_2)h_2(\tau_1)h_3(\tau_1')h_4(\tau_2')\,d\tau_1'\,d\tau_2'\,d\tau_1\,d\tau_2\,d\tau_0.\notag
\end{align}

Applying Young's convolution inequality with \(\tfrac1q = 2-2b\) and \(\tfrac1p = 5b-2\) we have
\begin{align*}
\left\|\int_{\R^2}g_1(\cdot - \tau_1 + \tau_2 - \tau_3) h_1(\tau_2)h_2(\tau_1+\tau_2)\,d\tau_1\,d\tau_2\right\|_{L^{1/b}}\lesssim \|g_1\|_{L^p}\|h_1\|_{L^q}\|h_2\|_{L^q}.
\end{align*}
An application of the Hardy--Littlewood--Sobolev Lemma then yields the bound
\begin{align*}
\LHS{1701}&\lesssim T^{6b-3} \left\|\bigl(\tfrac1T+|\tau_0|\bigr)^{2b-2}\right\|_{L^{q,\infty}} \|g_1\|_{L^p}\|g_2\|_{L^p}\|h_1\|_{L^q}\|h_2\|_{L^q}\|h_3\|_{L^q}\|h_4\|_{L^q}\\
&\lesssim T^{6b-3} \|g_1\|_{L^p}\|g_2\|_{L^p}\|h_1\|_{L^q}\|h_2\|_{L^q}\|h_3\|_{L^q}\|h_4\|_{L^q},
\end{align*}
where \(\|\cdot\|_{L^{q,\infty}}\) denotes the weak \(L^q\)-quasinorm.

It remains to bound the \(h_j\) and \(g_j\) contributions. Applying Young's convolution inequality and \eqref{recall this later}, for \(\frac1r = \frac32-2b\) we have
\begin{align*}
\|h_1\|_{L^q} &\lesssim \left\|\left(\tfrac1T + |\tau|\right)^{-\frac12}\right\|_{L^r_\tau}\left\|\left(1 + T|\sigma|\right)^b\widehat\phi_T(\sigma)\right\|_{L^1_\sigma}\left\|\left(\tfrac1T + |\alpha|\right)^{b-1}\prod_{s=1}^\mfj\left[\sum_{V_{s-1}<K_j\leq V_s} \frac1{\frac1T + |\alpha + \varpi_s|}\right]\right\|_{L^2_\alpha}\\
&\prec_{b,\mfj} T^{2b-1} T^{\widetilde D}N^{\widetilde n_1 + 2\widetilde n_2 + 3\widetilde n_3},
\end{align*}
where \(\widetilde D\), \(\widetilde n_d\) are the numbers of degenerate, respectively degree \(d\), vertices amongst the \(V_1,\dots,V_{\mfj}\). The corresponding estimates for \(h_2,h_3,h_4\) in \(L^q\) are similar.

Turning to the contribution of \(g_1\), we first note that
\[
|\widehat\chi_T(\sigma - \varpi_{2m-1})|\lesssim \frac 1{\frac1T + |\sigma - \varpi_{2m-1}|}.
\]
Using \eqref{Rigid} to bound \(|\varpi_{m-1}|\prec 1\), we then estimate the integral over the region where \(|\sigma|>N^\nu\) by
\begin{align*}
\left\|\left(\tfrac1T + |\sigma|\right)^{b-\frac12}\sum_{V_{2m-2}<K_j\leq V_{2m-1}}\frac 1{\frac1T + |\sigma - \varpi_{2m-1}|}\right\|_{L^p(|\sigma|>N^\nu)} &\prec N^d\left\|\left(\tfrac1T + |\sigma|\right)^{b-\frac32}\right\|_{L^p(|\sigma|>N^\nu)}\\
&\prec N^d N^{(6b-\frac72)\nu},
\end{align*}
where \(d\) denotes the degree of the vertex \(V_{2m-1}\). For the remaining region, we apply Lemma~\ref{l:reciprocal bound} to bound
\[
\left\|\left(\tfrac1T + |\sigma|\right)^{b-\frac12}\sum_{V_{2m-2}<K_j\leq V_{2m-1}}\frac 1{\frac1T + |\sigma - \varpi_{2m-1}|}\right\|_{L^p(|\sigma|\leq N^\nu)}\prec T^{\widetilde D}N^d N^{(6b -\frac52)\nu},
\]
where we now take \(\widetilde D = 1\) if \(V_{2m-1}\) is degenerate and \(\widetilde D=  0\) otherwise. As \(\nu>0\) is arbitrary, this yields the bound
\[
\|g_1\|_{L^p}\prec T^{\widetilde D}N^d.
\]
The estimate for \(g_2\) is in \(L^p\) similar.

Combining our bounds for the \(h_j\) and \(g_j\), we arrive at the estimate
\[
\left|\LHS{1701}\right|\prec T^{14b-7  +D}N^{n_1+2n_2+3n_3}.
\]
Using that \(T\leq N\) we may choose \(0<b-\frac12\ll1\) arbitrarily small to obtain \eqref{sprocket}.

It remains to consider the case where \(\mfi=0\) or \(\mfj = 0\). Let us consider the case that \(\mfi = 0\) as the case that \(\mfj = 0\) is similar. Here, the functions \(h_2,h_4\) become
\begin{align*}
h_2(\tau_1) &= \left(\tfrac1T + |\tau_1|\right)^{b-\frac12}\int_\R \frac{|\widehat \phi_T(\tau_1 - \alpha)|}{\frac1T + |\alpha|}\,d\alpha,\\
h_4(\tau_2') &= \left(\tfrac1T + |\tau_2'|\right)^{b-\frac12}\int_\R \frac{|\widehat \phi_T(\tau_2' - \alpha''')|}{\frac1T + |\alpha'''|}\,d\alpha''',
\end{align*}
and with the same choice of \(p\) as before, we may apply Young's convolution inequality to bound
\[
\|h_2\|_{L^q}\lesssim \left\|\left(\tfrac1T + |\alpha|\right)^{b-\frac32}\right\|_{L^q}\lesssim T^{b-\frac12},
\]
which is again acceptable.
\end{proof}

Arguing as in Corollary~\ref{c:TB}, we obtain the following corollary of Lemma~\ref{l:sprocket2}:
\begin{corollary}\label{c:TB2}
Let \(\ttn\geq 1\), \(\mfi,\mfj\geq 0\), and \(1\leq T\leq N\). Fix an interaction history \(\vec\ttL\) and a coloring \(C\) of the interaction graph \(G(\vec\ttL)\). Then we have the estimate
\eq{TB2}{
\sum_{\vec\ttK\in C}\left|\int_{\R^{2n}}\prod_{\tti=1}^\ttn\Oscillations\sprack\tti(\tau_3\sprack\tti,{\tau_0\sprack \tti}')\delta\bigl({\tau_0\sprack \tti}'-\tau_3\sprack{\tti+1}\bigr) \,d\vec\tau_0'\,d\vec\tau_3\right| \prec_{\mfi,\mfj,\ttn} T^DN^p.
}
\end{corollary}

We are finally in a position to complete the

\bpf[Proof of Lemma~\ref{l:target}]

Combining \eqref{OsciBound2} and \eqref{TB2}, with the estimate
\[
\sup_{\vec\ttK}\left|\prod_{\tti=1}^\ttn\Data\sprack\tti\right|\lesssim_{A,\mfi,\mfj} 1,
\]
we may argue as in the proof of Proposition~\ref{p:app} to obtain the estimate
\[
\bbE\tr\Bigl[\mf M_{\mfi \mfj}^n\Bigr]\lesssim_{\mfi,\mfj,n,\delta} \frac{T^D}{T_{\kin}^{mn}}\min\{N^{1-p},N^{-3mn},N^{-\frac{10}3m\ttn+\frac13n_0}\}N^{\delta+p},
\]
for any \(\delta>0\), where \(m = \mfi + \mfj +1\).

Employing \eqref{Dee} and \eqref{Aye} as in the proof of Proposition~\ref{p:app}, we may then bound
\[
\bbE\|\fL_{\mfi\mfj}\|_{L^2_t\ell^2\to L^2_t\ell^2}^\ttn\leq\bbE\tr\Bigl[\bigl(\fL_{\mfi\mfj}^*\fL_{\mfi\mfj}\bigr)^n\Bigr]=\bbE\tr\Bigl[\mf M_{\mfi \mfj}^n\Bigr] \lesssim_{\mfi,\mfj,n,\delta} N^{1+\delta}\left(\frac{T^{4/3}}{T_{\kin}}\right)^{mn}.
\]
As \(n\geq 1\) is arbitrary, the estimate \eqref{targetestimate} follows from an application of Markov's inequality.
\epf

%%      ---------------------------------------------------------------------
%%      --------------------------- BIBLIOGRAPHY ----------------------------
%%      ---------------------------------------------------------------------
%% PUT HERE THE BIBLIOGRAPHY IN YOUR FAVOURITE FORMAT
%% Please check that the format of the bibliography is uniform and coherent


\begin{thebibliography}{10}

\bibitem{ACG}
I.~Ampatzoglou, C.~Collot, and P.~Germain.
\newblock Derivation of the kinetic wave equation for quadratic dispersive
  problems in the inhomogeneous setting.
\newblock \textit{Preprint}, 2021.
\newblock \href{https://arxiv.org/abs/2107.11819}{\texttt{arXiv:2107.11819}}.

\bibitem{benaychgeorges2018lectures}
F.~Benaych-Georges and A.~Knowles.
\newblock Lectures on the local semicircle law for {W}igner matrices.
\newblock \textit{Preprint}, 2018.
\newblock \href{https://arxiv.org/abs/1601.04055}{\texttt{arXiv:1601.04055}}.

\bibitem{MR730191}
O.~Bohigas, M.-J. Giannoni, and C.~Schmit.
\newblock Characterization of chaotic quantum spectra and universality of level
  fluctuation laws.
\newblock {\em Phys. Rev. Lett.}, 52(1):1--4, 1984.

\bibitem{BordenaveCollins2020}
C.~Bordenave and B.~Collins.
\newblock Strong asymptotic freeness for independent uniform variables on
  compact groups associated to non-trivial representations.
\newblock \textit{Preprint}, 2020.
\newblock \href{https://arxiv.org/abs/2012.08759}{\texttt{arXiv:2012.08759}}.

\bibitem{BGHS}
T.~Buckmaster, P.~Germain, Z.~Hani, and J.~Shatah.
\newblock Onset of the wave turbulence description of the longtime behavior of
  the nonlinear {S}chr\"{o}dinger equation.
\newblock {\em Invent. Math.}, 225(3):787--855, 2021.

\bibitem{MR2333210}
M.~Christ.
\newblock Power series solution of a nonlinear {S}chr\"{o}dinger equation.
\newblock In {\em Mathematical aspects of nonlinear dispersive equations},
  volume 163 of {\em Ann. of Math. Stud.}, pages 131--155. Princeton Univ.
  Press, Princeton, NJ, 2007.

\bibitem{MR1959915}
B.~Collins.
\newblock Moments and cumulants of polynomial random variables on unitary
  groups, the {I}tzykson-{Z}uber integral, and free probability.
\newblock {\em Int. Math. Res. Not.}, (17):953--982, 2003.

\bibitem{CollinsMatsumoto2017}
B.~Collins and S.~Matsumoto.
\newblock Weingarten calculus via orthogonality relations: new applications.
\newblock {\em ALEA Lat. Am. J. Probab. Math. Stat.}, 14(1):631--656, 2017.

\bibitem{Collins2021}
B.~Collins, S.~Matsumoto, and J.~Novak.
\newblock The {W}eingarten {C}alculus.
\newblock \textit{Preprint}, 2021.
\newblock \href{https://arxiv.org/abs/2109.14890}{\texttt{arXiv:2109.14890}}.

\bibitem{MR2217291}
B.~Collins and P.~\'{S}niady.
\newblock Integration with respect to the {H}aar measure on unitary, orthogonal
  and symplectic group.
\newblock {\em Comm. Math. Phys.}, 264(3):773--795, 2006.

\bibitem{CG2}
C.~Collot and P.~Germain.
\newblock Derivation of the homogeneous kinetic wave equation: longer time
  scales.
\newblock \textit{Preprint}, 2020.
\newblock \href{https://arxiv.org/abs/2007.03508}{\texttt{arXiv:2007.03508}}.

\bibitem{CG1}
C.~Collot and P.~Germain.
\newblock On the derivation of the homogeneous kinetic wave equation.
\newblock \textit{Preprint}, 2021.
\newblock \href{https://arxiv.org/abs/1912.10368}{\texttt{arXiv:1912.10368}}.

\bibitem{DH1}
Y.~Deng and Z.~Hani.
\newblock On the derivation of the wave kinetic equation for {NLS}.
\newblock {\em Forum Math. Pi}, 9:Paper No. e6, 37, 2021.

\bibitem{DH3}
Y.~Deng and Z.~Hani.
\newblock Propagation of chaos and the higher order statistics in the wave
  kinetic theory.
\newblock \textit{Preprint}, 2021.
\newblock \href{https://arxiv.org/abs/2110.04565}{\texttt{arXiv:2110.04565}}.

\bibitem{DH2}
Y.~Deng and Z.~Hani.
\newblock Full derivation of the wave kinetic equation.
\newblock {\em Invent. Math.}, 233(2):543--724, 2023.

\bibitem{MR4227166}
A.~Dymov and S.~Kuksin.
\newblock Formal expansions in stochastic model for wave turbulence 1:
  {K}inetic limit.
\newblock {\em Comm. Math. Phys.}, 382(2):951--1014, 2021.

\bibitem{dymov2021formal}
A.~Dymov and S.~Kuksin.
\newblock Formal expansions in stochastic model for wave turbulence 2: {M}ethod
  of diagram decomposition.
\newblock {\em J. Stat. Phys.}, 190(1):Paper No. 3, 42, 2023.

\bibitem{dymov2021largeperiod}
A.~Dymov, S.~Kuksin, A.~Maiocchi, and S.~Vladuts.
\newblock The large-period limit for equations of discrete turbulence.
\newblock \textit{Preprint}, 2021.
\newblock \href{https://arxiv.org/abs/2104.11967}{\texttt{arXiv:2104.11967}}.

\bibitem{MR3068390}
L.~Erd\H{o}s, A.~Knowles, H.-T. Yau, and J.~Yin.
\newblock The local semicircle law for a general class of random matrices.
\newblock {\em Electron. J. Probab.}, 18:no. 59, 58, 2013.

\bibitem{hasselmann1}
K.~Hasselmann.
\newblock On the non-linear energy transfer in a gravity-wave spectrum part 1.
  general theory.
\newblock {\em Journal of Fluid Mechanics}, 12(4):481--500, 1962.

\bibitem{hasselmann2}
K.~Hasselmann.
\newblock On the non-linear energy transfer in a gravity wave spectrum part 2.
  conservation theorems; wave-particle analogy; irrevesibility.
\newblock {\em Journal of Fluid Mechanics}, 15(2):273--281, 1963.

\bibitem{kolmogorov2}
A.~N. Kolmogorov.
\newblock Dissipation of energy in the locally isotropic turbulence.
\newblock In {\em Dokl. Akad. Nauk SSSR A}, volume~32, pages 16--18, 1941.

\bibitem{kolmogorov1}
A.~N. Kolmogorov.
\newblock The local structure of turbulence in incompressible viscous fluid for
  very large reynolds numbers.
\newblock {\em Cr Acad. Sci. URSS}, 30:301--305, 1941.

\bibitem{MR2755061}
J.~Lukkarinen and H.~Spohn.
\newblock Weakly nonlinear {S}chr\"{o}dinger equation with random initial data.
\newblock {\em Invent. Math.}, 183(1):79--188, 2011.

\bibitem{NST}
S.~Nazarenko, A.~Soffer, and M.-B. Tran.
\newblock On the wave turbulence theory for the nonlinear {S}chr\"{o}dinger
  equation with random potentials.
\newblock {\em Entropy}, 21(9):Paper No. 823, 12, 2019.

\bibitem{Peierls}
R.~Peierls.
\newblock Zur kinetischen theorie der w{\"a}rmeleitung in kristallen.
\newblock {\em Annalen der Physik}, 395(8):1055--1101, 1929.

\bibitem{RST}
B.~Rumpf, A.~Soffer, and M.-B. Tran.
\newblock On the wave turbulence theory: ergodicity for the elastic beam wave
  equation.
\newblock \textit{Preprint}, 2021.
\newblock \href{https://arxiv.org/abs/2108.13223}{\texttt{arXiv:2108.13223}}.

\bibitem{SF}
G.~Schwiete and A.~Finkel'Stein.
\newblock Effective theory for the propagation of a wave packet in a disordered
  and nonlinear medium.
\newblock {\em Phys. Rev. A}, 87(4):043636, 2013.

\bibitem{SWDC}
T.~Scoquart, T.~Wellens, D.~Delande, and N.~Cherroret.
\newblock Quench dynamics of a weakly interacting disordered {B}ose gas in
  momentum space.
\newblock {\em Phys. Rev. Research}, 2(3):033349, 2020.

\bibitem{ST}
G.~Staffilani and M.-B. Tran.
\newblock On the wave turbulence theory for stochastic and random
  multidimensional {K}d{V} type equations.
\newblock \textit{Preprint}, 2021.
\newblock \href{https://arxiv.org/abs/2106.09819}{\texttt{arXiv:2106.09819}}.

\bibitem{MR2233925}
T.~Tao.
\newblock {\em Nonlinear dispersive equations}, volume 106 of {\em CBMS
  Regional Conference Series in Mathematics}.
\newblock Published for the Conference Board of the Mathematical Sciences,
  Washington, DC; by the American Mathematical Society, Providence, RI, 2006.
\newblock Local and global analysis.

\bibitem{TRM}
M.~Tavora, A.~Rosch, and A.~Mitra.
\newblock Quench dynamics of one-dimensional interacting bosons in a disordered
  potential: Elastic dephasing and critical speeding-up of thermalization.
\newblock {\em Phys. Rev. Lett.}, 113(1):010601, 2014.

\bibitem{ZLF}
V.~E. Zakharov, V.~S. L'vov, and G.~Falkovich.
\newblock {\em Kolmogorov spectra of turbulence I: Wave turbulence}.
\newblock Springer Science \& Business Media, 2012.

\end{thebibliography}
\end{document}